\RequirePackage{etex}
\documentclass[12pt,3p,final]{elsarticle}
\usepackage{amsmath}
\makeatletter
\def\ps@pprintTitle{%
 \let\@oddhead\@empty
 \let\@evenhead\@empty
 \def\@oddfoot{}%
 \let\@evenfoot\@oddfoot}
\makeatother

\usepackage[T1]{fontenc}
\usepackage[english]{babel}

\usepackage{amsthm, amssymb}
\usepackage{mathtools}      
\usepackage{upgreek}        
\usepackage{mathrsfs}       

\usepackage{tikz-cd}        
\usepackage{enumitem}       
\usepackage[all]{xy}

\usepackage[colorlinks=true, linkcolor=red, citecolor=blue, urlcolor=blue, pagebackref=false, hypertexnames=false]{hyperref}
\usepackage[nameinlink,noabbrev]{cleveref} 
\usepackage{autonum}        

\usepackage{xcolor} 
\usepackage{tcolorbox}

\makeatletter
\renewcommand*{\eqref}[1]{\hyperref[{#1}]{\textup{\tagform@{\ref*{#1}}}}}
\makeatother

\let\originalleft\left
\let\originalright\right
\renewcommand{\left}{\mathopen{}\mathclose\bgroup\originalleft}
\renewcommand{\right}{\aftergroup\egroup\originalright}

\theoremstyle{plain}
\newtheorem{theorem}{Theorem}[section]
\newtheorem{lemma}[theorem]{Lemma}
\newtheorem{proposition}[theorem]{Proposition}

\newtheorem{convention}{Convention}[section]

\theoremstyle{definition}
\newtheorem{definition}[theorem]{Definition}
\newtheorem{example}[theorem]{Example}
\newtheorem{remark}[theorem]{Remark}

\numberwithin{equation}{section}

\crefname{theorem}{Theorem}{Theorems}
\crefname{maintheorem}{Theorem}{Theorems}
\crefname{lemma}{Lemma}{Lemmata}
\crefname{proposition}{Proposition}{Propositions}
\crefname{corollary}{Corollary}{Corollaries}
\crefname{definition}{Definition}{Definitions}
\crefname{remark}{Remark}{Remarks}
\crefname{example}{Example}{Examples}
\crefname{section}{Section}{Sections}
\crefname{appendix}{Appendix}{Appendices}

\begin{document}

\renewcommand{\footnoterule}{} 

\begin{frontmatter}
\title{A Geometric Realization of Spherical T-Duality via $\star$-Diagrams}


\author[IMI]{Leonardo F. Cavenaghi}
\ead{leonardofcavenaghi@gmail.com}

\author[Unicamp]{Lino Grama}
\ead{lgrama@unicamp.br}

\author[Miami,IMSA,ICMS]{Ludmil Katzarkov}
\ead{lkatzarkov@gmail.com}


\address[Miami]{Department of Mathematics, University of Miami, Coral Gables, FL 33146, USA}
\address[Unicamp]{IMECC, Universidade Estadual de Campinas (Unicamp), Campinas, SP, 13083-856, Brazil}
\address[IMSA]{Institute of the Mathematical Sciences of the Americas (IMSA), Coral Gables, FL 33146, USA}
\address[ICMS]{International Center for Mathematical Sciences (ICMS), Sofia, Bulgaria}
\address[IMI]{Institute of Mathematics and Informatics (IMI), Sofia, Bulgaria}

\begin{abstract}
We relate spherical T-duality for oriented linear $\mathrm{S}^3$-bundles over $\mathrm{S}^4$ (the Milnor bundles $M_{m,n}$, which are $\mathrm{S}^3$-principal exactly when $m=0$ or $n=0$, and whose total spaces are homotopy $7$-spheres exactly when $m+n=\pm1$) to $\star$-diagrams and to a higher-dimensional generalization of the logarithmic transformations of $4$-manifold topology.

For a pair of $\mathrm{S}^3$-principal $(P,H)$ and $(\widehat P,\widehat H)$ over $\mathrm{S}^4$, we show that the T-duality correspondence space $P\times_{\mathrm{S}^4}\widehat P$ is itself a $\star$-diagram of a distinguished type, which we call \emph{bifree}, and that bifree $\star$-diagrams are precisely the fiber products of principal bundles; spherical T-duality of the decorated pair is then a condition on the fluxes carried by that diagram. For bundles of equal Euler class $ku$, principal or not, we show that the two bundles are spherical T-dual with the diagonal fluxes $[k]$, and that they occur as the two base manifolds of an explicit $\star$-diagram, obtained by pulling back a principal Milnor bundle; this diagram is never bifree. 

We then introduce product-preserving generalized logarithmic transformations on products $\Sigma\times\mathrm{S}^1$ of homotopy spheres with the circle, and prove that, after stabilization by $\mathrm{S}^1$, the spherical T-dualities between homotopy $7$-spheres are realized by such transformations. In particular, $\Sigma^7_{GM}\times\mathrm{S}^1$ is obtained from $\mathrm{S}^7\times\mathrm{S}^1$ by one of them, where $\Sigma^7_{GM}$ denotes the Gromoll--Meyer exotic sphere: spherical T-duality relates distinct smooth structures on the topological $7$-sphere, and the relation is implemented by an explicit cut-and-paste operation. Along the way we prove that the two base manifolds of a $\star$-diagram have Morita equivalent action groupoids and for connected structure group, Hausdorff–Morita equivalent orbit foliations, and, in
the almost-free case, isomorphic real equivariant cohomology.

\end{abstract}


\end{frontmatter}

\section{Introduction}

Topological T-duality, in the form established by Bouwknegt, Evslin, Hannabuss, and Mathai \cite{Bouwknegt20041,Bouwknegt20042,Bouwknegt20043,Bouwknegt2005}, assigns to each pair $(P,H)$, consisting of a principal circle bundle and an integral $3$-cocycle, a T-dual pair $(\widehat P,\widehat H)$ over the same base, unique up to isomorphism, and induces degree-shifting isomorphisms between the twisted cohomology and twisted K-theory of the two sides; the correspondence extends to generalized geometry \cite{gualtieri,cavalcanti}. Its spherical analog \cite{Bouwknegt2015,Bouwknegt20152} replaces circle bundles by $\mathrm{S}^3\cong\mathrm{SU}(2)$-bundles and $3$-cocycles by degree-seven fluxes, the duality being encoded by the Gysin relations $\mathrm{c}_2(\widehat P)=\pi_*H$, $\widehat\pi_*\widehat H=\mathrm{c}_2(P)$ together with the matching of the fluxes on the correspondence space; see \cite{Lind2020,cavalcanti2025topologicalsphericaltduality} for recent developments. In contrast with the circle case, spherical T-duals are generally not unique.

Over the base $\mathrm{S}^4$, the objects paired by spherical T-duality are classical: the oriented linear $\mathrm{S}^3$-bundles over $\mathrm{S}^4$ are the Milnor bundles $M_{m,n}$, classified by $\pi_3(\mathrm{SO}(4))\cong\mathbb{Z}\oplus\mathbb{Z}$. Their total spaces are homeomorphic to $\mathrm{S}^7$ precisely when $m+n=\pm1$, and among these Milnor exhibited the first examples of manifolds homeomorphic but not diffeomorphic to the standard sphere \cite{mi}. This paper is guided by two questions: which manifolds (and, among homotopy spheres, which smooth structures) do spherical T-duality pair; and to what extent can this essentially cohomological pairing be realized by an explicit geometric operation.

Our answers are organized around $\star$-diagrams, the two-quotient diagrams $M'\leftarrow P\rightarrow M$ of a manifold with two commuting free actions introduced by Speran\c{c}a \cite{speranca2016pulling} and developed in \cite{SperancaCavenaghiPublished}, whose original purpose was the construction of exotic spheres as quotients (the Gromoll--Meyer sphere $\Sigma^7_{GM}$ \cite{gromoll1974exotic} being the fundamental example) and around a surgery operation we introduce, generalizing the logarithmic transformations of elliptic fibrations \cite{GOMPF1991479,gompfmrowka,ZENTNER200637} to products $\Sigma\times\mathrm{S}^1$ of homotopy spheres with the circle.

\vspace{1em}

On the duality side, we show that $(M_{m,0},[j])$ and $(M_{0,-j},[m])$ are dual, and the degree-seven cohomology of the correspondence space is $\mathbb{Z}\oplus\mathbb{Z}_{\gcd(j,m)}$ (Theorem \ref{thm:tdualmilnor}). We also show that any two Milnor bundles of the same Euler class $ku$ are dual with diagonal fluxes $[k]$ (Theorem \ref{thm:nonprincipaltdualmilnor}). In particular, the homotopy $7$-spheres $M_{m,1-m}$, all of Euler class $u$, are pairwise spherical T-dual with flux $[1]$: spherical T-duality relates distinct smooth structures on the topological $7$-sphere.

\vspace{1em}

On the geometric side, we show that the principal T-duality correspondence $P\times_{\mathrm{S}^4}\widehat P$ is itself a $\star$-diagram of a distinguished type, which we call bifree, and that bifree $\star$-diagrams are exactly fiber products of principal bundles (Theorem \ref{thm:equivalence}, Lemma \ref{lem:bifree}). The pullback of a principal Milnor bundle along another Milnor bundle carries two natural $\star$-structures (Proposition \ref{prop:pullback-star}): a conjugation-type structure, never bifree, which preserves the Euler class, and, over principal bases, a translation-type structure, bifree, which shifts it. The first realizes the equal-Euler-class dualities geometrically: the $\star$-quotient produces the dual bundle together with the unique diagonal fluxes (Theorem \ref{thm:non-principal-equivalence}).

\vspace{1em}

Finally, we define product-preserving generalized logarithmic transformations on $\Sigma_r\times\mathrm{S}^1$, for homotopy spheres presented as twisted doubles (Definition \ref{def:generalized-log}, Proposition \ref{prop:admissible-star}). Since the diffeomorphism type of $\Sigma\times\mathrm{S}^1$ determines that of $\Sigma$ (Theorem \ref{thm:brieskorn-van-de-ven}; \cite{BRIESKORN1968389}), these transformations genuinely change smooth structures while preserving the integral homology. Our closing result, Theorem \ref{thm:Sperical-T-Dual-vs-Log}, joins the threads: for any $m,j$, the T-dual pair $M_{m,1-m}\times\mathrm{S}^1$ and $M_{j,1-j}\times\mathrm{S}^1$ are related by a product-preserving generalized logarithmic transformation, realized as the $\star$-quotient of an explicit pullback bundle; the case $(m,j)=(1,2)$ recovers $\Sigma^7_{GM}\times\mathrm{S}^1$ from $\mathrm{S}^7\times\mathrm{S}^1$. Within a fixed Euler class, spherical T-duality, pullback $\star$-diagrams, and generalized logarithmic transformations are thus three descriptions of one operation. 

\vspace{1em} 

Along the way we establish several facts about $\star$-diagrams that are of independent interest. The two base manifolds of a $\star$-diagram have Morita equivalent action groupoids (Theorem \ref{thm:morita}) and, for connected structure group, we have Hausdorff--Morita equivalent orbit foliations (Theorem \ref{thm:starsaremorita}); their orbit spaces are canonically isometric, their orbit types agree, and, when the actions are almost free, their real equivariant cohomologies are isomorphic (Theorem \ref{thm:equivanlentcohomologies}). Any obstruction to $M\cong M'$ must therefore be sought beyond the equivariant-topological invariants, and in the case of the Gromoll--Meyer sphere it is smooth in nature.

\vspace{1em}

\section{A Concise Account and New Results on \texorpdfstring{$\star$}{star}-diagrams}
\label{sec:stardiagrams}

The concept of an \emph{exotic sphere} originates from John Milnor's pioneering work in the 1950s \cite{mi}. Milnor introduced a family of 7-dimensional manifolds that are homeomorphic but not diffeomorphic to the standard 7-dimensional sphere, then termed exotic spheres. More broadly, the term \emph{exotic manifold} refers to a smooth manifold $M'$ that is homeomorphic but not diffeomorphic to another smooth manifold $M$. In \cite{speranca2016pulling}, building on the work of \cite{duran2001pointed} and \cite{rigas1996hopf}, Sperança introduced a general method for constructing exotic manifolds $M'$ from a \emph{classical} or \emph{standard} realization $M$, establishing a correspondence between their \emph{invariant} geometries.\footnote{Notably, in the aforementioned works, this concept generalizes the construction of an exotic sphere by Gromoll and Meyer in \cite{gromoll1974exotic}; see, e.g., Example \ref{ex:gm-revisited}.} 

Let $P$ be a smooth manifold which is the total space of a $G$-principal bundle whose principal $G$-action is denoted by $\bullet$, where $G$ is a \emph{compact} Lie group. Suppose that there exists another $G$-action on $P$, denoted by $\star$, which commutes with $\bullet$ and is free, making $P$, equipped with this $G$-action, the total space of another $G$-principal bundle. Let $M$ and $M'$ be the orbit spaces for the $\bullet$ and $\star$ actions, respectively. We have the following $G$-principal bundles' diagram, which we term a \emph{$\star$-diagram}:

\begin{equation}\label{eq:CD}
\begin{xy}\xymatrix{& G\ar@{..}[d]^{\bullet} & \\ G\ar@{..}[r]^{\star} & P\ar[d]^{\pi}\ar[r]^{\pi'} & M'\\ & M &}\end{xy}
\end{equation}

\ 

Throughout this manuscript, we occasionally employ the shorthand notation $M \stackrel{\pi}{\leftarrow} P \stackrel{\pi'}{\to} M'$ for a $\star$-diagram, possibly omitting the projections $\pi$ and $\pi'$. We may also refer to it as a \emph{$\star$-bundle}. We say that $G$ (sometimes omitted) is the \emph{structure group} of the given $\star$-diagram (or $\star$-bundle) and $P$ is a $G$-$G$-manifold. 

\ 

Let us present some examples. Example \ref{ex:gromollmeyer} below can be considered as the first one exemplifying a $\star$-diagram construction. More details can be found in \cite{gromoll1974exotic, duran2001pointed, speranca2016pulling}.

\begin{example}[The Gromoll--Meyer Exotic Sphere]\label{ex:gromollmeyer}
Consider the compact Lie group:
\begin{equation}
\mathrm{Sp}(2) = \left\{ \begin{pmatrix} a & c \\ b & d \end{pmatrix} \in \mathrm{S}^7 \times \mathrm{S}^7 \;\middle|\; \bar{a}c + \bar{b}d = 0 \right\}, \label{eq:Sp2}
\end{equation}
where $a, b, c, d \in \mathbb{H}$ are quaternions with their usual conjugation, multiplication, and norm. Note that the conditions precisely mandate that the columns be orthonormal with respect to the standard quaternionic inner product. Set $q\in \mathrm{S}^3\cong \text{unit quaternions}$. The projection $\pi: \mathrm{Sp}(2) \to \mathrm{S}^7$ of an element to its first column defines an $\mathrm{S}^3$-principal bundle with principal action:

\begin{equation} \label{eq:GMprincipalaction}
\begin{pmatrix} 
a & c \\
b & d 
\end{pmatrix} \bar{q} = \begin{pmatrix}
a & c \overline{q} \\
b & d \overline{q}
\end{pmatrix}.
\end{equation}

In \cite{gromoll1974exotic}, a $\star$-action is also considered:

\begin{equation} \label{eq:GMstaraction}
q \begin{pmatrix} 
a & c \\
b & d 
\end{pmatrix} = \begin{pmatrix} 
q a \overline{q} & q c \\
q b \overline{q} & q d 
\end{pmatrix},
\end{equation}
whose quotient is an exotic 7-sphere known as the \emph{Gromoll--Meyer exotic sphere}, denoted by $\Sigma_{GM}^7$. We incorporate actions \eqref{eq:GMprincipalaction} and \eqref{eq:GMstaraction} into the principal bundles' diagram:

\begin{equation} \label{eq:CDGM}
\begin{xy} \xymatrix{
& \mathrm{S}^3 \ar@{..}[d]^{\bullet} & \\ 
\mathrm{S}^3 \ar@{..}[r]^{\star} & \mathrm{Sp}(2) \ar[d]^{\pi} \ar[r]^{\pi'} & \Sigma^7_{GM} \\ 
& \mathrm{S}^7 & 
} \end{xy}
\end{equation}
\end{example}

\vspace{1em}

A manifold $M'$ fitting one of the base manifolds in a $\star$-diagram does not necessarily come with an exotic smooth structure (Example \ref{ex:localmodelstar}). The choice of the $\star$-action, described by a cocycle condition (Definition \ref{def:star}), determines whether exotic structures emerge.

\begin{example}[Pairs of Diffeomorphic Manifolds via $\star$-Diagrams] \label{ex:localmodelstar}
Let $M$ be a smooth manifold with a smooth action by a compact Lie group $G$, denoted by $\cdot$; no freeness or effectiveness is assumed. Consider the product manifold $M \times G$ with the following $\star$-action:

\[
g \star (x, g') := (g \cdot x, gg'), \quad x \in M, \quad g, g' \in G.
\]
Let $\bullet$ be the following $G$-action on $M \times G$:
\[
g \bullet (x, g') := (x, (g') g^{-1}), \quad x \in M, \quad g, g' \in G.
\]

The $G$-actions defined by $\bullet$ and $\star$ are free commuting actions on $M \times G$. The orbit maps for these actions are, respectively, $\pi: M \times G \rightarrow M$, $(x, g') \mapsto x$, and $\pi': M \times G \rightarrow M$, $(x, g') \mapsto (g')^{-1} x$. We can construct the corresponding $\star$-diagram:

\begin{equation} \label{eq:CDcheeger}
\begin{xy} \xymatrix{
& G \ar@{..}[d]^{\bullet} & \\ 
G \ar@{..}[r]^{\star} & M \times G \ar[d]^{\pi} \ar[r]^{\pi'} & M \\ 
& M & 
} \end{xy}
\end{equation}
\end{example}

\vspace{1em}

Next, we revisit the procedure detailed in \cite{SperancaCavenaghiPublished} which provides a general method for constructing $\star$-diagrams. Theorem \ref{thm:star} shows that the actions in Example \ref{ex:localmodelstar} are always local descriptions for general $\bullet$ and $\star$-actions. 

\vspace{1em} 

\subsection{Manufacturing exotic manifolds: the recipe for \texorpdfstring{$\star$}{star}-diagrams}
Let \( M \) be a \( G \)-manifold, i.e. a smooth manifold with a (here assumed left) action by a compact Lie group $G$ and let \( \{U_i\} \) be a collection of \( G \)-invariant open sets in \( M \). By reducing \( U_i \) if necessary, we can assume that \( GU_i = U_i \). Given two \( G \)-invariant open sets \( U_i \) and \( U_j \), we define \( U_{ij} := U_i \cap U_j \).

\begin{definition}\label{def:star}
A collection \( \phi_{ij} : U_{ij} \to G \) is called a \emph{\(\star\)-collection} if it satisfies the normalization conditions $\phi_{ii}(x) = e$ on $U_i$ and $\phi_{ji}(x) = \phi_{ij}(x)^{-1}$ on $U_{ij}$, the \textit{cocycle condition}
\begin{equation}\label{eq:cocyclocondition}
    \phi_{ij}(x) \phi_{jk}(x) = \phi_{ik}(x), \quad \forall x \in U_i \cap U_j \cap U_k,
\end{equation}
and the \textit{covariance condition}
\begin{equation}\label{eq:covariance}
    \phi_{ij}(gx) = g \phi_{ij}(x) g^{-1}, \quad \forall i,j, \forall g \in G, \forall x \in U_{ij}.
\end{equation}
\end{definition}

The covariance condition \eqref{eq:covariance} ensures that the \textit{adjoint} map
\[
\widehat{\phi}_{ij} : U_{ij} \to U_{ij}
\]
defined by
\begin{equation}\label{eq:autophagic}
\widehat{\phi}_{ij}(x) = \phi_{ij}(x) x
\end{equation}
is $G$-\textit{equivariant}. Moreover, \( \widehat{\phi}_{ij} \) is a diffeomorphism of \( U_{ij} \), with smooth inverse $x\mapsto\phi_{ij}(x)^{-1}x$: the covariance condition, evaluated at $g=\phi_{ij}(x)$ and at $g=\phi_{ij}(x)^{-1}$, gives $\phi_{ij}\big(\widehat{\phi}_{ij}(x)\big)=\phi_{ij}(x)=\phi_{ij}\big(\phi_{ij}(x)^{-1}x\big)$, from which both composites are the identity.

\ 

Let \( \{\phi_{ij} : U_{ij} \to G\} \) be a \( \star \)-collection. We define the space
\[
\bigcup_{\widehat{\phi}_{ij}} U_i
\]
as the quotient space under the equivalence relation \( x \in U_{ij} \sim \widehat{\phi}_{ij}(x) \in U_{ij} \). Theorem \ref{thm:star}, proved in \cite{SperancaCavenaghiPublished}, establishes the recipe for constructing \( \star \)-diagrams. To make proper sense of this, recall that one can build a principal bundle $P\rightarrow M$ out of a $\star$-collection whose underlying invariant open sets $\{U_i\}$ cover $M$, an assumption in force whenever a bundle is constructed from a collection:
\[P=\bigcup_{f_{\phi_{ij}}}U_i\times G,\qquad f_{\phi_{ij}}(x,g)=(x,\,g\,\phi_{ij}(x)),\]
the principal action being given chartwise by left translation, $s\bullet(x,g)=(x,sg)$, which is globally well defined because left and right translations of $G$ commute. The $\star$-action furnished by Theorem \ref{thm:star} is given chartwise by
\begin{equation}\label{eq:star-chartwise-def}
r\star(x,g)=(r\cdot x,\;g\,r^{-1}),
\end{equation}
and the covariance condition \eqref{eq:covariance} is exactly what makes these chartwise formulas compatible with the transition functions: $(g\,r^{-1})\,\phi_{ij}(r\cdot x)=g\,r^{-1}\,r\,\phi_{ij}(x)\,r^{-1}=\big(g\,\phi_{ij}(x)\big)r^{-1}$. In particular $\star$ is free, since $g\,r^{-1}=g$ forces $r=e$, and it commutes with $\bullet$.

\begin{theorem}[Theorem 2.2 in \cite{SperancaCavenaghiPublished}]\label{thm:star}
Let \( \pi : P \to M \) be the principal bundle associated to a \( \star \)-collection given by \( \{\phi_{ij} : U_{ij} \to G\} \). Then \( P \) admits a new action, denoted by \( \star \), such that:

\begin{enumerate}
    \item The \( \star \)-action on \( P \) is free.
    \item \label{item:quotientopen} The quotient \( P/\star \) is a \( G \)-manifold that is equivariantly diffeomorphic to
    \[
    M' := \bigcup_{\widehat{\phi}_{ij}} U_i.
    \]
  \item \label{item:isotropy} The \( \star \)-diagram obtained from this construction is such that, denoting by \( (G\times G)_p \) the isotropy group at \( p \) of the juxtaposed action
    \[
    (r,s)\cdot p := r\star p\cdot s^{-1},
    \]
    in which \( G \times \{e\} \) acts by the \( \star \)-action and \( \{e\} \times G \) by the principal action, and by \( e \) the identity of \( G \), there exists \( g \in G \) such that
    \begin{equation}
        (G \times G)_p = \{(h, ghg^{-1}) : h \in G_{\pi(p)}\}.
    \end{equation}
\end{enumerate}
\end{theorem}

\vspace{1em} 

A prolific method for constructing new $\star$-diagram examples involves taking pullbacks of the bundle \(P\) along a \(\star\)-diagram using appropriate smooth functions. Let \(M\) and \(N\) be \(G\)-manifolds, and let \(f: N \to M\) be a smooth equivariant map. Consider a \(\star\)-collection \(\{\phi_{ij} : U_{ij} \to G\}\) associated with the bundle \(\pi : P \to M\). The \textit{pullback bundle} \(f^*(P)\) over \(N\) has a total space given by
\[
f^*(P) := \{(n,p) \in N \times P \mid f(n) = \pi(p)\},
\]
and a projection \(\pi_f(n,p) := n\). The principal action induced on \(f^*(P)\) acts on the second coordinate only,
\[
s\bullet(n,p) := (n,\; s\bullet p), \quad \forall s \in G,
\]
which is well defined since \(\pi(s\bullet p) = \pi(p) = f(n)\), and it is free. The \(\star\)-action induced on \(f^*(P)\) is the diagonal one,
\[
r \star (n,p) := (r\cdot n,\; r\star p), \quad \forall r \in G,
\]
which maps into \(f^*(P)\) because the equivariance of \(f\) and of \(\pi\) gives \(f(r\cdot n) = r\cdot f(n) = r\cdot \pi(p) = \pi(r\star p)\); it is free because \(\star\) is free on \(P\), and it commutes with \(\bullet\). 

\begin{proposition}[Proposition 2.4 in \cite{SperancaCavenaghiPublished}]\label{prop:howtopullback}
Let \(f: N \to M\) be a smooth and \(G\)-equivariant map, and let \(M'\leftarrow P \rightarrow M\) be a \(\star\)-bundle obtained from a \(\star\)-collection given by \(\{\phi_{ij}: U_{ij} \to G\}\). Then, the following holds:
\begin{enumerate}
    \item The pullback bundle $\pi_f : f^{*}P \to N$ is equivariantly isomorphic to the $\star$-bundle associated with the pulled-back $\star$-collection given by \(\{\phi_{ij} \circ f: f^{-1}(U_{ij}) \to G\}\).
    \item The quotient \(f^*(P)/\star\) is equivariantly diffeomorphic to 

    \[
    N' = \bigcup_{\widehat{\phi_{ij} \circ f}} f^{-1}(U_i).
    \]
    \item A map \(f': N' \to M'\) is well defined and satisfies \(f'|_{f^{-1}(U_i)} = f|_{f^{-1}(U_i)}\).
\end{enumerate}
\end{proposition}

\vspace{1em} 

\subsection{Riemannian submersions and \texorpdfstring{$\star$}{*}-diagrams}
\label{sec:diff-geo}
As discussed in \cite{SperancaCavenaghiPublished}, the geometries of \(M\) and \(M'\) can be compared through a \(G \times G\)-invariant Riemannian metric \(\mathsf{g}\) on the bundle \(P\), which we construct below starting from a \(G\)-invariant Riemannian metric \(\mathsf{g}_M\) on \(M\). Once \(\mathsf{g}\) is fixed, we let \(\mathcal{H}''\) denote the \emph{bihorizontal distribution}, defined pointwise by
\[
\mathcal{H}''_p = \{X \in T_pP \mid X \perp T_p((G \times G)p)\},
\]
the orthogonal complement being taken with respect to \(\mathsf{g}\). Following Section 5 of \cite{SperancaCavenaghiPublished}, the metric \(\mathsf{g}\) induces a Riemannian metric \(\mathsf{g}_{M'}\) on \(M'\) such that both projections \(\pi\) and \(\pi'\) are Riemannian submersions. Writing \(\mathcal{H} \subseteq TM\) and \(\mathcal{H}' \subseteq TM'\) for the horizontal distributions of vectors orthogonal to the \(G\)-orbits on \(M\) and on \(M'\), respectively, we shall verify that the restrictions
\[
\mathrm{d}\pi|_{\mathcal{H}''}:(\mathcal{H}'',\mathsf{g}) \rightarrow (\mathcal{H},\mathsf{g}_M),
\qquad
\mathrm{d}\pi'|_{\mathcal{H}''}:(\mathcal{H}'',\mathsf{g}) \rightarrow (\mathcal{H}',\mathsf{g}_{M'})
\]
are isometries. We now carry out this construction.

\vspace{1em}

Start with a $G$-invariant Riemannian metric $\mathsf{g}_M$ on $M$. To construct \(\mathsf{g}_{M'}\), we recall that a connection 1-form on a \(G\)-principal bundle is a differential 1-form \(\omega_0: TP \to \mathfrak{g}\) satisfying the following properties: for every \(\xi \in \mathfrak{g}\), \(X \in TP\), and \(g \in G\), we have
\begin{itemize}
    \item \((\omega_0)_p(\xi^{\bullet}) = \xi\), where $\xi^{\bullet}$ denotes the action field of $\xi$ for the principal action (see Equation \eqref{eq:action} below);
    \item \((\omega_{0})_{p\cdot s}\big(X\cdot s\big) = \operatorname{Ad}_{s^{-1}}(\omega_0)_p(X)\), where $p\cdot s$ denotes the principal action written on the right, as in the juxtaposition convention of Theorem \ref{thm:star}.
\end{itemize}
Averaging \(\omega_0\) over the commuting \(\star\)-action,
\[
\omega := \int_G (\mathrm{L}^\star_r)^{*} \omega_0 \, \mathrm{d}\mu(r),
\]
with $\mathrm{d}\mu$ the normalized Haar measure of the compact group $G$ and $\mathrm{L}^\star_r(p)=r\star p$, again yields a principal connection for the $\bullet$-action — the two defining properties are preserved because $\star$ commutes with $\bullet$, hence $(\mathrm{L}^\star_r)_*\xi^{\bullet}=\xi^{\bullet}$ and $\mathrm{L}^\star_r$ intertwines the principal action with itself — and $\omega$ is in addition $\star$-invariant: $(\mathrm{L}^\star_r)^{*}\omega = \omega$ for all $r \in G$, by invariance of the Haar measure. Thus, if \(\mathsf{g}_M\) is a \(G\)-invariant metric on \(M\), and \(Q\) is an \(\mathrm{Ad}\)-invariant inner product on the Lie algebra \(\mathfrak{g}\), we define a \(G \times G\)-invariant Kaluza--Klein metric on \(P\) as
\[
\mathsf{g} := \pi^{\ast}\mathsf{g}_M + Q(\omega, \omega).
\]
The metric $\mathsf{g}$ is indeed $G\times G$-invariant. It is $\star$-invariant because $\pi$ is $\star$-equivariant and $\mathsf{g}_M$ is $G$-invariant, so that $\pi^{*}\mathsf{g}_M$ is $\star$-invariant, while $Q(\omega,\omega)$ is $\star$-invariant by the averaging performed above. It is invariant under the principal action because $\pi^{*}\mathsf{g}_M$ is, the principal orbits being the fibers of $\pi$, and because $Q$ is $\mathrm{Ad}$-invariant, which together with the second connection axiom gives \[ Q\big(\omega_{p\cdot s}(X\cdot s),\,\omega_{p\cdot s}(Y\cdot s)\big) = Q\big(\mathrm{Ad}_{s^{-1}}\omega_p(X),\,\mathrm{Ad}_{s^{-1}}\omega_p(Y)\big) = Q\big(\omega_p(X),\,\omega_p(Y)\big). \]  The quotient metric $\mathsf{g}_{M'}$ is then defined as the unique metric for which $\pi'$ is a Riemannian submersion: for $v, w \in T_{x'}M'$ we set $\mathsf{g}_{M'}(v, w) := \mathsf{g}_p(v^{\mathrm{hor}}_p, w^{\mathrm{hor}}_p)$, where $p \in (\pi')^{-1}(x')$ and $v^{\mathrm{hor}}_p, w^{\mathrm{hor}}_p \in (\ker \mathrm{d}\pi'_p)^{\perp} \subseteq T_pP$ are the unique horizontal lifts. Independence of the representative $p$ follows from the $\star$-invariance of $\mathsf{g}$ just established. By construction, $\pi'$ is a Riemannian submersion; and so is $\pi$, since $\bar{\mathcal H}=\ker\omega=(\ker\mathrm{d}\pi)^{\perp}$ and $\mathrm{d}\pi|_{\bar{\mathcal H}_p}:(\bar{\mathcal H}_p,\mathsf{g})\to (T_{\pi(p)}M,\mathsf{g}_M)$ is an isometry, directly from the definition of $\mathsf{g}=\pi^{*}\mathsf{g}_M+Q(\omega,\omega)$.  Finally, we verify the assertion made above concerning $\mathcal H''$. By definition, $\mathcal{H}''_p$ consists of the vectors orthogonal to the whole juxtaposed orbit, hence in particular to the principal orbit, so $\mathcal{H}''\subseteq \bar{\mathcal H}$; the computation carried out in the proof of Proposition \ref{prop:fibradonormal} below shows that $\mathrm{d}\pi$ maps $\mathcal{H}''_p$ bijectively onto $\mathcal H_{\pi(p)}$, and since $\mathrm{d}\pi|_{\bar{\mathcal H}}$ is an isometry, so is $\mathrm{d}\pi|_{\mathcal{H}''}:(\mathcal{H}'',\mathsf{g})\to(\mathcal{H},\mathsf{g}_M)$. The same argument with the roles of the two actions interchanged, using the definition of $\mathsf{g}_{M'}$, shows that $\mathrm{d}\pi'|_{\mathcal{H}''}:(\mathcal{H}'',\mathsf{g})\to(\mathcal{H}',\mathsf{g}_{M'})$ is an isometry. Henceforth we assume that \(M\), \(M'\), and \(P\) are equipped with these invariant Riemannian metrics.

\vspace{1em} 

Define \(\bar{\mathcal{H}} := \ker \omega\) and choose \(x' \in \pi'(\pi^{-1}(x))\) for some fixed \(x \in M\). Let \(\nu Gx\) and \(\nu Gx'\) represent the normal bundles of the \(G\)-orbits through \(x\) and \(x'\), respectively. To set the stage for the proof of many of the coming results, we revisit Proposition 5.3 in \cite{SperancaCavenaghiPublished}.

\begin{proposition}[Proposition 5.3 in \cite{SperancaCavenaghiPublished}]\label{prop:fibradonormal}
Given \(x \in M\), there exists \(x' \in \pi'(\pi^{-1}(x))\) and an isomorphism \(\Phi: \nu Gx \to \nu Gx'\) such that, for any open set \(\mathcal{O} \subseteq \nu Gx\), if \(\exp|_{\mathcal{O}}: \mathcal{O} \to M\) is a diffeomorphism onto its image, then \(\exp|_{\Phi(\mathcal{O})}: \Phi(\mathcal{O}) \to M'\) is also a diffeomorphism onto its image.
\end{proposition}
\begin{proof}
	According to item \eqref{item:isotropy} of Theorem \ref{thm:star}, for any $x \in M$, there exists $p' \in \pi^{-1}(x)$ and $g \in G$ such that $(G \times G)_{p'} = \{(h, ghg^{-1}) \mid h \in G_x\}$. Translating along the fiber by the principal element $(e,g^{-1})$, that is, setting $p := (e,g^{-1})\cdot p' = p'g \in \pi^{-1}(x)$, conjugates the isotropy to the diagonal: indeed $(e,g^{-1})(h,ghg^{-1})(e,g) = (h,h)$, so $(G\times G)_p = \Delta G_x = \{(q,q) \mid q\in G_x\}$. Set $x'=\pi'(p)$.
	
	Consider the map $\Psi: Gx\to P$, $\Psi(rx)=rpr^{-1}$, which is well defined since $(G{\times}G)_p=\Delta G_x$: if $rx=r'x$, then $r^{-1}r'\in G_x$, so $(r^{-1}r',r^{-1}r')\in (G\times G)_p$ and therefore $r'p\,r'^{-1}=r\,p\,r^{-1}$. Given $y\in Gx$, $X\in T_yM$ and $q\in \pi^{-1}(y)$, denote by $\mathcal L_q(X)\in \bar{\mathcal H}_q$ the unique bundle-horizontal vector with $\mathrm{d}\pi(\mathcal L_q(X))=X$. Both maps are equivariant:
	\begin{equation}\label{eq:equivariances}
	\Psi(ry)=r\Psi(y)r^{-1},\qquad \mathcal L_{rqs^{-1}}(rX)=r\,\mathcal L_q(X)\,s^{-1}.
	\end{equation}
	The first identity is immediate from the definition: for $y=tx$ one has $\Psi(rtx)=(rt)p(rt)^{-1}=r\big(tpt^{-1}\big)r^{-1}$. For the second, recall that the juxtaposed action preserves $\bar{\mathcal H}=\ker\omega$ — the principal directions by the second connection axiom, and the $\star$-directions because $\omega$ is $\star$-invariant — so $r\,\mathcal L_q(X)\,s^{-1}\in\bar{\mathcal H}_{rqs^{-1}}$; moreover, $\pi$ is $\star$-equivariant and constant along principal orbits, whence
	\[
	\mathrm{d}\pi\big(r\,\mathcal L_q(X)\,s^{-1}\big)=r\cdot\mathrm{d}\pi\big(\mathcal L_q(X)\big)=rX .
	\]
	Since the horizontal lift at a point is unique, the second identity follows.

 Now define $\Phi: \nu Gx\to \nu Gx'$ as
	\begin{equation}\label{eq:Phi}
	\Phi_{rx}(X)=\mathrm{d}\pi'(\mathcal L_{\Psi(rx)}(X)).
	\end{equation}
 We claim that $\Phi_{rx}$ defines an isometry between $\mathcal{H}_{rx}$ and $\mathcal H'_{rx'}$. Since both $\mathcal L_q$ and $\mathrm{d}\pi'|_{\mathcal H''}$ are isometries it suffices to show that for every $q$ it holds that $\mathcal L_q(\mathcal H_x)=\mathcal H''_q$ . 
 
 Fix $q\in\pi^{-1}(x)$ and let $X\in\mathcal H_x$, so that $X\perp T_x(Gx)$. We show that $\mathcal L_q(X)\perp T_q\big((G\times G)q\big)$. The tangent space to the juxtaposed orbit decomposes as \[ T_q\big((G\times G)q\big)=T_q\big(\star\text{-orbit}\big)+T_q\big(\bullet\text{-orbit}\big)=\{\xi^\star_q:\xi\in\mathfrak g\}+\{\eta^\bullet_q:\eta\in\mathfrak g\}, \] where $\bullet$-directions are the vertical space $\ker\mathrm d\pi_q$ of the principal bundle. Since $\mathcal L_q(X)\in\bar{\mathcal H}_q=\ker\omega_q$ and, by the definition of the Kaluza--Klein metric, $\bar{\mathcal H}_q=(\ker \mathrm d\pi_q)^{\perp}$, the vector $\mathcal L_q(X)$ is orthogonal to every $\eta^\bullet_q$. For the $\star$-directions, note that $\pi$ is equivariant, so $\mathrm d\pi_q(\xi^\star_q)=\xi^{\,\cdot}_{x}$, the action field of the $G$-action on $M$; decomposing $\xi^\star_q=\mathcal L_q\big(\xi^{\,\cdot}_x\big)+ \text{(a $\bullet$-vertical vector)}$ and using that $\mathrm d\pi|_{\bar{\mathcal H}_q}$ is an isometry onto $(T_xM,\mathsf g_M)$ together with the orthogonality of $\mathcal L_q(X)$ to the vertical space, we obtain \[ \mathsf g\big(\mathcal L_q(X),\,\xi^\star_q\big)=\mathsf g\big(\mathcal L_q(X),\,\mathcal L_q(\xi^{\,\cdot}_x)\big)=\mathsf g_M\big(X,\,\xi^{\,\cdot}_x\big)=0, \] the last equality because $X\perp T_x(Gx)$ and the $\xi^{\,\cdot}_x$ span $T_x(Gx)$. Hence $\mathcal L_q(\mathcal H_x)\subseteq\mathcal H''_q$. Since $\mathcal L_q$ is injective and, by item \eqref{item:isotropy} of Theorem \ref{thm:star}, $\dim (G\times G)q=\dim G+\dim Gx$, so that \[ \dim\mathcal H''_q=\dim P-\dim G-\dim Gx=\dim M-\dim Gx=\dim\mathcal H_x, \] the inclusion is an equality: $\mathcal L_q(\mathcal H_x)=\mathcal H''_q$.

	Let $\mathcal O\subseteq \nu Gx$ be such that $\exp|_\mathcal O$ is a diffeomorphism. Thus $\tilde\Psi: \exp(\mathcal O)\times G\to P$,
	\begin{equation}\label{eq:tildePsi}\tilde\Psi(\exp_{rx}(v),g)=\exp_{\Psi(rx)}(\mathcal L_{\Psi(rx)}(v))g^{-1},\end{equation}
	is a trivialization for $\pi$ along $\exp(\mathcal O)$. Moreover, for every $y\in Gx$,
	\begin{align}
	\tilde\Psi(\exp_{ry}(rv),sgr^{-1})&=\exp_{\Psi(ry)}(\mathcal L_{\Psi(ry)}(rv))r(sg)^{-1}=\exp_{\Psi(ry)}(\mathcal L_{r\Psi(y)r^{-1}}(rv))r(sg)^{-1} \nonumber\\\label{eq:eqvPhi}&=\exp_{\Psi(ry)}(r\mathcal L_{\Psi(y)}(v)r^{-1})r(sg)^{-1}=r\tilde\Psi(\exp_{y}(v),g) s^{-1}.
	\end{align}
Hence the image of the section $\iota:\exp(\mathcal O)\to P$, $\iota(\exp_y(v))=\tilde\Psi(\exp_y(v),e)$, meets each $\{e\}{\times}G$-orbit and each $G{\times}\{e\}$-orbit at most once. In particular, the map
	\[\exp_y(v)\mapsto \pi'(\tilde\Psi(\exp_y(v),e)) \]
is injective: two points of $\exp(\mathcal O)$ with the same image under $\pi'\circ\iota$ would place two points of the image of $\iota$ on a single $\star$-orbit. It is, moreover, an immersion. Indeed, taking $s=r$ and $g=e$ in \eqref{eq:eqvPhi} yields the diagonal equivariance $\iota(r\cdot z)=(r,r)\cdot\iota(z)$; differentiating along $r=\exp(t\eta)$, $\eta\in\mathfrak g$, gives $\mathrm{d}\iota\big(\eta^{\,\cdot}_z\big)=\eta^{\star}_{\iota(z)}-\eta^{\bullet}_{\iota(z)}$. Now let $W\in T_{\iota(z)}\big(\iota(\exp(\mathcal O))\big)\cap\ker\mathrm{d}\pi'$. Since $\star$ is free, $\ker\mathrm{d}\pi'_{\iota(z)}=\{\eta^{\star}_{\iota(z)}:\eta\in\mathfrak g\}$, so $W=\eta^{\star}_{\iota(z)}$ for some $\eta$; since $\pi\circ\iota=\mathrm{id}$ and $\mathrm{d}\pi\big(\eta^{\star}_{\iota(z)}\big)=\eta^{\,\cdot}_z$, the unique $\mathrm{d}\iota$-preimage of $W$ is $\eta^{\,\cdot}_z$, whence $W=\mathrm{d}\iota\big(\eta^{\,\cdot}_z\big)=\eta^{\star}_{\iota(z)}-\eta^{\bullet}_{\iota(z)}$. Comparing the two expressions gives $\eta^{\bullet}_{\iota(z)}=0$, and freeness of the principal action forces $\eta=0$, so $W=0$. Thus $\pi'\circ\iota$ is an injective immersion between manifolds of equal dimension, hence a diffeomorphism from $\exp(\mathcal O)$ onto the open subset $\pi'\big(\tilde{\Psi}(\exp(\mathcal O)\times\{e\})\big)$ of $M'$. On the other hand, $\mathcal L_{\Psi(rx)}(v)\in\mathcal H''\subseteq(\ker \mathrm{d}\pi')^{\bot}$ by the identification $\mathcal L_q(\mathcal H_x)=\mathcal H''_q$ established above; since $\pi'$ is a Riemannian submersion, it carries horizontal geodesics to geodesics, so that $\pi'\circ\exp=\exp\circ\,\mathrm{d}\pi'$ on horizontal vectors, and \begin{align}
	\pi'(\tilde\Psi(\exp_{rx}(v),e)) &= \pi'\exp_{\Psi(rx)}(\mathcal L_{\Psi(rx)}(v)) \nonumber \\
    &= \exp_{\pi'(\Psi(rx))}(\mathrm{d}\pi'\mathcal L_{\Psi(rx)}(v)) \nonumber \\
    &= \exp_{rx'}(\Phi(v)).
	\end{align}
	We conclude that $\pi' \circ \iota = \exp \circ \Phi \circ (\exp|_{\mathcal O})^{-1}$ on $\exp(\mathcal O)$. The left-hand side is a diffeomorphism onto its open image, and $\Phi\circ(\exp|_{\mathcal O})^{-1}$ is a diffeomorphism from $\exp(\mathcal O)$ onto $\Phi(\mathcal O)$, the map $\Phi$ being smooth — smoothness of $\Psi$ descends from the smooth map $r\mapsto (r,r)\cdot p$ through the submersion $G\to Gx$ — and a linear isometry on each fiber. Therefore $\exp|_{\Phi(\mathcal O)}:\Phi(\mathcal O)\to M'$ is a diffeomorphism onto its image, as claimed.
 \end{proof}

A notable consequence of Proposition \ref{prop:fibradonormal} is presented below: the construction of $\star$-diagrams reduces to the search for commuting \emph{$\star$-actions} satisfying the isotropy property \eqref{item:isotropy} of Theorem \ref{thm:star}, since every such action arises from a $\star$-collection. Theorem \ref{thm:star} and Theorem \ref{thm:equivalentstar} are thus converse to one another.

\begin{theorem}\label{thm:equivalentstar}
Let $G$ be a compact Lie group, let $\pi:P\to M$ be a $G$-principal bundle, and let $P$ carry a further smooth $G$-action $\star$ commuting with the principal one. Suppose the juxtaposed $G\times G$-action $(r,s)\cdot p=r\star p\cdot s^{-1}$ satisfies the isotropy condition \eqref{item:isotropy} of Theorem \ref{thm:star}. Then there exist a $G$-invariant open cover $\{U_i\}_{i\in I}$ of $M$ and local trivializations of $P$ over it whose transition functions $\{\phi_{ij}\}$ satisfy the cocycle condition and the covariance condition \eqref{eq:covariance}, that is, they form a $\star$-collection for $P$. Moreover, the given $\star$-action coincides with the one determined chartwise by this collection, by \eqref{eq:star-chartwise-def}, and
\[
M'\;\longleftarrow\;P\;\longrightarrow\;M,\qquad M':=P/\star,
\]
is a $\star$-diagram.
\end{theorem}
\begin{proof}
Equip $M$ with a $G$-invariant metric and $P$ with the associated $G\times G$-invariant Kaluza--Klein metric of Section \ref{sec:diff-geo}, so that $\pi$ is a Riemannian submersion and the exponential maps are equivariant. Throughout we write $(r,s)\cdot p=r\star p\cdot s^{-1}$ for the juxtaposed action and $p\cdot a$ for the principal one.

 For $x\in M$, the orbit $Gx$ is a compact embedded submanifold, and by the equivariant tubular neighborhood theorem \cite[Ch.~VI]{brebook} there is a $G$-invariant open neighborhood $\mathcal{O}_x\subseteq\nu Gx$ of the zero section such that $\exp|_{\mathcal{O}_x}$ is a diffeomorphism onto an open tubular neighborhood $U_x:=\exp(\mathcal{O}_x)$ of $Gx$. Since $\exp$ is equivariant, $U_x$ is $G$-invariant. The sets $U_x$ depend only on the orbit of $x$; choose a subfamily $\{U_i\}_{i\in I}$ covering $M$.

 By Proposition \ref{prop:fibradonormal} applied to each $\mathcal{O}_i$, the bundle $P$ admits over $U_i$ the trivialization $\tilde\Psi_i$ of \eqref{eq:tildePsi}, which satisfies \eqref{eq:eqvPhi}:
\begin{equation}\label{eq:eqv-used}
\tilde\Psi_i(r\cdot z,\;s\,g\,r^{-1})\;=\;(r,s)\cdot\tilde\Psi_i(z,\,g),\qquad r,s,g\in G,\ z\in U_i .
\end{equation}
Set $\sigma_i:=\tilde\Psi_i(\cdot\,,e)$, a smooth local section of $\pi$ over $U_i$, so that $\tilde\Psi_i(z,g)=\sigma_i(z)\cdot g^{-1}$. Taking $s=e$ and $g=r$ in \eqref{eq:eqv-used},
\begin{equation}\label{eq:sigma-diagonal}
\sigma_i(r\cdot z)=r\star\big(\sigma_i(z)\cdot r^{-1}\big)=\big(r\star\sigma_i(z)\big)\cdot r^{-1}=(r,r)\cdot\sigma_i(z),
\end{equation}
i.e. each $\sigma_i$ is equivariant for the diagonal action.

 For $i,j\in I$ and $z\in U_i\cap U_j$, the points $\sigma_i(z)$ and $\sigma_j(z)$ lie in the same $\pi$-fiber, on which the principal action is simply transitive; let $\phi_{ji}(z)\in G$ be the unique element with
\[
\sigma_i(z)=\sigma_j(z)\cdot\phi_{ji}(z).
\]
Then $\phi_{ji}$ is smooth, being the composite of $(\sigma_j,\sigma_i)$ with the division map of the principal bundle, and $\sigma_i=\sigma_k\cdot\phi_{ki}$ together with $\sigma_i=\sigma_j\cdot\phi_{ji}$ and $\sigma_j=\sigma_k\cdot\phi_{kj}$ gives the cocycle condition $\phi_{ki}=\phi_{kj}\,\phi_{ji}$ on triple overlaps. The overlap $U_i\cap U_j$ is $G$-invariant, and for $r\in G$ we compute, using \eqref{eq:sigma-diagonal} and the commutativity of the two actions,
\begin{align}
\sigma_j(r\cdot z)\cdot\phi_{ji}(r\cdot z)
&=\sigma_i(r\cdot z)=\big(r\star\sigma_i(z)\big)\cdot r^{-1}
=\Big(r\star\big(\sigma_j(z)\cdot\phi_{ji}(z)\big)\Big)\cdot r^{-1}\nonumber\\
&=\big(r\star\sigma_j(z)\big)\cdot\phi_{ji}(z)\,r^{-1}
=\Big[\big(r\star\sigma_j(z)\big)\cdot r^{-1}\Big]\cdot\big(r\,\phi_{ji}(z)\,r^{-1}\big)\nonumber\\
&=\sigma_j(r\cdot z)\cdot\big(r\,\phi_{ji}(z)\,r^{-1}\big).
\end{align}
Since the principal action is free, we conclude the covariance condition \eqref{eq:covariance}:
\[
\phi_{ji}(r\cdot z)=r\,\phi_{ji}(z)\,r^{-1},\qquad r\in G,\ z\in U_i\cap U_j .
\]
The normalization conditions also hold: $\sigma_i=\sigma_i\cdot\phi_{ii}$ forces $\phi_{ii}\equiv e$, and $\sigma_i=\sigma_j\cdot\phi_{ji}$ together with $\sigma_j=\sigma_i\cdot\phi_{ij}$ forces $\phi_{ij}=\phi_{ji}^{-1}$, by freeness of the principal action. Hence $\{\phi_{ij}\}$ is a $\star$-collection for $P$ relative to the invariant cover $\{U_i\}$.

 Rewriting \eqref{eq:eqv-used} with $s=e$ in the chart coordinates $(z,g)$ of $\tilde\Psi_i$, the given $\star$-action reads
\begin{equation}\label{eq:star-chartwise}
r\star(z,\,g)=\big(r\cdot z,\;g\,r^{-1}\big),
\end{equation}
which is precisely the action determined chartwise by the collection $\{\phi_{ij}\}$; the covariance condition is exactly what makes the chartwise formulas agree on overlaps, since
\[
\tilde\Psi_j^{-1}\big(r\star\tilde\Psi_i(z,g)\big)=\big(r\cdot z,\;g\,r^{-1}\phi_{ji}(r\cdot z)^{-1}\big)=\big(r\cdot z,\;g\,\phi_{ji}(z)^{-1}r^{-1}\big)=r\star\Big(\tilde\Psi_j^{-1}\tilde\Psi_i(z,g)\Big).
\]
By \eqref{eq:star-chartwise}, a fixed point of $r\star$ forces $g\,r^{-1}=g$, hence $r=e$: the $\star$-action is free — as it is for every $\star$-collection, by \eqref{eq:star-chartwise-def}. Since $G$ is compact, this free action is proper, so $M'=P/\star$ is a smooth manifold and $\pi':P\to M'$ is a $G$-principal bundle; alternatively, this is Theorem \ref{thm:star} applied to the collection $\{\phi_{ij}\}$. The two actions are free, commute, and have quotients $M$ and $M'$; this is the datum of the $\star$-diagram $M'\leftarrow P\rightarrow M$.
\end{proof}

\vspace{1em}

The two base manifolds of a $\star$-diagram therefore have closely related invariant geometries. On the one hand, Proposition \ref{prop:fibradonormal} produces, for each $x\in M$, a point $x'\in M'$ and an isomorphism $\Phi:\nu Gx\to\nu Gx'$ of normal bundles which intertwines the isotropy representations, the isotropy groups themselves being related by $G_{x'}=g\,G_x\,g^{-1}$ for the element $g$ of \eqref{item:isotropy}; by the Slice Theorem \cite[Theorem 3.82]{alexandrino2015lie}, the $G$-manifolds $M$ and $M'$ have the same orbit types, with matching slice representations.  On the other hand, they share the same transversal geometry. Since $\pi$ and $\pi'$ are $G$-equivariant surjective submersions with $\pi^{-1}(Gx)=(G\times G)p=\pi'^{-1}(Gx')$ for $p\in\pi^{-1}(x)$ and $x'=\pi'(p)$, both induce the same identification of orbit spaces, \[ M/G\;\cong\;P/(G\times G)\;\cong\;M'/G . \] This identification is an isometry for the quotient metrics: the restrictions $\mathrm{d}\pi|_{\mathcal H''}$ and $\mathrm{d}\pi'|_{\mathcal H''}$ are isometries onto $\mathcal H$ and $\mathcal H'$ (Section \ref{sec:diff-geo}), so the correspondence $\mathcal{H} \leftarrow \mathcal{H}'' \rightarrow \mathcal{H}'$ carries horizontal curves in $M$ to horizontal curves in $M'$ of the same length, and conversely; taking infima over such curves gives $\mathrm{d}_{M/G}=\mathrm{d}_{M'/G}$ under the identification above. In particular, horizontal geodesics correspond to horizontal geodesics. 

\vspace{1em} 

\subsection{\texorpdfstring{$\star$}{*}-diagrams and Morita equivalence}  
\label{sec:moritaequive}

This section examines the extent to which the two base manifolds of a $\star$-diagram are equivalent from the point of view of their $G$-invariant structure. We do so through two notions of Morita equivalence: one for Lie groupoids and one for singular foliations. The upshot is that, although $M$ and $M'$ may fail to be homeomorphic, their orbit foliations are Hausdorff--Morita equivalent and their action groupoids are Morita equivalent; consequently the invariants attached to these objects cannot distinguish $M$ from $M'$. Any obstruction to $M\cong M'$ must therefore be sought elsewhere.

\vspace{1em} 

Following the terminology in \cite{Garmendia2019}, a \emph{singular foliation} on a manifold $X$ is a $C^\infty(X)$-submodule $\mathcal F$ of the compactly supported vector fields $\mathfrak{X}_c(X)$, closed under the Lie bracket and locally finitely
generated. A \emph{foliated manifold} is a manifold with a singular foliation.

A smooth manifold $X$ with a proper action of a connected Lie group $G$ is a foliated manifold: the $C^\infty(X)$-module \[ \mathcal{F}_G:=\Big\langle\, U^*\;:\;U\in\mathfrak{g}\,\Big\rangle_{C^\infty_c(X)}\subseteq\mathfrak{X}_c(X) \] generated by the action vector fields is involutive, by $[U^*,V^*]=-[U,V]^*$, and finitely generated, a basis of $\mathfrak{g}$ providing global generators; we call it the \emph{orbit foliation} of the action, its leaves being the $G$-orbits. Specifically, if $\mathfrak{g}$ denotes the Lie algebra of the Lie group $G$ and $\exp: \mathfrak{g} \to G$ is the Lie exponential map, the action vector field $U^*$ associated with $U \in \mathfrak{g}$ is defined at any $x \in X$ by
\begin{equation}\label{eq:action}
U^*_x = \left.\frac{d}{dt}\right|_{t=0} \exp(tU) \cdot x,
\end{equation}
where $\cdot$ denotes the $G$-action on $X$.

Given a $\star$-diagram $M \leftarrow P \rightarrow M'$ with structure group $G$, recall that, because the two actions commute, each descends to the quotient by the other: the $\star$-action descends to $M=P/\bullet$ through $r\cdot\pi(p):=\pi(r\star p)$, and the $\bullet$-action descends to $M'=P/\star$ through $s\odot\pi'(p):=\pi'(p\cdot s^{-1})$. Both $M$ and $M'$ are thus $G$-manifolds, and we always regard them as such. Next, we show the existence (Theorem \ref{thm:starsaremorita}) of a \emph{Hausdorff--Morita equivalence} (Definition 2.1 in \cite{Garmendia2019}) between the induced singular foliations on the two $G$-manifolds $M$ and $M'$.  

\begin{theorem}\label{thm:starsaremorita} 
Any $\star$-bundle $M \leftarrow P \rightarrow M'$ with connected structure group $G$ is an example of a Hausdorff--Morita equivalence. Specifically, the singular foliation induced by the $G$-orbits on $M$ and $M'$ can be pulled back to isomorphic singular foliations in $P$ via the corresponding projections. 
\end{theorem} 

\begin{proof} 
As stated in Corollary 2.17 in \cite{Garmendia2019}, if two connected Lie groups $G_1$ and $G_2$ act freely and properly on a manifold $P$ with commuting actions, the singular foliation on $P/G_1$ induced by the residual $G_2$-action is Hausdorff--Morita equivalent to the singular foliation on $P/G_2$ induced by the residual $G_1$-action. A $\star$-diagram is exactly this situation with $G_1=G_2=G$: taking $G_1$ to be the $\bullet$-action gives $P/G_1=M$, carrying the foliation induced by the residual $\star$-action, and taking $G_2$ to be the $\star$-action gives $P/G_2=M'$, carrying the foliation induced by the residual $\bullet$-action. Both actions are free by the definition of a $\star$-diagram, and proper because $G$ is compact. Connectedness of $G$ is used here, and only here: the singular foliation induced by a $G$-action is generated by the action vector fields \eqref{eq:action}, hence depends only on the identity component of $G$, and Corollary 2.17 of \cite{Garmendia2019} is stated under this hypothesis.
\end{proof} 

\vspace{1em}

\begin{definition}[Projective foliations]\label{def:projectivefoliation} 
A manifold $M$ equipped with a singular foliation $\mathcal{F}$ is termed \emph{projective} if there exist a vector bundle $E \to M$ and a bundle map $\rho: E \to TM$, injective on a dense open subset of $M$, such that
\[
\rho\big(\Gamma_c(E)\big)=\mathcal{F}\qquad\text{and}\qquad \Gamma_c(E) \cong \mathcal{F} \ \text{ as } C^{\infty}(M)\text{-modules},
\]
where $\Gamma_c(E)$ denotes the space of compactly supported smooth sections of $E$. The map $\rho$ is then an \emph{almost injective} anchor, and $E$ carries the structure of an almost injective Lie algebroid; see \cite[p.~484]{clairedebord} and \cite[Chapter 3]{mackenzie_2005} for further details. 
\end{definition}

As shown in Example \ref{ex:almostfree} below, if either the $G$-action on $M$ or the $G$-action on $M'$ is \emph{almost free} (i.e., every isotropy subgroup is discrete), then the induced singular foliations on $M$ and $M'$ are projective foliations.

\begin{example}\label{ex:almostfree} 
Let $X$ be a smooth manifold carrying an almost free action of a compact Lie group $G$; compactness guarantees that the action is proper, so the orbit foliation $\mathcal{F}_G$ of \eqref{eq:action} is defined. Let $\mathfrak{g}$ be the Lie algebra of $G$, define the trivial vector bundle $E = X \times \mathfrak{g} \to X$, and let $\rho: E \to TX$ be the anchor
\[
\rho_x(v) = v^*_x,\qquad v\in\mathfrak{g},
\]
the evaluation at $x$ of the action vector field of $v$. For every $x\in X$ one has $\ker\rho_x = \mathfrak{g}_x := \mathrm{Lie}(G_x)$, since $v^*_x=0$ precisely when the one-parameter subgroup $\exp(tv)$ fixes $x$. Because the action is almost free, each isotropy group $G_x$ is discrete, so $\mathfrak{g}_x = 0$ and $\rho_x$ is injective for every $x$; hence $\rho$ is injective as a bundle map, in particular almost injective in the sense of Definition \ref{def:projectivefoliation}. Moreover $\rho\big(\Gamma_c(E)\big)=\mathcal{F}_G$, a basis of $\mathfrak{g}$ providing global generators, and $\Gamma_c(E)\cong\mathcal{F}_G$ as $C^\infty(X)$-modules. Thus the orbit foliation is projective, induced by the almost injective Lie algebroid $(E,\rho)$ \cite[p.~496]{clairedebord}.

Item \eqref{item:isotropy} of Theorem \ref{thm:star} relates the two isotropy groups along the diagram: for $p\in P$ there is $g\in G$ with \[ (G\times G)_p=\{(h,\,ghg^{-1})\;:\;h\in G_{\pi(p)}\}, \] and projecting onto the two factors gives $G_{\pi'(p)}=g\,G_{\pi(p)}\,g^{-1}$. In particular, corresponding isotropy groups are isomorphic, hence one is discrete precisely when the other is. Since every point of $M$ is of the form $\pi(p)$ and every point of $M'$ of the form $\pi'(p)$, the $G$-action on $M$ is almost free if and only if the $G$-action on $M'$ is almost free. Thus, the foliation on $M$ is projective if and only if the foliation on $M'$ is projective. This is in accordance with Proposition 2.7 in \cite{Garmendia2019}, which states that an existing Hausdorff--Morita equivalence between two manifolds with singular foliations ensures that if one of such foliations is projective, the other one is also projective.
\end{example}

\ 

Recall that every $G$-manifold $X$ gives rise to an action Lie groupoid, conventionally denoted $G \ltimes X$. The object space is $X$, and the arrow space is $G \times X$, equipped with source map $s(g,x) = x$ and target map $t(g,x) = g \cdot x$. Here, we prove that there is a Morita equivalence (Definition \ref{def:morita}) between the Lie groupoids $G \ltimes M$ and $G \ltimes M'$\footnote{This fact was pointed out by O. Brahic and C. Ortiz in 2015 to L. Sperança, who then communicated it to the authors.}.

\begin{definition}[Morita equivalence of action groupoids, \cite{Garmendia2019}] \label{def:morita}
Two action groupoids $G \ltimes M$ and $H \ltimes N$ are said to be \emph{Morita equivalent} if there exist a Hausdorff manifold $P$ and two surjective submersions $\pi_M: P \rightarrow M$ and $\pi_N: P \rightarrow N$ such that the pullback groupoids $\pi_M^{*}(G \ltimes M)$ and $\pi_N^{*}(H \ltimes N)$, both of which have $P$ as object space, are isomorphic as Lie groupoids by an isomorphism which is the identity on objects. 
\end{definition}

\begin{theorem}\label{thm:morita}
Let $M'\xleftarrow{\ \pi'\ } Q\xrightarrow{\ \pi\ } M$ be a $\star$-diagram with structure group $G$, and endow the two base manifolds with the residual $G$-actions
\[
r\cdot \pi(q):=\pi(r\star q),\qquad s\odot \pi'(q):=\pi'(q\,s^{-1}).
\]
Then the action Lie groupoids $G\ltimes M$ and $G\ltimes M'$ are Morita equivalent. Moreover, for every $q\in Q$ the assignment $g\mapsto k$, where $k$ is determined by $g\star q=q\,k$, is an isomorphism of isotropy groups $G_{\pi(q)}\xrightarrow{\ \cong\ }G_{\pi'(q)}$, and the orbit spaces are canonically identified: $M/G\cong Q/(G\times G)\cong M'/G$.
\end{theorem}
\begin{proof}
By definition of a $\star$-diagram, the $\bullet$- and $\star$-actions on $Q$ are free, commute, and exhibit $\pi$ and $\pi'$ as principal $G$-bundles. We write the $\bullet$-action on the right, $(q,s)\mapsto qs$, and the $\star$-action on the left, $(r,q)\mapsto r\star q$, so that commutativity reads $r\star(qs)=(r\star q)s$.

Since $\pi\big(r\star(qs)\big)=\pi\big((r\star q)s\big)=\pi(r\star q)$, the assignment $r\cdot\pi(q):=\pi(r\star q)$ is well defined; it is a left action because $\star$ is, and it is smooth because $\pi$ is a surjective submersion. Symmetrically, $\pi'\big((r\star q)s^{-1}\big)=\pi'\big(r\star(qs^{-1})\big)=\pi'(qs^{-1})$ shows that $s\odot\pi'(q):=\pi'(qs^{-1})$ is well defined, and it is a \emph{left} action, since
\[
s_1\odot\big(s_2\odot\pi'(q)\big)=\pi'\big(q\,s_2^{-1}s_1^{-1}\big)=\pi'\big(q\,(s_1s_2)^{-1}\big)=(s_1s_2)\odot\pi'(q).
\]
Thus $G\ltimes M$ and $G\ltimes M'$ are the action groupoids of left $G$-actions, with arrows $(g,x):x\to g\cdot x$.

Because the two actions commute,
\[
(g,k)\ast q:=(g\star q)\,k^{-1},\qquad (g,k)\in G\times G,\ q\in Q,
\]
defines a smooth left $G\times G$-action on $Q$:
\[
(g_1,k_1)\ast\big((g_2,k_2)\ast q\big)=\Big(g_1\star\big((g_2\star q)k_2^{-1}\big)\Big)k_1^{-1}
=\big((g_1g_2)\star q\big)\,k_2^{-1}k_1^{-1}=(g_1g_2,\;k_1k_2)\ast q .
\]

 Since $\pi$ is a surjective submersion, the pullback $\mathcal{G}_M:=\pi^{*}(G\ltimes M)\rightrightarrows Q$ is a Lie groupoid, with arrow space
\[
\mathcal{G}_M=\big\{(q_1,g,q_2)\in Q\times G\times Q:\ \pi(q_1)=g\cdot\pi(q_2)\big\}
\]
and composition $(q_0,h,q_1)\circ(q_1,g,q_2)=(q_0,hg,q_2)$. We claim that
\[
\Theta:(G\times G)\ltimes Q\longrightarrow\mathcal{G}_M,\qquad
\Theta\big((g,k),q\big)=\big((g\star q)k^{-1},\;g,\;q\big),
\]
is an isomorphism of Lie groupoids over $\mathrm{id}_Q$. It lands in $\mathcal{G}_M$, since $\pi\big((g\star q)k^{-1}\big)=\pi(g\star q)=g\cdot\pi(q)$; it is smooth; it preserves sources and targets by construction, and composition by the computation above. It is bijective: given $(q_1,g,q_2)\in\mathcal{G}_M$, the points $q_1$ and $g\star q_2$ lie in the same $\pi$-fiber, on which the $\bullet$-action is simply transitive, so there is a unique $k\in G$ with $q_1=(g\star q_2)k^{-1}$. The inverse map $(q_1,g,q_2)\mapsto\big((g,k),q_2\big)$ is smooth. Indeed, recall that a principal $G$-bundle $\pi:Q\to M$ comes equipped with its \emph{division map} \[ \delta:Q\times_M Q\longrightarrow G,\qquad q'=q\cdot\delta(q,q'), \] defined on the fiber product $Q\times_M Q=\{(q,q'): \pi(q)=\pi(q')\}$; the element $\delta(q,q')$ exists and is unique because the principal action is simply transitive on the fibers, and it depends smoothly on $(q,q')$ because the map $Q\times G\to Q\times_M Q$, $(q,a)\mapsto (q,qa)$, is a diffeomorphism. In the case at hand, $\pi(g\star q_2)=g\cdot\pi(q_2)=\pi(q_1)$, so the pair $(g\star q_2,\,q_1)$ lies in $Q\times_M Q$ and \[ k^{-1}=\delta\big(g\star q_2,\;q_1\big), \] which is smooth in $(q_1,g,q_2)$.

Exchanging the roles of the two actions, the map
\[
\Theta':(G\times G)\ltimes Q\longrightarrow\mathcal{G}_{M'}:=(\pi')^{*}(G\ltimes M'),\qquad
\Theta'\big((g,k),q\big)=\big((g\star q)k^{-1},\;k,\;q\big),
\]
is likewise an isomorphism of Lie groupoids over $\mathrm{id}_Q$. It lands in $\mathcal{G}_{M'}$ because
\[
\pi'\big((g\star q)k^{-1}\big)=\pi'\big(g\star(q\,k^{-1})\big)=\pi'\big(q\,k^{-1}\big)=k\odot\pi'(q),
\]
using commutativity and the $\star$-invariance of $\pi'$; and it preserves composition, since $(h,k_h)\,(g,k_g)=(hg,\,k_hk_g)$ is sent to $k_hk_g$. For bijectivity, let $(q_1,k,q_2)\in\mathcal{G}_{M'}$, so that $\pi'(q_1)=k\odot\pi'(q_2)=\pi'(q_2k^{-1})$; thus $q_1$ and $q_2k^{-1}$ lie in the same fiber of $\pi'$, on which the $\star$-action is simply transitive, and there is a unique $g\in G$ with $q_1=g\star(q_2k^{-1})=(g\star q_2)k^{-1}$. Writing $\delta':Q\times_{M'}Q\to G$ for the division map of the principal bundle $\pi'$, normalized by $q''=\delta'(q'',q''')\star q'''$, we obtain $g=\delta'\big(q_1,\;q_2k^{-1}\big)$, which depends smoothly on $(q_1,k,q_2)$; hence the inverse of $\Theta'$ is smooth.

 Hence $\mathcal{G}_M\cong(G\times G)\ltimes Q\cong\mathcal{G}_{M'}$ as Lie groupoids, by an isomorphism which is the identity on objects. Since the pullback of a Lie groupoid along a surjective submersion is Morita equivalent to it, and Morita equivalence is an equivalence relation,
\[
G\ltimes M\ \sim\ \mathcal{G}_M\ \cong\ \mathcal{G}_{M'}\ \sim\ G\ltimes M'.
\]
For the isotropy statement, restrict $\Theta'\circ\Theta^{-1}$ to the arrows from $q$ to $q$. In $\mathcal{G}_M$ these are the triples $(q,g,q)$ with $g\in G_{\pi(q)}$, and $\Theta^{-1}(q,g,q)=\big((g,k),q\big)$ with $q=(g\star q)k^{-1}$, that is, $g\star q=q\,k$; in $\mathcal{G}_{M'}$ they are the triples $(q,k,q)$ with $k\in G_{\pi'(q)}$. The correspondence $g\mapsto k$ is therefore a bijection $G_{\pi(q)}\to G_{\pi'(q)}$, and it is a homomorphism:
\[
(hg)\star q=h\star(g\star q)=h\star(q\,k_g)=(h\star q)\,k_g=q\,k_h k_g .
\]
Finally, both $M/G$ and $M'/G$ are the quotient of $Q$ by the action $\ast$, whence the identification of orbit spaces.
\end{proof}


One consequence concerns equivariant cohomology.

\begin{theorem}\label{thm:equivanlentcohomologies}
    Let $M\leftarrow P\rightarrow M'$ be a $\star$-diagram with compact structure group $G$, and assume that the induced $G$-action on $M$ is almost free. Then the induced $G$-action on $M'$ is almost free, and there is an isomorphism of equivariant cohomologies
    \[\mathrm{H}^*_G(M;\mathbb{R})\cong \mathrm{H}^*_G(M';\mathbb{R}).\]
\end{theorem}
\begin{proof}
That the $G$-action on $M$ is almost free if and only if the $G$-action on $M'$ is almost free was established in Example \ref{ex:almostfree}. Let $X$ denote either $M$ or $M'$. The isotropy groups of the $G$-action on $X$ are then finite: they are discrete by almost freeness and closed in the compact group $G$, hence compact, and a compact discrete group is finite.

For a compact Lie group acting with finite isotropy groups, the projection of the Borel construction onto the orbit space induces an isomorphism in real cohomology,
\[
\mathrm{H}_G^{*}(X;\mathbb{R})\;=\;\mathrm{H}^{*}\big(EG\times_G X;\mathbb{R}\big)\;\cong\;\mathrm{H}^{*}(X/G;\mathbb{R});
\]
see \cite{Brion1998}. Applying this to $X=M$ and to $X=M'$, and using the canonical identification of orbit spaces $M/G\cong P/(G\times G)\cong M'/G$ furnished by Theorem \ref{thm:morita}, we obtain
\[
\mathrm{H}_G^{*}(M;\mathbb{R})\;\cong\;\mathrm{H}^{*}(M/G;\mathbb{R})\;\cong\;\mathrm{H}^{*}(M'/G;\mathbb{R})\;\cong\;\mathrm{H}_G^{*}(M';\mathbb{R}). \qedhere
\]
\end{proof}

\vspace{1em} 

\section{Spherical T-duality}
\label{sec:Tdualityandstardiagrams}
In a series of works \cite{Bouwknegt20041, Bouwknegt20042, Bouwknegt20043, Bouwknegt2005}, Bouwknegt, Evslin, Hannabuss, and Mathai proved that for each pair $(P, H)$ consisting of a manifold $P$ with a free circle action and an integral 3-cocycle $H$ on $P$, there exists a unique (up to isomorphism) associated T-dual pair $(\widehat{P}, \widehat{H})$. Here, $\widehat{P}$ denotes a manifold with a free circle action and a cocycle $\widehat{H}$, such that the orbit space for both circle actions is identical. Although $P$ and $\widehat{P}$ are generally not homeomorphic, it has been proven that T-duality induces various degree-shifting isomorphisms between different structures, such as twisted cohomology and twisted K-theory in $P$ and $\widehat{P}$. Subsequent work \cite{gualtieri, cavalcanti} established that T-duality also induces isomorphisms in Dirac structures, Courant algebroids, generalized complex structures, and generalized Kähler structures. Later, a higher-dimensional version of T-duality was introduced in \cite{Bouwknegt2015, Bouwknegt20152}, focusing on sphere bundles. More recently, this concept was further generalized in \cite{Lind2020,cavalcanti2025topologicalsphericaltduality}.

A primary goal of the present work is to realize different smooth structures as spherical T-dual manifolds and establish a correspondence between these T-dual pairs and $\star$-diagrams. The former promotes geometric realizations of these essentially topological (T-duality) constructions. 

\vspace{1em} 

\subsection{Milnor bundles: a review}
\label{sec:milnorbundles}

\begin{definition}
    A sphere bundle $\mathrm{S}^{n{-}1}\hookrightarrow M\to B$ is called \textit{linear} if it admits $\mathrm{O}(n)$ (acting in the standard way on $\mathrm{S}^{n{-}1}$) as its structure group. Equivalently, this holds if there exists a set of transition functions $\{\phi_{ij}:U_i\cap U_j\to \mathrm{O}(n)\}$.
\end{definition}

The usual boundary map in the long homotopy sequence of the fibration $EG\to BG$, for $G=\mathrm{SO}(n)$, provides a bijection between the set of linear $\mathrm{S}^{n-1}$-bundles over $\mathrm{S}^l$ and $\pi_{l-1}(\mathrm{SO}(n))$. For topological computations aiming to determine $\pi_{l-1}(\mathrm{SO}(4))$ for $l\geq 3$, it suffices to compute $\pi_{l-1}(\mathrm{S}^{3}\times\mathrm{S}^3)$, since coverings induce isomorphisms on homotopy groups in degrees at least two and we have a canonical two-fold covering homomorphism of Lie groups
\begin{equation}
    \Psi : \mathrm{S}^3\times \mathrm{S}^3\rightarrow \mathrm{SO}(4)
\end{equation}
given by
\begin{equation}
    \Psi(x,y)v = xv\bar{y}, \quad \text{for } x,y \in \mathrm{S}^3,~v\in \mathbb H.
\end{equation}
Specializing to $l=n=4$, we obtain $ \pi_3(\mathrm{SO}(4)) \cong \pi_3(\mathrm{S}^3\times\mathrm{S}^3)\cong \mathbb Z\oplus \mathbb Z$. The linear $\mathrm{S}^3$-bundles over $\mathrm{S}^4$ are known as \emph{Milnor bundles}.

As in \cite{mi}, for pairs of integers $(m,n) \in \mathbb Z\oplus \mathbb Z \cong \pi_3(\mathrm{SO}(4))$, we can build the maps $t_{m,n}:\mathrm{S}^3\to \mathrm{SO}(4)$,
\begin{equation}
    t_{m,n}(x)v=x^mvx^n, \quad v\in\mathbb H,
\end{equation}
whose classes realize the isomorphism $\mathbb{Z}\oplus\mathbb{Z}\cong\pi_3(\mathrm{SO}(4))$, $(m,n)\mapsto[t_{m,n}]$. Let $f_{m,n}:=f_{t_{m,n}}$ be the associated clutching diffeomorphism
\begin{equation}
     f_{m,n}(x,g) := (x,t_{m,n}(x)g) = (x, x^m g x^n), \quad \text{for } (x,g) \in \mathrm{S}^3\times \mathrm{S}^3.
\end{equation}
The manifold $M_{m,n}=\left(\mathrm{D}^4_+\times \mathrm{S}^3\right)\cup_{f_{m,n}}\left(\mathrm{D}^4_-\times \mathrm{S}^3\right)$, obtained by gluing two copies of $\mathrm{D}^4\times\mathrm{S}^3$ along the common boundary $\mathrm{S}^3\times \mathrm{S}^3$ over the decomposition $\mathrm{S}^4=\mathrm{D}^4_+\cup_{\mathrm{S}^3}\mathrm{D}^4_-$, gives rise to a Milnor bundle. (The presentation $\mathrm{D}^4\times\mathrm{S}^3\cup\mathrm{S}^3\times\mathrm{D}^4$ employed in Section \ref{sec:logtrans_final} for the homotopy spheres $m+n=\pm1$ is obtained from this one by relabeling the second piece through the factor swap $\mathrm{D}^4_-\times\mathrm{S}^3\to\mathrm{S}^3\times\mathrm{D}^4$, the gluing map being composed with the swap accordingly, and conversely; the two presentations therefore describe the same manifold, and we pass between them freely.) In \cite{mi}, Milnor observed that $M_{m,n}$ is homeomorphic to $\mathrm{S}^7$ if and only if $m+n=\pm 1$, and that it fails to be diffeomorphic to $\mathrm{S}^7$ for suitable choices of $(m,n)$, such as $(2,-1)$. This is established by examining Milnor's \emph{$\lambda$-invariant}, an element of $\mathbb{Z}_7$ which vanishes for the standard sphere and is given by
\begin{equation}
    \lambda(M_{m,1-m}) = (-1+2m)^2-1 \pmod{7}.
\end{equation}
For later use we record the characteristic classes: denoting by $u\in\mathrm{H}^4(\mathrm{S}^4;\mathbb{Z})$ the orientation generator, the Milnor bundle $M_{m,n}$ has Euler class $\mathrm{e}(M_{m,n})=(m+n)u$ and, with respect to the orientation conventions of \cite{mi}, first Pontryagin class $p_1 = \pm 2(m-n)u$ of the associated rank-four bundle \cite{mi,steenrod}; we fix orientations once and for all so that $p_1=2(m-n)u$. With this normalization, $p_1=2\,\mathrm{e}$ precisely when $n=0$, and $p_1=-2\,\mathrm{e}$ precisely when $m=0$; these are exactly the two families of Milnor bundles admitting a principal structure (Lemma \ref{lem:principal_s3_condition} below).

\vspace{1em} 

\subsection{Oriented \texorpdfstring{$\mathrm{S}^3$}{S3}-bundles over 4-manifolds}

An \emph{oriented $\mathrm{S}^3$-bundle} $E$ over a manifold $B$ is a fiber bundle over $B$ with fiber $\mathrm{S}^3$ and structure group $\mathrm{Diff}_+(\mathrm{S}^3)$, the group of orientation-preserving diffeomorphisms of $\mathrm{S}^3$; such bundles need not be principal. The inclusion $\mathrm{SO}(4) = \mathrm{Iso}_+(\mathrm{S}^3) \hookrightarrow \mathrm{Diff}_+(\mathrm{S}^3)$ of the orientation-preserving isometries of $\mathrm{S}^3$ into the orientation-preserving diffeomorphisms of $\mathrm{S}^3$ is a homotopy equivalence \cite[Theorem A]{hatcherproof}. Consequently, the quotient space $\mathrm{Diff}_+(\mathrm{S}^3)/\mathrm{SO}(4)$ is contractible. This topological triviality implies that the associated bundle with this quotient space as fiber admits a global section, allowing us to assume, without loss of generality, that the structure group of a general oriented $\mathrm{S}^3$-bundle $E$ reduces to $\mathrm{SO}(4)$; that is, every oriented $\mathrm{S}^3$-bundle is isomorphic to an oriented \emph{linear} $\mathrm{S}^3$-bundle, and we use the latter terminology throughout.

Any principal $\mathrm{SO}(4)$-bundle $P$ over $\mathrm{S}^4$ is classified by $\pi_3(\mathrm{SO}(4)) \cong \mathbb{Z} \oplus \mathbb{Z}$. Thus any pair of integers $(m, n)$ defines a principal $\mathrm{SO}(4)$-bundle $P_{m,n}$ whose associated oriented linear $\mathrm{S}^3$-bundle is the Milnor bundle $M_{m,n}$ over $\mathrm{S}^4$. Throughout, the characteristic classes of a principal $\mathrm{SO}(4)$-bundle are understood to be those of its associated oriented rank-four vector bundle $\xi_P = P\times_{\mathrm{SO}(4)}\mathbb{R}^4$; with this understanding and the identification $\mathrm{H}^4(\mathrm{S}^4;\mathbb{Z})\cong\mathbb{Z}$ given by $u$, the classes recorded in Section \ref{sec:milnorbundles} read $p_1(P_{m,n}) = 2(m-n)$ and $\mathrm{e}(P_{m,n}) = m+n$. In more generality, the following holds:

\begin{theorem}[{following Dold and Whitney; see \cite[\S 2]{Bouwknegt20152}}]\label{thm:so(4)principals}
Let $B$ be a simply-connected compact oriented 4-dimensional manifold, and let all characteristic classes of a principal $\mathrm{SO}(4)$-bundle $P\to B$ refer to those of the associated oriented rank-four vector bundle $\xi_P = P \times_{\mathrm{SO}(4)} \mathbb{R}^4$. Fix an integral lift $b$ of $w_2(P)$, which exists since $B$ is simply connected. Then $P$ is completely classified by its second Stiefel--Whitney class $w_2(P) \in \mathrm{H}^2(B; \mathbb{Z}_2)$ together with a pair of integers $m, n \in \mathbb{Z}$, determined by $w_2(P)$, the lift $b$, and the equations:
\begin{equation}
    \langle p_1(P),[B]\rangle = 2(m-n) + \beta,
\end{equation}
\begin{equation}
    \langle\mathrm{e}(P),B\rangle = m+n,
\end{equation}
where $\beta := \langle b \cup b, [B]\rangle \in \mathbb{Z}$ is the self-intersection of the chosen lift $b$. This $\beta$ is compatible with the mod-$4$ Pontryagin square in the sense that $\beta \equiv \langle \mathcal{P}(w_2(P)), [B]\rangle \pmod 4$, where $\mathcal{P}: \mathrm{H}^2(B; \mathbb{Z}_2) \to \mathrm{H}^4(B; \mathbb{Z}_4)$ is the Pontryagin square; in particular, $\beta\bmod 4$, and hence the residue class of $p_1(P)+2\mathrm{e}(P)$, is independent of the chosen lift. 
\end{theorem}

\begin{remark}
In \cite{CROWLEY2003363}, a classification of total spaces of bundles over the four-sphere with fiber a three-sphere is provided, considering both orientation-preserving and reversing homotopy equivalence, homeomorphism, and diffeomorphism. Care must be taken, however, as the definition of $M_{m,n}$ in that context differs slightly from ours. We refer the reader to Remark 1.1 in the cited reference.
\end{remark}

\vspace{1em} 

\subsection{Principal Spherical T-duality}

Here, we recall the definition of spherical T-duality in the principal bundle case, first introduced in \cite{Bouwknegt2015}. 

\begin{theorem}[Theorem 1 in \cite{Bouwknegt2015}]\label{thm:gysin}
    Let $\pi: P\to B$ be an $\mathrm{S}^3$-principal bundle. We have the following exact sequence, known as the Gysin sequence for cohomology with integer coefficients:
    \begin{equation} \label{eq:eqBaa}
       \xymatrix{
        \cdots \ar[r] & \mathrm{H}^{p-4}(B) \ar[r]^{\mathrm{c}_2 \cup} & \mathrm{H}^p(B) \ar[r]^{\pi^*} & \mathrm{H}^p(P) \ar[r]^{\pi_*} & \mathrm{H}^{p-3}(B) \ar[r]^{\mathrm{c}_2 \cup} & \cdots}
    \end{equation}
  where $\pi^*$ denotes the pull-back map, $\pi_*$ denotes the Gysin map (integration along the fiber), and $\mathrm{c}_2\cup$ denotes the cup product with the second Chern class $\mathrm{c}_2(P)\in \mathrm{H}^4(B; \mathbb{Z})$. Here we have identified the Euler class $\mathrm{e}(\xi_P) \in \mathrm{H}^4(B;\mathbb{Z})$ of the associated oriented real rank-four bundle $\xi_P = P\times_{\mathrm{S}^3} \mathbb{R}^4$, where $\mathrm{S}^3$ acts on $\mathbb{R}^4\cong\mathbb{H}$ by left multiplication, with the second Chern class $\mathrm{c}_2(V)$ of the complex rank-two bundle $V = P\times_{\mathrm{SU}(2)} \mathbb{C}^2$ obtained under $\mathrm{S}^3 \cong \mathrm{SU}(2)$; with compatible orientations, $\mathrm{e}(\xi_P) = \mathrm{c}_2(V)$. Equivalently, $V$ is the quaternionic line bundle $L = P\times_{\mathrm{S}^3}\mathbb{H}$ regarded as a complex bundle via the right $\mathbb{C}$-module structure on $\mathbb{H}$ (left multiplication by $\mathrm{S}^3$ being complex linear for it).
\end{theorem}

\begin{definition}[Principal Spherical T-duality]\label{def:principaltdual}
Given a pair $(\pi : P\rightarrow B,H)$ consisting of an $\mathrm{S}^3$-principal bundle $P$ over $B$ and a class $H\in \mathrm{H}^7(P; \mathbb{Z})$, called the \emph{$H$-flux}, a \emph{spherical T-dual} of $(P,H)$ is a pair $(\widehat\pi : \widehat{P}\rightarrow B,\widehat H)$ consisting of an $\mathrm{S}^3$-principal bundle $\widehat{P}$ over $B$ equipped with a class $\widehat H \in \mathrm{H}^7(\widehat{P}; \mathbb{Z})$ satisfying the relations:
\begin{align}
    \mathrm{c}_2(\widehat{P}) &= \pi_* H, \\
    \widehat\pi_*\widehat H &= \mathrm{c}_2(P),
\end{align}
and such that their pullbacks to the fiber product $P\times_{B}\widehat{P}$ match: $\widehat p^*H=p^*\widehat H$, where $\widehat p$ and $p$ denote the canonical projections onto $P$ and $\widehat{P}$, respectively. This fits into the commutative diagram:
\begin{equation} \label{correspondenceb}
\xymatrix 
@=4pc @ur 
{ P \ar[d]_{\pi} & 
P\times_{B}  \widehat{P} \ar[d]^{p} \ar[l]_-{\widehat p} \\{B} & \widehat{P}\ar[l]^{\widehat \pi}}
\end{equation}
\end{definition}

The connection between Theorem~\ref{thm:gysin} and Definition~\ref{def:principaltdual} is direct. Given a pair $(P, H)$ over $B$, integration along the fiber yields $\pi_* H \in \mathrm{H}^4(B; \mathbb{Z})$. The defining property of T-duality asks whether this class $\pi_* H$ can be realized as the second Chern class of a dual principal bundle $\widehat{P}$. As a first goal, we determine how spherical T-duality applies to those Milnor bundles which are $\mathrm{S}^3$-principal. We start with the following lemma. 

\begin{lemma}\label{lem:principal_s3_condition}
Let $P_{m,n}$ be the principal $\mathrm{SO}(4)$-bundle over $\mathrm{S}^4$ classified by the pair $(m,n) \in \mathbb{Z} \oplus \mathbb{Z}$, and let $M_{m,n} = P_{m,n} \times_{\mathrm{SO}(4)} \mathrm{S}^3$ be its associated $\mathrm{S}^3$-fiber bundle. The bundle $M_{m,n}$ admits the structure of an $\mathrm{S}^3$-principal bundle if and only if $m=0$ or $n=0$. This constraint is equivalent to the characteristic classes of the principal bundle satisfying $p_1(P_{m,n}) = \pm 2 \mathrm{e}(P_{m,n})$.
\end{lemma}
\begin{proof}
A principal $\mathrm{S}^3$-structure on $M_{m,n}$ identifies the fibers with the group in such a way that the clutching map acts by translations; hence $[t_{m,n}]$ lies in the image of $\pi_3(\mathrm{S}^3)\to\pi_3(\mathrm{SO}(4))$ induced by $g\mapsto L_g$ or by $g\mapsto R_g$, the left and right translations of $\mathbb{H}$. Since $L_g=t_{1,0}(g)$ and $R_g=t_{0,1}(g)$, these images are $\mathbb{Z}\,(1,0)$ and $\mathbb{Z}\,(0,1)$ respectively, so $n=0$ or $m=0$. Conversely, on $M_{n,0}$ the right translations $v\mapsto v\,g$ commute with the clutching $v\mapsto x^{n}v$ and act freely and transitively on the fibers, exhibiting $M_{n,0}$ as principal; symmetrically, the left translations exhibit $M_{0,n}$ as principal. The characteristic-class reformulation follows from the normalization of Section \ref{sec:milnorbundles}: with $p_1=2(m-n)$ and $\mathrm{e}=m+n$, one has $p_1=2\,\mathrm{e}$ if and only if $n=0$, and $p_1=-2\,\mathrm{e}$ if and only if $m=0$.
\end{proof}

Assume there exists an $\mathrm{S}^3$-principal bundle $\widehat{P}$ such that $\mathrm{c}_2(\widehat{P}) = \pi_* H$ (which may not be unique if $\dim(B) > 4$; see \cite[Section 1, p.910]{Bouwknegt2015}). The Gysin sequence for $\widehat \pi : \widehat{P}\to B$ ensures the existence of a class $\widehat H \in \mathrm{H}^7(\widehat{P};\mathbb{Z})$ such that $\widehat \pi_* \widehat H = \mathrm{c}_2(P)$, and that $\widehat H$ is determined by this condition up to an element $\widehat \pi^* h$, where $h\in \mathrm{H}^7(B;\mathbb{Z})$. To resolve the non-uniqueness of $\widehat H$, the constraint $p^*\widehat H - \widehat p^*H = 0 \in \mathrm{H}^7(P\times_{B}\widehat{P};\mathbb{Z})$ is imposed on the correspondence space $P\times_{B}\widehat{P}$, yielding: 
\begin{equation} \label{eq:eqBab}
\xymatrix{
& (P\times_B\widehat{P},\;p^*(\widehat H) - {\widehat p}^*(H) = 0) \ar[dl]_{\widehat  p}  \ar[dr]^{p} \\
(P,H) \ar[dr]_{\pi} && (\widehat{P},\widehat H) \ar[dl]^{\widehat \pi}  \\
& B & 
}
\end{equation}

\begin{theorem}[Theorem 2 in \cite{Bouwknegt2015}] \label{thm:thBAb}
Let $P$ be an $\mathrm{S}^3$-principal bundle with second Chern class $\mathrm{c}_2 \equiv \mathrm{c}_2(P) \in \mathrm{H}^4(B)$,
and let $H\in \mathrm{H}^7(P)$ be an $H$-flux on $P$. Suppose there exists an $\mathrm{S}^3$-principal bundle $\widehat{P}$ such that $\widehat{\mathrm{c}_2} \equiv \mathrm{c}_2(\widehat{P}) = \pi_*H$. Then
\begin{itemize}
\item[(i)] (Existence) there exists a class $\widehat H\in \mathrm{H}^7(\widehat{P})$ such that
\begin{equation} \label{eq:eqBa}
\widehat \pi_* \widehat H = \mathrm{c}_2(P) \,,\quad \text{and}\quad p^*\widehat H - \widehat p^*H = 0 \,,
\end{equation}
\item[(ii)] (Uniqueness) $\widehat H$ is uniquely determined by \eqref{eq:eqBa} up to the addition of a term
$\widehat \pi^*( a \cup \mathrm{c}_2)$, with $a\in \mathrm{H}^3(B)$.
\end{itemize}
\end{theorem}

\begin{remark}
If $\dim(B)\leq 6$, then $\mathrm{H}^7(B)=0$, so the indeterminacy $\widehat\pi^{*}(a \cup \mathrm{c}_2)$ of Theorem \ref{thm:thBAb}(ii) vanishes and diagram \eqref{correspondenceb} specifies the class $\widehat H$ uniquely \cite{Bouwknegt2015}. This is the case throughout the present work, where $B=\mathrm{S}^4$.
\end{remark}

From Lemma \ref{lem:principal_s3_condition}, the only Milnor bundles admitting the structure of a principal $\mathrm{S}^3$-bundle are $M_{m,0}$, principal for right translations, and $M_{0,n}$, principal for left translations. We fix once and for all the compatible orientations of Theorem \ref{thm:gysin} on the family $M_{m,0}$, so that \[ \mathrm{c}_2(M_{m,0})=\mathrm{e}(M_{m,0})=m, \] upon identifying $\mathrm{H}^4(\mathrm{S}^4;\mathbb{Z})\cong\mathbb{Z}$ via the orientation generator $u$. The fiberwise quaternionic conjugation $\kappa(x,v)=(x,\bar v)$ satisfies $\kappa\circ f_{0,-m}=f_{m,0}\circ\kappa$, since $\overline{v\,x^{-m}}=x^{m}\bar{v}$, and therefore descends to a diffeomorphism $M_{0,-m}\to M_{m,0}$ over $\mathrm{S}^4$. It intertwines the two principal structures: conjugation is an anti-automorphism of $\mathbb{H}$, $\overline{v\,g}=\bar g\,\bar v$, so it carries the right action $v\mapsto v\,g$ of $M_{0,-m}$ to the left action $\bar v\mapsto \bar g\,\bar v$ of $M_{m,0}$, that is, to the same principal action read through the group anti-automorphism $g\mapsto g^{-1}$. Hence $M_{0,-m}\cong M_{m,0}$ as principal $\mathrm{S}^3$-bundles. Since $\mathrm{c}_2$ is an invariant of the principal bundle, $\mathrm{c}_2(M_{0,n})=\mathrm{c}_2(M_{-n,0})=-n$; note that $\kappa$ reverses the fiber orientation, consistently with $\mathrm{e}(M_{0,n})=n$, so that $\mathrm{c}_2=-\mathrm{e}$ on this family. Throughout the principal theory (Definition \ref{def:principaltdual}, Theorem \ref{thm:tdualmilnor}, and Theorem \ref{thm:equivalence}), every principal Milnor bundle is understood to carry the compatible orientation of Theorem \ref{thm:gysin}, namely the one for which the Euler class of the associated rank-four bundle equals $\mathrm{c}_2$: on the family $M_{m,0}$ this agrees with the orientation fixed in Section \ref{sec:milnorbundles}, while on the family $M_{0,n}$ it is the opposite one, and it accordingly reverses the sign of the Gysin map on that family. With these orientations, $\kappa$ preserves the fiber orientations, so the isomorphism $M_{0,-m}\cong M_{m,0}$ preserves all the normalizations below. Finally, for an oriented $\mathrm{S}^3$-bundle $\pi:E\to\mathrm{S}^4$ we write $[a]\in\mathrm{H}^7(E;\mathbb{Z})$ for the unique class with $\pi_*[a]=a\,u$, the Gysin map being taken with respect to the orientation just specified when $E$ is principal; this is well defined because $\pi_*$ is an isomorphism in this degree, as $\mathrm{H}^7(\mathrm{S}^4)=\mathrm{H}^8(\mathrm{S}^4)=0$ in the Gysin sequence.

\begin{theorem} \label{thm:tdualmilnor}
Let $M_{m,0}$ be any principal Milnor bundle, equipped with the $H$-flux $H=[j]\in \mathrm{H}^7(M_{m,0};\mathbb Z)\cong \mathrm{H}^4(\mathrm{S}^4;\mathbb Z)\cong \mathbb{Z}$. Then we have the following spherical T-duality diagram
 \begin{equation} \label{eq:principalmilnordual}
\xymatrix{
& M_{m,0}\times_{\mathrm{S}^4}M_{0,-j} \ar[dl]_{\widehat  p}  \ar[dr]^{p} \\
(M_{m,0},[j]) \ar[dr]_{\pi} && (M_{0,-j},[m]) \ar[dl]^{\widehat \pi}  \\
& \mathrm{S}^4 & 
}
\end{equation}

In particular, $(M_{m,0},[j])$ and $(M_{0,-j},[m])$ are spherical T-dual in the sense of Definition \ref{def:principaltdual}. Moreover, 
\begin{align}
\mathrm{H}^0(M_{m,0})&\cong\mathrm{H}^0(M_{0,-j})\cong\mathrm{H}^7(M_{m,0})\cong\mathrm{H}^7(M_{0,-j})\cong\mathbb{Z}\\
\mathrm{H}^4(M_{m,0})&\cong\mathbb{Z}_m\\\mathrm{H}^4(M_{0,-j})&\cong\mathbb{Z}_j\\
    \mathrm{H}^7(M_{m,0}\times_{\mathrm{S}^4}M_{0,-j})&\cong\mathbb{Z}\oplus\mathbb{Z}_{\gcd(j,m)},
\end{align}
where throughout $\mathbb{Z}_a:=\mathbb{Z}/|a|\,\mathbb{Z}$, with the convention $\mathbb{Z}_0:=\mathbb{Z}$.
\end{theorem}
\begin{proof}
By Lemma \ref{lem:principal_s3_condition} and the discussion preceding the statement, the principal Milnor bundles are exactly those of the form $M_{n,0}$, with clutching by left multiplications, or $M_{0,n}$, with clutching by right multiplications:
\begin{equation}
t_{n,0}(x)v := x^{n}v \quad\text{and}\quad t_{0,n}(x)v := v\,x^{n},\qquad x\in \mathrm{S}^3,~v\in \mathbb H,
\end{equation}
so that
\[M_{n,0} = \mathrm{D}^4_+\times \mathrm{S}^3\cup_{f_{n,0}}\mathrm{D}^4_-\times \mathrm{S}^3,\qquad
M_{0,n} = \mathrm{D}^4_+\times \mathrm{S}^3\cup_{f_{0,n}}\mathrm{D}^4_-\times \mathrm{S}^3.\]
By the convention fixed above, the second Chern classes are $\mathrm{c}_2(M_{n,0})=n$ and $\mathrm{c}_2(M_{0,n})=-n$. Given as initial data $(M_{m,0},[j])$ with $[j]\in \mathrm{H}^7(M_{m,0};\mathbb Z)\cong \mathrm{H}^4(\mathrm{S}^4;\mathbb Z)$, the dual $\widehat{M_{m,0}}$ must satisfy $\mathrm{c}_2(\widehat{M_{m,0}})=\pi_*H=[j]$. Realizing this class within the family $M_{0,n}$ via $\mathrm{c}_2(M_{0,n})=-n=j$, i.e. $n=-j$, we obtain
\begin{equation}
   \widehat{M_{m,0}} = \mathrm{D}^4_+\times \mathrm{S}^3\cup_{f_{0,-j}} \mathrm{D}^4_-\times \mathrm{S}^3 = M_{0,-j},
\end{equation}
which indeed has $\mathrm{c}_2(M_{0,-j})=-(-j)=j$.

Now, observe that the Gysin sequence (Theorem \ref{thm:gysin}) provides both the cohomology of $M_{m,0}$ and $\widehat{M_{m,0}}$:
\begin{align}
\mathrm{H}^0(M_{m,0})&\cong\mathrm{H}^0(\widehat{M_{m,0}})\cong\mathrm{H}^7(M_{m,0})\cong\mathrm{H}^7(\widehat{M_{m,0}})\cong\mathbb{Z}\\
\mathrm{H}^4(M_{m,0})&\cong\mathbb{Z}_m\\\mathrm{H}^4(\widehat{M_{m,0}})&\cong\mathbb{Z}_j.
\end{align} 
Indeed, the cohomology $\mathrm{H}^4(M_{m,0})$ comes from
\begin{equation}\label{eq:0e4}
\xymatrix{
 \mathrm{H}^0(\mathrm{S}^4) \ar[d]^\cong \ar[r]^{\cup \mathrm{c}_2}& \mathrm{H}^4(\mathrm{S}^4)\ar[d]^\cong\ar[r]^{\pi^*} & \mathrm{H}^4(M_{m,0})\ar[d]^\cong\ar[r]^{\pi_*} & \mathrm{H}^1(\mathrm{S}^4) \ar[d]^\cong \\
 \mathbb{Z} \ar[r]^{\times m} & \mathbb{Z} \ar[r] & \ast \ar[r] & 0 
}
\end{equation}  
which implies that $\pi^*:\mathrm{H}^4(\mathrm{S}^4)\rightarrow\mathrm{H}^4(M_{m,0})$ has kernel $m\,\mathrm{H}^4(\mathrm{S}^4)$ and is surjective, so $\mathrm{H}^4(M_{m,0})\cong \mathbb{Z}_m$. The same computation,\emph{mutatis-mutandis}, verifies the claim for $\widehat{M_{m,0}}$.

Looking at
\begin{equation}\label{eq:3e7}
\xymatrix{
\mathrm{H}^7(\mathrm{S}^4) \ar[d]^\cong  \ar[r]^{\pi^*} & \mathrm{H}^7(M_{m,0}) \ar[d]^{\cong}\ar[r]^{\pi_*} & \mathrm{H}^{4}(\mathrm{S}^4) \ar[d]^\cong\ar[r]^{\cup \mathrm{c}_2} & \mathrm{H}^8(\mathrm{S}^4) \ar[d]^\cong \\
0 \ar[r] & \mathbb{Z}  \ar[r]^{\cong} & \mathbb{Z}  \ar[r] & 0
}
\end{equation}  
one derives that $\pi_*:\mathrm{H}^7(M_{m,0})\rightarrow\mathrm{H}^4(\mathrm{S}^4)$ is an isomorphism. Similarly, $\widehat\pi_*:\mathrm{H}^7(\widehat{M_{m,0}})\rightarrow\mathrm{H}^4(\mathrm{S}^4)$ is an isomorphism.
In this way, the flux $H\in\mathrm{H}^7(M_{m,0})$ is the unique class with $\pi_* H=\widehat{\mathrm{c}_2}=[j]$, while $\widehat H\in\mathrm{H}^7(\widehat{M_{m,0}})$ is the unique class with $\widehat\pi_*\widehat H=\mathrm{c}_2(M_{m,0})=[m]$. Therefore, $\widehat H = [m]$. This concludes the first part of the result.

It remains to verify the third condition of Definition \ref{def:principaltdual},
\begin{equation}
p^*(\widehat H) - \widehat p^*(H) = 0 \in\mathrm{H}^7(M_{m,0}\times_{\mathrm{S}^4} \widehat{M_{m,0}})  \label{corrcond}
\end{equation}
For the remainder of the argument we assume $m,j\neq 0$; since every group involved depends only on $|m|$ and $|j|$, we may and do further assume $m,j>0$. (If $m=0$, then $M_{0,0}=\mathrm{S}^4\times\mathrm{S}^3$ and $\mathrm{H}^3(M_{0,0})\cong\mathbb{Z}$, so the first Gysin subsequence below degenerates to $\mathbb{Z}\xrightarrow{\times j}\mathbb{Z}\xrightarrow{\widehat p^{\,*}}\mathrm{H}^7(\,\cdot\,)\rightarrow\mathbb{Z}\rightarrow 0$, yielding $\mathbb{Z}\oplus\mathbb{Z}_{|j|}$, in agreement with $\gcd(j,0)=|j|$; the case $j=0$ is symmetric and yields $\mathbb{Z}\oplus\mathbb{Z}_{|m|}$, while $m=j=0$ gives $\mathbb{Z}\oplus\mathbb{Z}$ by the K\"unneth formula, consistently with $\mathbb{Z}_0=\mathbb{Z}$. In each degenerate case the pullback equality is immediate: the flux carried by the non-trivial side vanishes, being $[0]$, and the pullback of the flux carried by the trivial side is $j$ times, respectively $m$ times, a class of order dividing $j$, respectively $m$, hence zero.) With $m,j>0$, the desired cohomology group appears in two short exact Gysin subsequences
\begin{equation}
\xymatrix{
\mathrm{H}^3(M_{m,0})\ar[d]^\cong \ar[r]^{\cup j} & \mathrm{H}^7(M_{m,0})\ar[d]^\cong \ar[r]^{\widehat p^*\quad} & \mathrm{H}^7(M_{m,0}\times_{\mathrm{S}^4} \widehat{M_{m,0}}) \ar[d] \ar[r]^{\quad\widehat p_*} & 
\mathrm{H}^{4}(M_{m,0})\ar[d]^\cong\ar[r]^{\cup j} & \mathrm{H}^8(M_{m,0}) \ar[d]^\cong \\
0 \ar[r]  & \mathbb{Z} \ar[r] & \ast \ar[r] & \mathbb{Z}_m  \ar[r] & 0
} \label{pseq}
\end{equation}  
and
\begin{equation}
\xymatrix{
\mathrm{H}^3(\widehat{M_{m,0}}) \ar[d]^\cong \ar[r]^{\cup m} & \mathrm{H}^7(\widehat{M_{m,0}})\ar[d]^\cong \ar[r]^{p^*\quad} & \mathrm{H}^7(M_{m,0}\times_{\mathrm{S}^4} \widehat{M_{m,0}}) \ar[d] 
\ar[r]^{\quad p_*} & \mathrm{H}^{4}(\widehat{M_{m,0}})\ar[d]^\cong\ar[r]^{\cup m}& \mathrm{H}^8(\widehat{M_{m,0}}) \ar[d]^\cong \\
0 \ar[r]  & \mathbb{Z} \ar[r] & \ast \ar[r] & \mathbb{Z}_j  \ar[r] & 0
}
\end{equation}  
The first sequence exhibits $\mathrm{H}^7(M_{m,0}\times_{\mathrm{S}^4} \widehat{M_{m,0}})$ as an extension of $\mathbb{Z}_m$ by $\mathbb{Z}$, and the second as an extension of $\mathbb{Z}_j$ by $\mathbb{Z}$. In either case the torsion subgroup meets the free kernel trivially and so injects into the cyclic quotient; hence
\begin{equation}
\mathrm{H}^7(M_{m,0}\times_{\mathrm{S}^4}\widehat{M_{m,0}})\cong\mathbb{Z}\oplus\mathbb{Z}_i\qquad\text{for some }i\geq 1\text{ with } i\mid m,\ i\mid j.
\end{equation}
Let $(a,b)$, with $a\in\mathbb{Z}$ and $b\in\mathbb{Z}_i$, be the image of $1\in \mathbb{Z}$ under $\widehat p^*$:
\begin{equation}
\xymatrix{
\mathrm{H}^7(M_{m,0}) \ar[d]^\cong \ar[r]^{\widehat p^*\quad} & \mathrm{H}^7(M_{m,0}\times_{\mathrm{S}^4} \widehat{M_{m,0}}) \ar[d]^\cong \\
\mathbb{Z} \ar[r] & \mathbb{Z}\oplus\mathbb{Z}_i \\
1 \ar[r] & (a,b). }  
\end{equation}

By the exactness of \eqref{pseq}, the map $\widehat p_*:\mathrm{H}^7(M_{m,0}\times_{\mathrm{S}^4} \widehat{M_{m,0}})\rightarrow\mathrm{H}^4(M_{m,0})=\mathbb{Z}_m$ is surjective. Therefore, we may compute the quotient via the Smith normal form of the relation matrix $\left(\begin{smallmatrix} a & b\\ 0 & i\end{smallmatrix}\right)$:
\begin{equation}
\mathbb{Z}_m \cong \mathrm{Im}(\widehat p_*) \cong \frac{\mathrm{H}^7(M_{m,0}\times_{\mathrm{S}^4} \widehat{M_{m,0}})}{\ker(\widehat p_*)} \cong \frac{\mathrm{H}^7(M_{m,0}\times_{\mathrm{S}^4} \widehat{M_{m,0}})}{\mathrm{Im}(\widehat p^*)} \cong \frac{\mathbb{Z}\oplus\mathbb{Z}_i}{\langle(a,b)\rangle} \cong \mathbb{Z}_{\frac{ai}{\gcd(a,b,i)}}\oplus\mathbb{Z}_{\gcd(a,b,i)}. 
\end{equation}
Comparing orders on the two sides gives $|a|\,i=m$, so that, after replacing the chosen generator of the free summand by its negative if necessary, $a=m/i$. For the quotient to be isomorphic to the single cyclic group $\mathbb{Z}_m$, the first invariant factor must be trivial, that is, $\gcd(a, b, i) = \gcd(m/i, b, i) = 1$. Similarly, applying this to $\widehat{M_{m,0}}$ yields:
\begin{equation}
p^*:\mathrm{H}^7(\widehat{M_{m,0}})\cong\mathbb{Z}\rightarrow\mathrm{H}^7(M_{m,0}\times_{\mathrm{S}^4} \widehat{M_{m,0}})\cong\mathbb{Z}\oplus\mathbb{Z}_i:1\mapsto(\widehat a,\widehat b),\qquad \widehat a=\pm\,j/i \,, 
\end{equation}
with the analogous constraint $\gcd(\widehat a, \widehat b, i)=1$. Since the generator of the free summand has already been fixed by the normalization $a=m/i$, the sign of $\widehat a$ is no longer a matter of choice, and we determine it now. In the Serre spectral sequence of the fibration $\mathrm{S}^3\times\mathrm{S}^3\to M_{m,0}\times_{\mathrm{S}^4}\widehat{M_{m,0}}\to\mathrm{S}^4$, the two fiber generators $x,y\in E_2^{0,3}$, taken with the compatible orientations fixed above, transgress to the respective Euler classes, $d_4(x)=\mathrm{c}_2(M_{m,0})\,u=m\,u$ and $d_4(y)=\mathrm{c}_2(\widehat{M_{m,0}})\,u=j\,u$, while $\widehat p^{\,*}(1)$ and $p^{*}(1)$ are detected in $E_\infty^{4,3}$ by $u\otimes x$ and $u\otimes y$, respectively; the detections are equalities because the filtration of $\mathrm{H}^7$ has a single nonzero quotient (the same mechanism is used, with full details, in the proof of Theorem \ref{thm:nonprincipaltdualmilnor} below). The Leibniz rule gives $d_4(xy)=m\,(u\otimes y)-j\,(u\otimes x)$, and this is the only relation on the total-degree-seven line; hence
\begin{equation}
j\,\widehat p^{\,*}(1)\;=\;m\,p^{*}(1)\qquad\text{in }\mathrm{H}^7(M_{m,0}\times_{\mathrm{S}^4}\widehat{M_{m,0}}).
\end{equation}
Comparing the free components, $j\,(m/i)=m\,\widehat a$, whence $\widehat a=+\,j/i$.
 
Finally, we use the push--pull identity associated with the Cartesian square \eqref{correspondenceb} (base change for integration along the fiber), which holds for any fiber product of oriented sphere bundles and is independent of the duality:
\begin{equation}
\widehat\pi^*\pi_*=p_*\widehat p^*:\mathrm{H}^7(M_{m,0})\cong\mathbb{Z}\rightarrow\mathrm{H}^4(\widehat{M_{m,0}})\cong\mathbb{Z}_j. \label{comm}
\end{equation}
Because $\pi_*:\mathrm{H}^7(M_{m,0})\rightarrow\mathrm{H}^4(\mathrm{S}^4)$ is an isomorphism and $\widehat\pi^*:\mathrm{H}^4(\mathrm{S}^4)\rightarrow\mathrm{H}^4(\widehat{M_{m,0}})$ is surjective, $\widehat\pi^*\pi_*$ maps the generator $1 \in \mathrm{H}^7(M_{m,0})$ to an element of order $j$ in $\mathrm{H}^4(\widehat{M_{m,0}})\cong\mathbb{Z}_j$. On the other hand, we have established that:
\begin{equation}
\widehat p^*(1)=(m/i, b)\in\mathrm{H}^7(M_{m,0}\times_{\mathrm{S}^4}\widehat{M_{m,0}})\cong\mathbb{Z}\oplus\mathbb{Z}_i. 
\end{equation}
The kernel of the map $p_*:\mathrm{H}^7(M_{m,0}\times_{\mathrm{S}^4}\widehat{M_{m,0}})\rightarrow\mathrm{H}^4(\widehat{M_{m,0}})$ is identically the image of $p^*:\mathrm{H}^7(\widehat{M_{m,0}})\rightarrow\mathrm{H}^7(M_{m,0}\times_{\mathrm{S}^4}\widehat{M_{m,0}})$, which is generated by $(j/i,\widehat b)$. Therefore, $p_*\widehat p^*(1)$ possesses an order of $j$ precisely if $\widehat p^*(1)=(m/i,b)$ has order $j$ in the quotient space $(\mathbb{Z}\oplus\mathbb{Z}_i)/\langle (j/i,\widehat b) \rangle$.

The order of $(m/i,b)$ in this quotient divides $j$: since $i$ divides $j$, we have $jb=0$ in $\mathbb{Z}_i$, so
\begin{equation}
j\left(\frac{m}{i},b\right) = \left(\frac{jm}{i},\, 0\right)
= m\left(\frac{j}{i},\widehat b\right)\in\ker(p_*),
\end{equation}
the last equality holding because $i$ divides $m$ as well, whence $m\widehat b=0$ in $\mathbb{Z}_i$. Let $d:=\gcd(j,m)$. Because $i$ is a common divisor of $j$ and $m$, the ratio $d/i$ is an integer, and $ji/d=(j/d)\,i$ is an integer multiple of $i$. Therefore $(ji/d)\,b=0$ in $\mathbb{Z}_i$ and, likewise, $(mi/d)\,\widehat b=0$, so
\begin{equation}
\frac{ji}{d}\left(\frac{m}{i},b\right) = \left(\frac{jm}{d},\, 0\right)
= \frac{mi}{d}\left(\frac{j}{i},\widehat b\right)\in\ker(p_*).
\end{equation}
Hence $\widehat p^{\,*}(1)=(m/i,b)$ has order dividing $ji/d$ in the quotient. To satisfy \eqref{comm}, its order must be exactly $j$; thus $j$ divides $ji/d$, that is, $d$ divides $i$. Combined with $i\mid d$, this forces $i=d=\gcd(j,m)$. Thus:
\begin{equation}\label{eq:sete}
    \mathrm{H}^7(M_{m,0}\times_{\mathrm{S}^4}\widehat{M_{m,0}})\cong\mathbb{Z}\oplus\mathbb{Z}_{\gcd(j,m)} 
\end{equation}
\begin{equation}
a = \frac{m}{d}, \quad \text{and} \quad \widehat a = \frac{j}{d},\qquad d=\gcd(j,m).
\end{equation}

As the final step, we compute the pullbacks of the two fluxes to the correspondence space. Since $d$ divides both $j$ and $m$, we have $j\,b=0$ and $m\,\widehat b=0$ in $\mathbb{Z}_d$, so the torsion components of both pullbacks vanish:
\begin{equation}
\widehat p^{\,*} H = \widehat p^{\,*}(j \cdot 1) = j\,(a, b) = \left(\frac{jm}{d},\; j\,b\right) = \left(\frac{jm}{d},\; 0\right),
\end{equation}
\begin{equation}
p^{*}\widehat H = p^{*}(m \cdot 1) = m\,(\widehat a, \widehat b) = \left(\frac{jm}{d},\; m\,\widehat b\right) = \left(\frac{jm}{d},\; 0\right).
\end{equation}
Therefore $\widehat p^{\,*} H = p^{*}\widehat H$: the class $\widehat H$, determined by $\widehat\pi_*\widehat H=\mathrm{c}_2(M_{m,0})$, agrees with $H$ once both are pulled back to the correspondence space. This completes the verification of Definition \ref{def:principaltdual} for the pair $(M_{m,0},[j])$, $(M_{0,-j},[m])$.
\end{proof}

\vspace{1em}

\subsection{Non-principal Spherical T-duality}
\label{sec:nonprincipal}

Following \cite{Bouwknegt20152}, we now discuss spherical T-duality for oriented linear $\mathrm{S}^3$-bundles, not necessarily principal. The formalism below is based on Euler classes rather than second Chern classes and subsumes the principal case of the previous subsection, the two being reconciled by the dichotomy $\mathrm{c}_2 = \pm\,\mathrm{e}$ recorded there.

\begin{convention}[Euler class of an oriented linear $\mathrm{S}^3$-bundle]\label{conv:euler}
A principal $\mathrm{SO}(4)$-bundle $P\to B$ carries no Euler class of its own; the Euler class is an invariant of an \emph{oriented vector (or sphere) bundle}. Throughout, the \emph{Euler class of an oriented linear $\mathrm{S}^3$-bundle} $E\to B$, principal or not, means the Euler class of its associated oriented rank-four vector bundle,
\[
\mathrm{e}(E) := \mathrm{e}(\xi_P)\in\mathrm{H}^4(B;\mathbb{Z}),\qquad \xi_P = P\times_{\mathrm{SO}(4)}\mathbb{R}^4,
\]
where $E=S(\xi_P)=P\times_{\mathrm{SO}(4)}\mathrm{S}^3$ is the unit-sphere bundle of $\xi_P$. The three objects $P$, $\xi_P$, and $E$ share the same classifying data, and for a Milnor bundle $E=M_{m,n}$ over $\mathrm{S}^4$ this number is $\mathrm{e}(M_{m,n})=(m+n)u$ (Section \ref{sec:milnorbundles}). Thus, when we speak below of $E$ and $\widehat E$ ``having the same Euler class,'' we mean an equality of these associated vector-bundle Euler classes in $\mathrm{H}^4(B;\mathbb{Z})$, not a comparison between principal bundles. In the principal case this is consistent with Section \ref{sec:milnorbundles}: on the family $M_{m,0}$ one has $\mathrm{c}_2=\mathrm{e}$, while on the family $M_{0,n}$ one has $\mathrm{c}_2=-\mathrm{e}$. This is the same identification used in the principal case (Theorem \ref{thm:gysin}), where $\mathrm{e}(\xi_P)$ is matched with the second Chern class of the associated complex rank-two bundle.
\end{convention}

Given an oriented $\mathrm{S}^3$-non-principal bundle $E \xrightarrow{\pi} B$, there exists an analogous Gysin sequence for singular cohomology \cite[Proposition 14.33]{bott1982differential}:
\begin{equation}\label{gysin}
\cdots \to \mathrm{H}^p(B; \mathbb{Z}) \xrightarrow{\cup \mathrm{e}(E)} \mathrm{H}^{p+4}(B; \mathbb{Z})  
\xrightarrow{\pi^*} \mathrm{H}^{p+4}(E; \mathbb{Z}) \xrightarrow{\pi_{\ast}}  
\mathrm{H}^{p+1}(B; \mathbb{Z}) \xrightarrow{\cup \mathrm{e}(E)} \cdots
\end{equation}
Motivated by this, we adopt the following definition, which is Definition \ref{def:principaltdual} with the second Chern class replaced by the Euler class of Convention \ref{conv:euler}.  \begin{definition}[Spherical T-duality, oriented linear case]\label{def:nonprincipaltdual} Let $\pi:E\to B$ and $\widehat\pi:\widehat E\to B$ be oriented linear $\mathrm{S}^3$-bundles, and let $H\in\mathrm{H}^7(E;\mathbb{Z})$ and $\widehat H\in\mathrm{H}^7(\widehat E;\mathbb{Z})$. The pairs $(E,H)$ and $(\widehat E,\widehat H)$ are \emph{spherical T-dual} if \begin{align} \pi_{\ast}H &= \mathrm{e}(\widehat E), \\ \widehat\pi_{\ast}\widehat H &= \mathrm{e}(E), \end{align} and their pullbacks to the correspondence space agree, \[ \widehat p^{\,*}H = p^{*}\widehat H \ \ \text{ in } \ \mathrm{H}^7\big(E\times_{B}\widehat E;\mathbb{Z}\big), \] where $\widehat p$ and $p$ denote the canonical projections of $E\times_B\widehat E$ onto $E$ and $\widehat E$, respectively. When $E$ and $\widehat E$ are principal, this reduces to Definition \ref{def:principaltdual} under the identification of Theorem \ref{thm:gysin}. \end{definition} 

\begin{theorem}\label{thm:nonprincipaltdualmilnor}
Let $m,k\in\mathbb{Z}$ and let $M_{m,k-m}$ be the Milnor bundle of Euler class $\mathrm{e}=ku$, equipped with the $H$-flux $H = [k] \in \mathrm{H}^7(M_{m,k-m}; \mathbb{Z}) \cong \mathbb{Z}$, whose fiber integral $\pi_*H = ku$ equals its Euler class. Then, for every integer $j \in \mathbb{Z}$, the bundle $M_{j,k-j}$, which also has Euler class $ku$, equipped with the flux $\widehat{H} = [k]$, is a spherical T-dual of $M_{m,k-m}$ in the sense of Definition \ref{def:nonprincipaltdual}. The duality fits into the diagram:
\begin{equation} \label{eqn:nonprincipalmilnordual}
\xymatrix{
& M_{m,k-m} \times_{\mathrm{S}^4} M_{j,k-j} \ar[dl]_{\widehat{p}}  \ar[dr]^{p} \\
(M_{m,k-m}, [k]) \ar[dr]_{\pi} && (M_{j,k-j}, [k]) \ar[dl]^{\widehat{\pi}}  \\
& \mathrm{S}^4 & 
}
\end{equation}
If $k = 1$, then both $M_{m,1-m}$ and $M_{j,1-j}$ are homeomorphic to $\mathrm{S}^7$. Each of them fails to be diffeomorphic to the standard sphere whenever its $\lambda$-invariant is nonzero, that is, whenever
\[
(-1 + 2m)^2 \not\equiv 1 \pmod{7}, \qquad\text{respectively}\qquad (-1 + 2j)^2 \not\equiv 1 \pmod{7};
\]
the vanishing of $\lambda$, being a $\mathbb{Z}_7$-valued invariant of a class in $\Theta^7\cong\mathbb{Z}_{28}$, does not by itself imply that the sphere is standard. For $m=1$ the bundle $M_{1,0}\to\mathrm{S}^4$ is the Hopf bundle, with total space the standard $\mathrm{S}^7$; for $m=2$ one has $\lambda=1\neq 0$, and $M_{2,-1}$ is the Gromoll--Meyer sphere. Consequently, spherical T-duality relates homotopy $7$-spheres carrying distinct smooth structures and, in particular, relates the standard sphere to an exotic one.
\end{theorem}

\begin{proof}
Let $E = M_{m,k-m}$ and $\widehat{E} = M_{j,k-j}$. In the sense of Convention \ref{conv:euler}, both have the same Euler class, $\mathrm{e}(E) = (m+(k-m))u = ku = (j+(k-j))u = \mathrm{e}(\widehat{E})$, where $u \in \mathrm{H}^4(\mathrm{S}^4; \mathbb{Z})$ is the orientation generator. No use is made below of whether either bundle is principal.

Applying the Gysin sequence \eqref{gysin} with $p=3$ over $B=\mathrm{S}^4$ gives the exact segment \[ \mathrm{H}^{7}(\mathrm{S}^4;\mathbb{Z}) \xrightarrow{\ \pi^*\ } \mathrm{H}^{7}(E;\mathbb{Z}) \xrightarrow{\ \pi_*\ } \mathrm{H}^{4}(\mathrm{S}^4;\mathbb{Z}) \xrightarrow{\ \cup\,\mathrm{e}(E)\ } \mathrm{H}^{8}(\mathrm{S}^4;\mathbb{Z}). \] The source of $\pi^*$ and the target of $\cup\,\mathrm{e}(E)$ both vanish, so $\pi_*$ is injective and surjective; integration along the fiber thus yields an isomorphism $\pi_*: \mathrm{H}^7(E; \mathbb{Z}) \xrightarrow{\cong} \mathrm{H}^4(\mathrm{S}^4; \mathbb{Z}) \cong \mathbb{Z}$. Let $\eta_E \in \mathrm{H}^7(E; \mathbb{Z})$ be the normalized generator such that $\pi_* \eta_E = u$, and similarly define $\eta_{\widehat{E}}$ for $\widehat{E}$. The assigned fluxes are $H = k \eta_E$ and $\widehat{H} = k \eta_{\widehat{E}}$.

We first verify the two fiber-integration conditions of Definition \ref{def:nonprincipaltdual}:
\[
\pi_* H = \pi_*(k \eta_E) = k u = \mathrm{e}(\widehat{E}), \quad \text{and} \quad \widehat{\pi}_* \widehat{H} = \widehat{\pi}_*(k \eta_{\widehat{E}}) = k u = \mathrm{e}(E).
\]
It remains to verify the correspondence space condition $\widehat{p}^*H = p^*\widehat{H}$ on the fiber product $C = E \times_{\mathrm{S}^4} \widehat{E}$. To do so, we analyze $\mathrm{H}^7(C; \mathbb{Z})$ using the Serre spectral sequence for the fibration $\mathrm{S}^3 \times \mathrm{S}^3 \to C \to \mathrm{S}^4$. 

The relevant generators of the $E_2$ page are $u \in E_2^{4,0} \cong \mathrm{H}^4(\mathrm{S}^4)$ and the fiber generators $x, y \in E_2^{0,3} \cong \mathrm{H}^3(\mathrm{S}^3 \times \mathrm{S}^3)$, pulled back from the fibers of $E$ and of $\widehat{E}$ respectively under the two projections. Since $C\to\mathrm{S}^4$ is the fiber product, the projections $\widehat p$ and $p$ are maps of fibrations over $\mathrm{S}^4$, and naturality of the spectral sequence identifies the transgressions of $x$ and $y$ with those of the fiber generators of $E$ and $\widehat E$; these are the respective Euler classes, so \[ d_4(x) = \mathrm{e}(E) = k u,\qquad d_4(y) = \mathrm{e}(\widehat E) = k u . \]

On the total-degree-seven line, the only possibly nonzero bidegree is $(4,3)$, since $\mathrm{H}^{p}(\mathrm{S}^4)=0$ for $p\neq 0,4$ and $\mathrm{H}^{7}(\mathrm{S}^3\times\mathrm{S}^3)=0$, the fiber being six-dimensional. Hence $\mathrm{H}^7(C;\mathbb{Z})=E_\infty^{4,3}$, and $E_4^{4,3}$ is generated by $u \otimes x$ and $u \otimes y$. The only relation comes from the differential $d_4: E_4^{0,6} \to E_4^{4,3}$ acting on the volume form $xy$ of the fiber. By the Leibniz rule:
\[
d_4(xy) = d_4(x)y - x d_4(y) = k u y - k u x.
\]
Since the differentials out of position $(4,3)$ vanish (their targets $E^{8,0}$ and beyond are zero), $E_\infty^{4,3} = E_4^{4,3}/\mathrm{Im}(d_4)$, and the image of $d_4(xy) = kuy - kux$ imposes precisely the relation $k(u \otimes x) = k(u \otimes y)$. 
It remains to identify the pullbacks of the fluxes with these classes. The projection $\widehat p: C \to E$ is a map of fibrations over $\mathrm{S}^4$, covering the identity of the base and restricting on fibers to the projection $\mathrm{S}^3\times\mathrm{S}^3\to\mathrm{S}^3$ onto the first factor; by naturality of the Serre spectral sequence, the induced map on $E_2$-pages sends the fiber generator of $\mathrm{H}^3(\mathrm{S}^3)$ to $x$, and hence sends the class $\eta_E$, which in the spectral sequence of $E$ is detected in $E_\infty^{4,3}$ by $u\otimes(\text{fiber generator})$, to the class detected by $u \otimes x$. Symmetrically, $p^*\eta_{\widehat{E}}$ is detected by $u \otimes y$. Since in both cases the filtration of $\mathrm{H}^7$ has a single nonzero quotient, these detections are equalities, and the relation established above gives
\[
\widehat{p}^* H = \widehat{p}^*(k \eta_E) = k(u \otimes x) = k(u \otimes y) = p^*(k \eta_{\widehat{E}}) = p^* \widehat{H},
\]
which is the required pullback equality. Together with the two fiber-integration conditions verified above, this shows that $(E,[k])$ and $(\widehat E,[k])$ are spherical T-dual in the sense of Definition \ref{def:nonprincipaltdual}.
\end{proof}

\vspace{1em}

\section{Higher-Dimensional Logarithmic Transformations, Exotic Hopf Manifolds and Spherical T-duality}
\label{sec:logtrans_final}

In this section, we extend the concept of logarithmic transformations \cite{gompfmrowka, GOMPF1991479} to a larger class of manifolds. We are interested in their topological descriptions rather than their algebraic or complex ones. To motivate our proposal for this \emph{generalized logarithmic transformation}, we first study in detail, following \cite{ZENTNER200637}, the logarithmic transformations performed on two regular fibers of the elliptic fibration associated with the Hopf surface $\mathrm{S}^1\times\mathrm{S}^3$. Our logarithmic transformations act on manifolds of the form $\Sigma\times\mathrm{S}^1$, with $\Sigma$ a homotopy sphere, and may change their diffeomorphism type. We then explain the relationship between these surgeries and $\star$-diagrams, and finally connect both notions to that of spherical T-duality.

\vspace{1em}

\subsection{Logarithmic transformations on \texorpdfstring{$\mathrm{T}^2$}{T²}-fibrations}
\label{sec:topological}

Recall that a $C^{\infty}$ elliptic fibration is a smooth surjective map $\pi: V\rightarrow C$ from a closed oriented 4-manifold $V$ onto a closed oriented surface $C$ that restricts to a $\mathrm{T}^2$-bundle over an open dense subset $U\subset C$. The fibers over points of $U$ are called \emph{regular}. 

\begin{definition}[Logarithmic transformations, $C^{\infty}$-version \cite{GOMPF1991479}]\label{def:classical-log}
Let $\pi: V\rightarrow C$ be as above. We say that a 4-dimensional manifold $W$ is obtained from $V$ via \emph{logarithmic transformations} if $W$ can be built out of $V$ via the following procedure: Choose regular fibers $F_1,\ldots, F_k$ in $V$ with disjoint closed tubular neighborhoods $\nu F_i,~i=1,\ldots,k$, with each $\nu F_i$ the pre-image of a disk under $\pi$. Each $\nu F_i$ is the preimage of a disk of regular values; since fiber bundles over disks are trivial, $\nu F_i$ is diffeomorphic to $\mathrm{T}^2\times\mathrm{D}^2$, and in particular $\partial\nu F_i\cong\mathrm{T}^2\times\mathrm{S}^1$. Delete the interior $\mathring{\nu F_i}$ of each $\nu F_i$ from $V$ and replace it with $\mathrm{T}^2\times\mathrm{D}^2$, glued by some diffeomorphism $\varphi_i: \mathrm{T}^2\times\mathrm{S}^1\rightarrow \partial \nu F_i$.
\end{definition}

\begin{example}
Consider the Hopf surface $\mathrm{S}^1\times\mathrm{S}^3$, endowed with the elliptic fibration $\pi:=\mathrm{pr}_2\circ(\mathrm{id}\times h):\mathrm{S}^1\times\mathrm{S}^3\rightarrow\mathrm{S}^2$, where
\[h:\mathrm{S}^3\rightarrow \mathrm{S}^2\]
is the Hopf map; this is a $\mathrm{T}^2$-bundle, so every fiber is regular. We first describe $\mathrm{S}^3$ as two solid tori $\mathrm{S}^1\times\mathrm{D}^2$ glued together: the two solid tori are the preimages, under $h$, of the two closed hemispheres of $\mathrm{S}^2$ (those around the poles $[1:0]$ and $[0:1]$).

Spelling it out, write $\mathrm{S}^3=\{(z,w)\in \mathbb C^2 \mid |z|^2+|w|^2=2\}$. The Hopf fibration is given by the map $h:(z,w)\mapsto [z:w]\in \mathbb{CP}^1\cong \mathrm{S}^2$. Set
\[\mathrm{S}^3_{+}=\{(z,w)\in \mathrm{S}^3 \mid 0\leq |w|^2\leq 1\}\]
\[\mathrm{S}^3_-=\{(z,w)\in \mathrm{S}^3 \mid 0\leq |z|^2\leq 1\}.\]
There are diffeomorphisms
\[f_+:\mathrm{S}^3_+\rightarrow \mathrm{S}^1\times \mathrm{D}^2, \quad f_+(z,w):=\left(\dfrac{z}{|z|},\dfrac{w}{z}\right)\]
\[f_-:\mathrm{S}^3_-\rightarrow \mathrm{S}^1\times \mathrm{D}^2, \quad f_-(z,w):=\left(\dfrac{w}{|w|},\dfrac{z}{w}\right).\]

The restriction of $f_+\circ f_-^{-1}$ to the boundary $\mathrm{S}^1\times\partial \mathrm{D}^2\rightarrow \mathrm{S}^1\times\partial \mathrm{D}^2$ is given by
\[f_+\circ f_-^{-1}(u,\xi)=(u\xi,\bar\xi).\]
Extending this map by the identity on the extra circle factor, and writing $\mathrm{T}^2=\mathrm{S}^1\times\mathrm{S}^1$ with the first factor the extra one, yields the diffeomorphism
\[\zeta:\mathrm{T}^2\times\partial\mathrm{D}^2\rightarrow \mathrm{T}^2\times\partial\mathrm{D}^2,\qquad \zeta(t,u,\xi)=(t,\,u\xi,\,\bar\xi),\]
so that the Hopf surface $\mathrm{S}^1\times\mathrm{S}^3$ is given by the gluing
\begin{equation}\label{eq:Hopf-zentner}
\mathrm{S}^1\times\mathrm{S}^3=\left(\mathrm{T}^2\times\mathrm{D}^2\right) \cup_{\zeta} \left(\mathrm{T}^2\times\mathrm{D}^2\right).
\end{equation}

Let $X'$ be the manifold obtained from the Hopf surface by performing a logarithmic transformation on the fibers $F_{\pm}$ over the North and South poles $x_+=[1:0],~x_-=[0:1]$, respectively. Denote by $\varphi_{\pm}$ the respective diffeomorphisms $\varphi_{\pm}:\mathrm{T}^2\times\mathrm{S}^1\rightarrow \partial\nu F_{\pm}$. According to the decomposition \eqref{eq:Hopf-zentner}, we may take $\nu F_{\pm}$ to be the copies of $\mathrm{T}^2\times\mathrm{D}^2_{1/2}$ inside the two pieces, where $\mathrm{D}^2_{1/2}\subset\mathrm{D}^2$ denotes the closed disk of radius $1/2$; the boundaries $\partial\nu F_{\pm}$ are then the inner boundaries of the two copies of $\mathrm{T}^2\times \left(\mathrm{D}^2 \setminus \mathring{\mathrm{D}}^2_{1/2}\right)$. Let $X_{\pm}$ be
\[X_{\pm}=\left(\mathrm{T}^2\times \left(\mathrm{D}^2 \setminus \mathring{\mathrm{D}}^2_{1/2}\right)\right)\cup_{\varphi_{\pm}} \left(\mathrm{T}^2\times\mathrm{D}^2\right).\]
Then
\begin{equation}\label{eq:log-zentner}
X'=X_+\cup_{\zeta}X_-.
\end{equation}
\end{example}

\vspace{1em}

In \cite{ZENTNER200637}, Zentner shows that if $X'$ has the same integral homology as $\mathrm{S}^1\times\mathrm{S}^3$, then $X'$ is in fact diffeomorphic to $\mathrm{S}^1\times\mathrm{S}^3$. Here, we propose a notion of generalized logarithmic transformation applied to manifolds of the form $\Sigma\times\mathrm{S}^1$, where $\Sigma$ is a homotopy sphere that bounds a parallelizable manifold. Unlike in the Hopf surface case, these transformations may change the diffeomorphism type of the resulting manifold while preserving the integral homology of $\Sigma\times\mathrm{S}^1$. This is supported by the following theorem:

\begin{theorem}[Brieskorn--Van de Ven \cite{BRIESKORN1968389}]\label{thm:brieskorn-van-de-ven}
Let $\Sigma$ and $\Sigma'$ be homotopy spheres of dimension $n \geq 5$. Then $\Sigma\times \mathrm{S}^1$ is diffeomorphic to $\Sigma'\times\mathrm{S}^1$ if and only if $\Sigma$ is diffeomorphic to $\Sigma'$.
\end{theorem}

\vspace{1em}

\subsection{A generalized logarithmic transformation proposal}

Let $G$ be a compact Lie group and assume there exists a linear action of $G$ on the standard unit spheres $\mathrm{S}^k,~\mathrm{S}^l, ~k,l\geq 2$. Let 
\[\mathrm{S}^{k+l+1}=\left(\mathrm{D}^{k+1}\times\mathrm{S}^l\right) \cup \left(\mathrm{S}^k\times\mathrm{D}^{l+1}\right)\] 
be the $G$-manifold obtained by regarding $\mathrm{S}^{k+l+1}$ as the unit sphere of $\mathbb{R}^{k+1}\oplus\mathbb{R}^{l+1}$, on which $G$ acts by the orthogonal direct sum $\Delta_1\oplus\Delta_2$ of the two linear actions; the induced action restricts to the product action on the Clifford-type common boundary $\mathrm{S}^k\times\mathrm{S}^l$ and preserves the two pieces. Set $r(x,y)=a(x)b(y)^{-1}$ for smooth maps $a:\mathrm{S}^k\rightarrow G,~b: \mathrm{S}^l\rightarrow G$ satisfying the covariance conditions $a(gx)=ga(x)g^{-1}$ and $b(gy)=gb(y)g^{-1}$ for every $g\in G$. Then $r$ is equivariant with respect to the diagonal subgroup $\Delta G=\{(g,g):g\in G\}\leq G\times G$, acting on $\mathrm{S}^k\times\mathrm{S}^l$ by the product action and on $G$ by conjugation. Namely,
\[r(gx,gy)=gr(x,y)g^{-1}, \quad \forall g\in G.\]
Indeed,
\begin{align}
	r(gx,gy) &= a(gx)b(gy)^{-1} \nonumber \\
	&= ga(x)g^{-1}(gb(y)g^{-1})^{-1} \nonumber \\
	&= ga(x)g^{-1}gb(y)^{-1}g^{-1} \nonumber \\
	&= ga(x)b(y)^{-1}g^{-1}.
\end{align}
More generally, given a $G$-manifold $X$ and a smooth map $\Phi:X\to G$, we write
\begin{equation}\label{eq:hat-operator}
\widehat{\Phi}:X\to X,\qquad \widehat{\Phi}(z):=\Phi(z)\cdot z,
\end{equation}
the action being the one fixed on $X$; the notation is used throughout, and the relevant action will always be specified. For $X=\mathrm{S}^k\times\mathrm{S}^l$ with the diagonal action and $\Phi=r$, this gives
\[\widehat{r}:(x,y)\mapsto r(x,y)\cdot(x,y)=\big(a(x)b(y)^{-1}x,\;a(x)b(y)^{-1}y\big).\]

Throughout, a twisted double $\left(\mathrm{D}^{k+1}\times\mathrm{S}^l\right)\cup_{\phi}\left(\mathrm{S}^{k}\times\mathrm{D}^{l+1}\right)$ is the closed manifold obtained by identifying $z\in\partial\left(\mathrm{D}^{k+1}\times\mathrm{S}^l\right)$ with $\phi(z)\in\partial\left(\mathrm{S}^{k}\times\mathrm{D}^{l+1}\right)$, for $\phi\in\mathrm{Diff}(\mathrm{S}^k\times\mathrm{S}^l)$. If $F$ and $H$ are diffeomorphisms of $\mathrm{D}^{k+1}\times\mathrm{S}^l$ and of $\mathrm{S}^{k}\times\mathrm{D}^{l+1}$, respectively (so that each restricts to a diffeomorphism of the boundary $\mathrm{S}^k\times\mathrm{S}^l$) then $F\sqcup H$ descends to a diffeomorphism
\begin{equation}\label{eq:regluing}
\left(\mathrm{D}^{k+1}\times\mathrm{S}^l\right)\cup_{\phi}\left(\mathrm{S}^{k}\times\mathrm{D}^{l+1}\right)\;\cong\;\left(\mathrm{D}^{k+1}\times\mathrm{S}^l\right)\cup_{H|_{\partial}\circ\phi\circ (F|_{\partial})^{-1}}\left(\mathrm{S}^{k}\times\mathrm{D}^{l+1}\right).
\end{equation}
In particular, isotopic gluing diffeomorphisms produce diffeomorphic twisted doubles. The next lemma, due to Sperança, identifies $\Sigma_r$; we spell out its proof in full, since the Mayer--Vietoris argument used in its proof will be needed to justify our definition of generalized logarithmic transformation.

\begin{lemma}[{\cite[Prop.~4 and Cor.~5]{speranca2016pulling}}]\label{lem:speranca-sigma-r}
Let
\[f_a(x,y)=(x,\,a(x)y),\qquad g_b(x,y)=(b(y)x,\,y)\]
be the one-sided clutching diffeomorphisms of $\mathrm{S}^k\times\mathrm{S}^l$ determined by $a$ and $b$. Then:
\begin{enumerate}
\item $\Sigma_r:=\left(\mathrm{D}^{k+1}\times\mathrm{S}^l\right) \cup_{\widehat r} \left(\mathrm{S}^{k}\times\mathrm{D}^{l+1}\right)$ is diffeomorphic to the twisted double
\[\left(\mathrm{D}^{k+1}\times\mathrm{S}^l\right) \cup_{g_b^{-1}\circ f_a} \left(\mathrm{S}^k\times\mathrm{D}^{l+1}\right),\]
whose diffeomorphism type depends only on the homotopy classes $[\Delta_2\circ a]\in\pi_k(\mathrm{SO}(l+1))$ and $[\Delta_1\circ b]\in\pi_l(\mathrm{SO}(k+1))$, where $\Delta_1:G\to\mathrm{SO}(k+1)$ and $\Delta_2:G\to\mathrm{SO}(l+1)$ are the homomorphisms given by the linear $G$-actions on $\mathrm{S}^k$ and $\mathrm{S}^l$.
\item $\Sigma_r$ is a homotopy $(k+l+1)$-sphere. In particular, $[\Sigma_r]\in\Theta^{k+l+1}$, the group of homotopy $(k+l+1)$-spheres up to diffeomorphism under connected sum.
\end{enumerate}
\end{lemma}
\begin{proof}
Set $\widehat a(x):=a(x)x$ and $\widehat b(y):=b(y)y$. The covariance condition $a(gx)=ga(x)g^{-1}$, evaluated at $g=a(x)$, gives
\begin{equation}\label{eq:hat-idempotence}
a(\widehat a(x)) = a(a(x)x)=a(x)a(x)a(x)^{-1}=a(x);
\end{equation}
evaluated at $g=a(x)^{-1}$, it gives $a(a(x)^{-1}x)=a(x)$ as well. Consequently $x\mapsto a(x)^{-1}x$ is a two-sided inverse of $\widehat a$, so $\widehat a$ is a diffeomorphism of $\mathrm{S}^k$ with $\widehat a^{-1}(x)=a(x)^{-1}x$; likewise $\widehat b$ is a diffeomorphism of $\mathrm{S}^l$ with $\widehat b^{-1}(y)=b(y)^{-1}y$. Both maps are smooth, being composites of smooth maps.

Write $A(x,y)=a(x)$, $B(x,y)=b(y)$, so that $\widehat A(x,y)=(a(x)x,\,a(x)y)$ and $\widehat B(x,y)=(b(y)x,\,b(y)y)$. Using \eqref{eq:hat-idempotence}:
\begin{align}
	f_a(\widehat a\times\mathrm{id})(x,y) &= f_a(a(x)x,y) = (a(x)x,\,a(\widehat a(x))y) = (a(x)x,\,a(x)y) = \widehat A(x,y) \nonumber\\
    &= (\widehat a\times\mathrm{id})f_a(x,y),
\end{align}
and analogously $g_b(\mathrm{id}\times\widehat b)=(\mathrm{id}\times\widehat b)g_b=\widehat B$; in particular $\widehat B^{-1}=g_b^{-1}(\mathrm{id}\times\widehat b^{-1})=(\mathrm{id}\times\widehat b^{-1})g_b^{-1}$. Next, $\widehat r=\widehat B^{-1}\circ\widehat A$: indeed, since $b(a(x)y)=a(x)b(y)a(x)^{-1}$ by covariance,
\begin{align}
\widehat B^{-1}\widehat A(x,y) &= \widehat B^{-1}(a(x)x,\,a(x)y) \nonumber\\
&= \big(b(a(x)y)^{-1}a(x)x,\; b(a(x)y)^{-1}a(x)y\big) \nonumber\\
&= \big(a(x)b(y)^{-1}x,\;a(x)b(y)^{-1}y\big) = \widehat r(x,y).
\end{align}
Combining the last two paragraphs,
 \begin{align}
 	\widehat r = \widehat B^{-1}\widehat A &= g_b^{-1}\circ(\mathrm{id}\times\widehat b^{-1})\circ f_a\circ (\widehat a\times\mathrm{id}) \nonumber \\
 	&= g_b^{-1}\circ f_a \circ (\mathrm{id}\times \widehat b^{-1})\circ (\widehat a\times\mathrm{id}) \nonumber \\
 	&= (\widehat a\times\mathrm{id})\circ g_b^{-1}f_a\circ (\mathrm{id}\times \widehat b^{-1}),
 \end{align}
where the second equality uses $(\mathrm{id}\times\widehat b^{-1})f_a = f_a(\mathrm{id}\times\widehat b^{-1})$ and the third uses that $(\widehat a\times\mathrm{id})$ commutes with $(\mathrm{id}\times\widehat b^{-1})$ (disjoint coordinates) and with $g_b^{-1}$: both $g_b^{-1}(\widehat a\times\mathrm{id})$ and $(\widehat a\times\mathrm{id})g_b^{-1}$ equal $(x,y)\mapsto(b(y)^{-1}a(x)x,\,y)$, again by covariance.

Since $F:=\mathrm{id}\times\widehat b$ and $H:=\widehat a\times\mathrm{id}$ are diffeomorphisms of $\mathrm{D}^{k+1}\times\mathrm{S}^l$ and of $\mathrm{S}^k\times\mathrm{D}^{l+1}$, respectively, the regluing principle \eqref{eq:regluing} applied to $\phi=g_b^{-1}f_a$ yields item (1); the dependence only on the homotopy classes follows from \eqref{eq:regluing}, since homotopies of $\Delta_2\circ a$ and $\Delta_1\circ b$ induce isotopies of $f_a$ and $g_b$.

For item (2), set $\psi:=g_b^{-1}f_a$ and compute the Mayer--Vietoris sequence of $\Sigma_r=A\cup_\psi B$ with $A=\mathrm{D}^{k+1}\times\mathrm{S}^l\simeq\mathrm{S}^l$ and $B=\mathrm{S}^k\times\mathrm{D}^{l+1}\simeq\mathrm{S}^k$. In degree $l$, the class $[\mathrm{pt}\times\mathrm{S}^l]$ maps to a generator of $\mathrm{H}_l(A)$ through the identity-side inclusion, so the relevant map is unimodular. In degree $k$, the map is multiplication by
\[P_{00}:=\deg\Big(\mathrm{pr}_1\circ\psi\big\vert_{\mathrm{S}^k\times\{y_0\}}\Big),\]
and $\mathrm{H}_*(\Sigma_r)\cong\mathrm{H}_*(\mathrm{S}^{k+l+1})$ if and only if $P_{00}=\pm 1$. When $k=l$ the two degrees combine into the single matrix $\bigl(\begin{smallmatrix}0&1\\ P_{00}&P_{01}\end{smallmatrix}\bigr)$ of determinant $-P_{00}$, and the same criterion results. Now, by covariance, $\psi(x,y_0)=\big(a(x)b(y_0)^{-1}a(x)^{-1}x,\;a(x)y_0\big)$, so $\mathrm{pr}_1\circ\psi\vert_{\mathrm{S}^k\times\{y_0\}}$ has the form $x\mapsto R(x)x$ with $R(x)=\Delta_1\big(a(x)b(y_0)^{-1}a(x)^{-1}\big)\in\mathrm{SO}(k+1)$. Since $k\geq 2$, the standard co-$H$ decomposition gives $\deg\big(x\mapsto R(x)x\big)=1+\deg\big(\mathrm{ev}_{x_0}\circ R\big)$ in $\pi_k(\mathrm{S}^k)\cong\mathbb{Z}$, where $\mathrm{ev}_{x_0}(A)=Ax_0$ for a fixed $x_0\in\mathrm{S}^k$. Here $R$ is nullhomotopic: since $\mathrm{SO}(k+1)$ is connected, choose a path $\gamma_t$ from $\Delta_1(b(y_0)^{-1})$ to the identity and set $R_t(x)=\Delta_1(a(x))\,\gamma_t\,\Delta_1(a(x))^{-1}$, a homotopy from $R$ to the constant map $\mathrm{id}$. Hence $\deg(\mathrm{ev}_{x_0}\circ R)=0$, so $P_{00}=1$ and $\Sigma_r$ is an integral homology $(k+l+1)$-sphere. We emphasize that this step uses covariance essentially: for an arbitrary gluing of the form $g_\beta^{-1}f_\alpha$ it can fail (Remark \ref{rem:milnor-caution}).

Finally, since $\mathrm{D}^{k+1}\times\mathrm{S}^l$, $\mathrm{S}^k\times\mathrm{D}^{l+1}$ and their overlap $\mathrm{S}^k\times\mathrm{S}^l$ are simply connected when $k,l\geq 2$, the Seifert--van Kampen theorem, applied to open thickenings of the two pieces, gives $\pi_1(\Sigma_r)=0$. Collapsing the complement of an open chart disk $\mathrm{D}^{k+l+1}\subset\Sigma_r$ produces a degree-one map $\Sigma_r\to\mathrm{S}^{k+l+1}$, which induces isomorphisms on all integral homology groups. The homological Whitehead theorem implies it is a homotopy equivalence. Hence $\Sigma_r$ is a homotopy $(k+l+1)$-sphere, and $[\Sigma_r]\in\Theta^{k+l+1}$. \qedhere
\end{proof}
   
In view of Lemma \ref{lem:speranca-sigma-r}, for covariant data $(a,b)$ we set \[\sigma(\Delta_2a,\Delta_1b):=[\Sigma_r]\in\Theta^{k+l+1}\] and, following \cite[Cor.~5]{speranca2016pulling}, call it the \emph{Milnor sphere} of the pair. By Lemma \ref{lem:speranca-sigma-r}(1), it depends only on the homotopy classes $[\Delta_2\circ a]$ and $[\Delta_1\circ b]$. Example \ref{ex:gm-revisited} below realizes the Gromoll--Meyer sphere $\Sigma^7_{GM}=M_{2,-1}$ in this form.  
\begin{remark}\label{rem:milnor-caution}
Covariance is essential in Lemma \ref{lem:speranca-sigma-r}(2). For arbitrary classes $\alpha\in\pi_k(\mathrm{SO}(l+1))$ and $\beta\in\pi_l(\mathrm{SO}(k+1))$, the twisted double glued by $g_\beta^{-1}\circ f_\alpha$, with $f_\alpha(x,y)=(x,\alpha(x)y)$ and $g_\beta(x,y)=(\beta(y)x,y)$, need not be a homotopy sphere. 

For instance, take $k=l=3$ and $\alpha=\beta=[t_{1,0}]$, the class of left quaternionic multiplication, with representatives $\alpha(x)=x$ and $\beta(y)=y$. We first record that these representatives are not covariant for the left-translation action of $G=\mathrm{S}^3$: covariance for $\alpha$ would require $\alpha(gx)=g\alpha(x)g^{-1}$, that is, $gx=gxg^{-1}$, which after cancelling $gx$ on the left forces $g=1$. The computation below shows more, namely that the pair $\big([t_{1,0}],[t_{1,0}]\big)$ admits \emph{no} covariant representatives at all, for any compact Lie group $G$ and any pair of linear actions $\Delta_1,\Delta_2$. Indeed, by the regluing principle \eqref{eq:regluing}, the diffeomorphism type of the twisted double glued by $g_\beta^{-1}\circ f_\alpha$ depends only on the homotopy classes $\alpha$ and $\beta$, and not on the group or the actions realizing them; were the pair realized by covariant data $(a,b)$ for some $(G,\Delta_1,\Delta_2)$, Lemma \ref{lem:speranca-sigma-r}(2) would force that twisted double to be a homotopy sphere, contradicting the computation below.

To evaluate the topological consequence of this failure, we compute the twisted double gluing map $\psi = g_\beta^{-1} \circ f_\alpha$. Setting $(u,v) = (\beta(y)x, y) = (yx, y)$ yields $v = y$ and $x = y^{-1}u$, so the inverse map is $g_\beta^{-1}(u,v) = (v^{-1}u, v)$. We evaluate the composition:
\begin{align}
    \psi(x,y) &= g_\beta^{-1}(x, xy) \\
              &= ((xy)^{-1}x, xy).
\end{align}
In the group of unit quaternions, the inverse is the conjugate, so $(xy)^{-1} = \overline{xy} = \bar{y}\bar{x}$. The first coordinate of $\psi(x,y)$ becomes:
\begin{equation}
    \overline{xy}x = (\bar{y}\bar{x})x = \bar{y}(\bar{x}x).
\end{equation}
Since $x \in \mathrm{S}^3$, we have $\bar{x}x = |x|^2 = 1$, leaving the first coordinate as $\bar{y}$. Thus, $\psi(x,y) = (\bar{y}, xy)$.

Restricting this map to the slice $\mathrm{S}^3 \times \{y_0\}$ for a fixed $y_0 \in \mathrm{S}^3$, the projection to the first coordinate is $\mathrm{pr}_1 \circ \psi(x, y_0) = \bar{y}_0$. This is a constant map from $\mathrm{S}^3$ to $\mathrm{S}^3$. Consequently, its topological degree is $P_{00} = 0$. 

Write $X:=\left(\mathrm{D}^4\times\mathrm{S}^3\right)\cup_{\psi}\left(\mathrm{S}^3\times\mathrm{D}^4\right)$ for the resulting twisted double. From the Mayer--Vietoris argument in the proof of Lemma \ref{lem:speranca-sigma-r}, the map $\mathrm{H}_3(\mathrm{S}^3\times\mathrm{S}^3)\rightarrow\mathrm{H}_3(\mathrm{D}^4\times\mathrm{S}^3)\oplus\mathrm{H}_3(\mathrm{S}^3\times\mathrm{D}^4)$ is given by the matrix $\bigl(\begin{smallmatrix}0&1\\ P_{00}&P_{01}\end{smallmatrix}\bigr)$, whose cokernel computes $\mathrm{H}_3(X)$ and whose kernel computes $\mathrm{H}_4(X)$. Because $P_{00}=0$, the image is the rank-one subgroup generated by the primitive vector $(1,P_{01})$, whence $\mathrm{H}_3(X)\cong\mathbb{Z}\neq 0$, and likewise $\mathrm{H}_4(X)\cong\mathbb{Z}$. The manifold $X$ is therefore not an integral homology sphere, and in particular not a homotopy sphere. Accordingly, the Milnor sphere $\sigma(\Delta_2a,\Delta_1b)$ is defined here only for pairs of classes admitting covariant representatives, which is the only situation arising in this paper. Finally, we note that in \cite{speranca2016pulling} the twisted doubles considered above are themselves called \emph{plumbings}; we reserve the latter term for the disk-bundle construction and do not employ it here.
\end{remark}

\begin{definition}[Product-Preserving Generalized Logarithmic Transformations]\label{def:generalized-log}
    Let $\Sigma_r\times\mathrm{S}^1$ be the manifold obtained via the twisted-double gluing: 
    \[\Sigma_r\times\mathrm{S}^1 := \left(\mathrm{D}^{k+1}\times\mathrm{S}^l\times\mathrm{S}^1\right) \cup_{\zeta} \left(\mathrm{S}^{k}\times\mathrm{D}^{l+1}\times\mathrm{S}^1\right)\]
where $\zeta := \widehat r \times \mathrm{id}_{\mathrm{S}^1}$, i.e. $\zeta(x,y,\theta) = \big(\widehat r(x,y),\, \theta\big)$ for $(x,y)\in\mathrm{S}^k\times\mathrm{S}^l$ and $\theta\in \mathrm{S}^1$. One says that a manifold $(\Sigma_r\times\mathrm{S}^1)'$ is obtained from $\Sigma_r\times\mathrm{S}^1$ via a \emph{product-preserving generalized logarithmic transformation} if 
\[(\Sigma_r\times\mathrm{S}^1)' \cong \left(\mathrm{D}^{k+1}\times\mathrm{S}^l\times\mathrm{S}^1\right) \cup_{\zeta\circ \phi} \left(\mathrm{S}^{k}\times\mathrm{D}^{l+1}\times\mathrm{S}^1\right)\]
for a boundary diffeomorphism $\phi$ of the product form
\[\phi = \psi \times \mathrm{id}_{\mathrm{S}^1},\]
where $\psi \in \mathrm{Diff}(\mathrm{S}^k \times \mathrm{S}^l)$ is \emph{admissible}, meaning that the twisted double
\[\Sigma_\psi := \left(\mathrm{D}^{k+1}\times\mathrm{S}^l\right) \cup_{\widehat r \circ \psi} \left(\mathrm{S}^{k}\times\mathrm{D}^{l+1}\right)\]
is again a homotopy $(k+l+1)$-sphere. 

For any such $\phi$, the gluing factors as $(\Sigma_r\times\mathrm{S}^1)'\cong\Sigma_\psi\times\mathrm{S}^1$, since $\zeta\circ\phi=(\widehat r\circ \psi)\times \mathrm{id}_{\mathrm{S}^1}$ and gluing commutes with taking the product with $\mathrm{S}^1$; admissibility then ensures that the factor $\Sigma_\psi$ is a homotopy sphere. Moreover, using the identity $\widehat r = (\widehat a\times\mathrm{id})\circ g_b^{-1}f_a\circ(\mathrm{id}\times\widehat b^{-1})$ established in the proof of Lemma \ref{lem:speranca-sigma-r}, together with the regluing principle \eqref{eq:regluing} crossed with $\mathrm{id}_{\mathrm{S}^1}$, the same manifold is obtained from the boundary map $(g_b^{-1}\circ f_a\circ \widetilde{\psi})\times\mathrm{id}_{\mathrm{S}^1}$, where $\widetilde{\psi} = (\mathrm{id}\times\widehat b^{-1})\circ\psi\circ(\mathrm{id}\times\widehat b)$: the factors $\widehat a\times\mathrm{id}$ and $\mathrm{id}\times\widehat b^{-1}$ are absorbed into diffeomorphisms of $\mathrm{S}^{k}\times\mathrm{D}^{l+1}$ and $\mathrm{D}^{k+1}\times\mathrm{S}^l$, respectively.
\end{definition}

\begin{remark}[Admissibility is a homological, and genuinely restrictive, condition]\label{rem:admissibility}
The Mayer--Vietoris analysis in the proof of Lemma \ref{lem:speranca-sigma-r} applies verbatim to the twisted double associated with an \emph{arbitrary} gluing diffeomorphism $\phi\in\mathrm{Diff}(\mathrm{S}^k\times\mathrm{S}^l)$: in degree $l$ the relevant map is unimodular through the identity-side inclusion, independently of $\phi$, while in degree $k$ it is multiplication by
\[P_{00}(\phi):=\deg\Big(\mathrm{pr}_1\circ\phi\big\vert_{\mathrm{S}^k\times\{y_0\}}\Big),\]
so that the twisted double glued by $\phi$ is an integral homology $(k+l+1)$-sphere if and only if $P_{00}(\phi)=\pm1$. Since $k,l\geq 2$, the Seifert--van Kampen argument of that proof gives simple connectivity regardless of $\phi$, and the collapse-map argument upgrades the homological conclusion to the homotopy-sphere property. Applying this to the diffeomorphism $\phi=\widehat r\circ\psi$, we conclude that $\psi$ is admissible in the sense of Definition \ref{def:generalized-log} if and only if
\[P_{00}(\widehat r\circ\psi)=\pm 1;\]
in particular, since the degree is a homotopy invariant, admissibility depends only on the isotopy class of $\psi$.

Admissibility is a genuine restriction. Take $k=l$ and the trivial covariant data $a\equiv b\equiv e$, so that $\widehat r=\mathrm{id}$ and $\Sigma_r=\mathrm{S}^{2k+1}$. For the coordinate swap $\psi(x,y)=(y,x)$, the restriction of $\mathrm{pr}_1\circ\psi$ to $\mathrm{S}^k\times\{y_0\}$ is the constant map $x\mapsto y_0$, so $P_{00}(\widehat r\circ\psi)=0$ and $\mathrm{H}_k(\Sigma_\psi)\cong\mathbb{Z}$. Indeed, relabeling the second piece by the swap identifies $\Sigma_\psi$ with two copies of $\mathrm{D}^{k+1}\times\mathrm{S}^k$ glued along their common boundary by the identity, whence
\begin{equation}\Sigma_\psi\;\cong\;\left(\mathrm{D}^{k+1}\cup_{\mathrm{S}^k}\mathrm{D}^{k+1}\right)\times\mathrm{S}^k\;=\;\mathrm{S}^{k+1}\times\mathrm{S}^k. \end{equation}
\end{remark}

\begin{remark}[Definition \ref{def:classical-log} vs \ref{def:generalized-log}]
Recall that, given a closed submanifold $K$ of codimension $c$ inside an ambient manifold $M$, with trivialized closed tubular neighborhood $\nu K \cong K \times \mathrm{D}^c$, one obtains a new manifold by excising the interior of $\nu K$ and regluing $K \times \mathrm{D}^c$ via a diffeomorphism onto $\partial\nu K$; in the 4-dimensional literature, this cut-and-paste operation is referred to as a \emph{surgery on $K$} (see \cite[Section 8.3]{gompf19994}), and we caution that it differs from the classical surgery of \cite[Chapter VI]{ko}, in which the excised $\mathrm{S}^p\times\mathrm{D}^q$ is replaced by the different piece $\mathrm{D}^{p+1}\times\mathrm{S}^{q-1}$. Smoothly, the cut-and-paste operation coincides with the (generalized) normal connected sum, or fiber sum, of $M$ with the manifold $Y = K \times \mathrm{S}^c$ along $K$. Indeed, decomposing the sphere $\mathrm{S}^c = \mathrm{D}^c_- \cup_{\mathrm{S}^{c-1}} \mathrm{D}^c_+$, we can write $Y = (K \times \mathrm{D}^c_-) \cup (K \times \mathrm{D}^c_+)$. Removing the tubular neighborhood $K \times \mathring{\mathrm{D}}^c_-$ from $Y$ leaves exactly $K \times \mathrm{D}^c_+$, which is then glued to $M \setminus \mathring{\nu}K$, recovering precisely the cut-and-paste operation just described.

In the classical $C^{\infty}$-logarithmic transformation, the ambient space is a 4-manifold and the surgery locus is a torus fiber $K = \mathrm{T}^2 = \mathrm{S}^1 \times \mathrm{S}^1$. The normal bundle is framed as $\mathrm{T}^2 \times \mathrm{D}^2$, implying codimension $c=2$. By the identification above, this operation is exactly a surgery on $\mathrm{S}^1 \times \mathrm{S}^1$ in the present sense; equivalently, it is the fiber sum of the 4-manifold with $\mathrm{T}^2 \times \mathrm{S}^2$ along $\mathrm{T}^2$ (see \cite[Section 8.3]{gompf19994}).

Our generalized procedure operates via the same mechanism in higher dimensions. The ambient manifold $\Sigma_r\times\mathrm{S}^1$ has dimension $k+l+2$ and is decomposed into $M_1 = \mathrm{D}^{k+1}\times\mathrm{S}^l\times\mathrm{S}^1$ and $M_2 = \mathrm{S}^{k}\times\mathrm{D}^{l+1}\times\mathrm{S}^1$. Within $M_2$, we identify the $(k+1)$-dimensional submanifold $K = \mathrm{S}^k \times \{0\} \times \mathrm{S}^1$. The piece $M_2$ serves precisely as the trivial tubular neighborhood of $K$, with its normal bundle framed by $\mathrm{D}^{l+1}$, giving a codimension of $c=l+1$. For the sake of completeness, let us check that explicitly.

For any point $p = (x,0,z) \in K$, the tangent space of the ambient manifold splits as $T_p M_2 \cong T_x \mathrm{S}^k \oplus \mathbb{R}^{l+1} \oplus T_z \mathrm{S}^1$, since $0$ is an interior point of $\mathrm{D}^{l+1}$. The tangent space of the submanifold is $T_p K \cong T_x \mathrm{S}^k \oplus \{0\} \oplus T_z \mathrm{S}^1$. The normal space is the quotient $N_p K = T_p M_2 / T_p K \cong \mathbb{R}^{l+1}$. Because this splitting is global, the normal bundle is the trivial bundle $\nu K \cong K \times \mathbb{R}^{l+1}$, giving a codimension of $l+1$. By the tubular neighborhood theorem \cite[Chapter~III]{ko}, a closed tubular neighborhood of $K$ is diffeomorphic to the closed unit disk bundle of $\nu K$, which is exactly $K \times \mathrm{D}^{l+1} = \mathrm{S}^k \times \mathrm{D}^{l+1} \times \mathrm{S}^1 = M_2$. The explicit Cartesian factor $\mathrm{D}^{l+1}$ serves as the canonical framing for this neighborhood.

Our proposed generalized logarithmic transformation excises the interior of $M_2$ and reglues it via the modified map $\zeta \circ (\psi \times \mathrm{id}_{\mathrm{S}^1})$. Consequently, this is exactly a surgery on the submanifold $K=\mathrm{S}^k \times\{0\}\times \mathrm{S}^1$ in the sense above. Equivalently, it is the fiber sum of $\Sigma_r \times \mathrm{S}^1$ with the manifold $(\mathrm{S}^k \times \mathrm{S}^1) \times \mathrm{S}^{l+1}$ along $\mathrm{S}^k \times \mathrm{S}^1$.
\end{remark}

\vspace{1em}

In the next example, we expand upon Example \ref{ex:gromollmeyer}, providing a clearer identification between the Gromoll--Meyer exotic sphere and generalized logarithmic transformations. This will serve as the starting point to connect our logarithmic transformations with $\star$-diagrams in a more general manner.

\begin{example}[Gromoll--Meyer, Sperança]\label{ex:gm-revisited}
    As in Example \ref{ex:gromollmeyer}, let $\mathrm{Sp}(2)$ be the set of $2\times 2$ quaternionic matrices respecting the identity $\bar Q^TQ=\mathrm{id}$, where $\bar Q^T$ is the conjugate transpose of $Q$. For $\mathrm{S}^4$ and $\mathrm{S}^7$, the unit spheres in $\mathbb R\times\mathbb H$ and $\mathbb H\times \mathbb H$, we define the Hopf map $h:\mathrm{S}^7\to \mathrm{S}^4$ as
\begin{equation}\label{h}
h(x,y)=(|x|^2-|y|^2, 2x\bar y).
\end{equation}
Consider the following commutative diagram (borrowed from \cite{rigas1996hopf}):
\begin{equation}\label{eq:chasing}
\begin{tikzcd}
\mathrm{Sp}(2) \arrow[dd, "\mathrm{p}_1"'] \arrow[rr, "\mathrm{p}_2"] &  & \mathrm{S}^7 \arrow[dd, "-h"] \\
& & \\
\mathrm{S}^7 \arrow[rr, "h"'] & & \mathrm{S}^4
\end{tikzcd}
\end{equation}
where $\mathrm{p}_i$ are the projections to the $i$-th column. Notice that the maps $\mathrm{p}_i:\mathrm{Sp}(2)\to \mathrm{S}^7$ are $\mathrm{S}^3$-principal bundles with principal actions given by right multiplication by the matrices $\mathrm{diag}(1,\bar q)$ and $\mathrm{diag}(\bar q, 1)$, respectively, where $q\in \mathrm{S}^3$, the group of unit quaternions.

Let $\mathrm{S}^4_{+}=\mathrm{S}^4 \setminus \{(-1,0)\}$ and $\mathrm{S}^4_-=\mathrm{S}^4 \setminus \{(1,0)\}$. Then, $\mathrm{S}^7$, viewed as the total space of a principal bundle over $\mathrm{S}^4$, is identified with the clutching presentation:
\[\left(\mathrm{S}^4_+\times \mathrm{S}^3\right) \cup_{f_a} \left(\mathrm{S}^4_-\times \mathrm{S}^3\right),\]
where the gluing is performed over the overlap $\mathrm{S}^4_+\cap\mathrm{S}^4_-$ by $f_a(\lambda,x,g)=(\lambda,x,\,g\,x/|x|)$, which is well defined there since $x\neq 0$. Recall from Example \ref{ex:gromollmeyer} that there is an $\mathrm{S}^3$-action on $\mathrm{Sp}(2)$ given by:
\begin{equation}\label{action Sp2} 
q \begin{pmatrix}a&c\\b&d\end{pmatrix} = \begin{pmatrix}qa\bar q&qc\\qb\bar q&qd\end{pmatrix}.
\end{equation} 
Next, we argue that the Gromoll--Meyer sphere $\Sigma^7_{GM}$ is identified with the Milnor bundle $M_{2,-1}$ and explain how this connects with generalized logarithmic transformations.

More precisely, we will prove that the quotient of $\mathrm{Sp}(2)$ by the action \eqref{action Sp2} is diffeomorphic to the twisted double:
\[\left(\mathrm{D}^4\times \mathrm{S}^3\right) \cup_{\beta_b^{-1} \circ \beta_a} \left(\mathrm{S}^3\times \mathrm{D}^4\right),\]
where the boundary diffeomorphisms $\beta_a,\beta_b:\mathrm{S}^3\times \mathrm{S}^3\rightarrow \mathrm{S}^3\times\mathrm{S}^3$ are defined by $\beta_a(x,y)=(x,xy\bar x)$ and $\beta_b(x,y)=(yx\bar y,y)$.\footnote{L. Sperança originally communicated this specific quotient construction to L.~F.~C.} These are precisely the one-sided clutching diffeomorphisms $f_a$ and $g_b$ of Lemma \ref{lem:speranca-sigma-r} for the covariant data $a(x)=x$, $b(y)=y$, with $G=\mathrm{S}^3$ acting on both sphere factors by conjugation; they are not to be confused with the bundle transition map $f_a$ above. In particular, Lemma \ref{lem:speranca-sigma-r}(2) shows that the twisted double glued by $\psi:=\beta_b^{-1}\circ\beta_a$ is a homotopy $7$-sphere, so that $\psi$ is admissible in the sense of Definition \ref{def:generalized-log}; by the classical Gromoll--Meyer computation recalled below, this twisted double is moreover a generator of $\Theta^7\cong \mathbb Z_{28}$. Consequently, $\Sigma^7_{GM}\times\mathrm{S}^1$ is obtained from $\mathrm{S}^7\times\mathrm{S}^1$, with $\mathrm{S}^7$ presented as the trivial twisted double $(\mathrm{D}^4\times\mathrm{S}^3)\cup_{\mathrm{id}}(\mathrm{S}^3\times\mathrm{D}^4)$, via the product-preserving generalized logarithmic transformation with boundary map $\phi=\psi\times\mathrm{id}_{\mathrm{S}^1}$.

\vspace{1em}

Chasing diagram \eqref{eq:chasing}, we conclude that
\begin{equation}\label{bundle Sp2}
\mathrm{Sp}(2) = \left(h^{-1}(\mathrm{S}^4_+)\times \mathrm{S}^3\right) \cup_{f_{ah}} \left(h^{-1}(\mathrm{S}^4_-)\times \mathrm{S}^3\right)
\end{equation}
as a principal bundle, where $f_{a h}(x,y,g)=\big(x,y,\,g\,x\bar y/|x\bar y|\big)$. Defining $a:\mathrm{S}^4_+\cap \mathrm{S}^4_-\to \mathrm{S}^3$ by $a(\lambda,x)=x/|x|$, the subindex $ah$ stands for the composition $a\circ h:h^{-1}(\mathrm{S}^4_+\cap \mathrm{S}^4_-)\to \mathrm{S}^3$, so that $ah(x,y)=x\bar y/|x\bar y|$; we henceforth write $ah:=a\circ h$. Notice that $a$ and $ah$ serve as transition maps for the bundle $h$ and for the presentation \eqref{bundle Sp2} of $\mathrm{p}_1$, respectively. Consider the following $\mathrm{S}^3$-action on $\mathrm{S}^7$:
\begin{equation}\label{action 7} 
q(x,y)=(qx\bar q,qy\bar q).
\end{equation}
Then, $\mathrm{p}_1(qQ)=q\mathrm{p}_1(Q)$. Furthermore, $a h(g(x,y))=ga h(x,y)g^{-1}$. In particular, the action \eqref{action Sp2}, under the identification \eqref{bundle Sp2}, is written as:
\[q(x,y,g)=(qx\bar q,qy\bar q,gq^{-1}).\]

\vspace{1em}

Notice that there is a diffeomorphism $F:\mathrm{Sp}(2)\rightarrow \mathrm{Sp}(2)$ fully characterized by the following diagram:
\begin{equation}
\begin{tikzcd}
	h^{-1}(\mathrm{S}^4_+) \times \mathrm{S}^3 \arrow[d, "F_+"]
	& h^{-1}(\mathrm{S}^4_+ \cap \mathrm{S}^4_-) \times \mathrm{S}^3 \arrow[l] \arrow[r, "f_{a h}"]
	& h^{-1}(\mathrm{S}^4_-) \times \mathrm{S}^3 \arrow[d, "F_-"]
	\\
	h^{-1}(\mathrm{S}^4_+) \times \mathrm{S}^3
	& h^{-1}(\mathrm{S}^4_+ \cap \mathrm{S}^4_-) \times \mathrm{S}^3 \arrow[l] \arrow[r, "{(\widehat{a h} \times \mathrm{id})f_{a h}^{-1}}"]
	& h^{-1}(\mathrm{S}^4_-) \times \mathrm{S}^3
\end{tikzcd}
\end{equation}
where $\widehat{ah}:h^{-1}(\mathrm{S}^4_+\cap \mathrm{S}^4_-)\to h^{-1}(\mathrm{S}^4_+\cap \mathrm{S}^4_-)$ is defined by $\widehat{ah}(x,y)=\left(ah(x,y)\right)(x,y)$ and $F_{\pm}(q(x,y,g)r^{-1})=r(F_{\pm}(x,y,g))q^{-1}$. Indeed, define $F_{\pm}$ by $(x,y,g)\mapsto (g(x,y),g^{-1})$, where $g(x,y)$ denotes the action \eqref{action 7}. Each $F_\pm$ is a smooth involution: applying it twice returns $\big(g^{-1}(g(x,y)),\,g\big)=(x,y,g)$. Moreover:
\[F_+^{-1} \circ f_{a h} \circ F_-(x,y,g) = F_+(g(x,y),g^{-1}a h(g(x,y))) = (a h(x,y)(x,y),ga h(x,y)^{-1})\]
\[F_{\pm}(q(x,y,g)r^{-1}) = F_{\pm}(q(x,y),rgq^{-1}) = (rg(x,y),qg^{-1}r^{-1}) = r(F_{\pm}(x,y,g))q^{-1},\]
using the elementary identity $(rgq^{-1})^{-1}=qg^{-1}r^{-1}$, and, in the first line, that $F_+$ is an involution, so $F_+^{-1}=F_+$. Since $f_{ah}^{-1}(x,y,g)=(x,y,\,g\,ah(x,y)^{-1})$, the first display identifies $F_+^{-1}\circ f_{ah}\circ F_-$ with $(\widehat{ah}\times\mathrm{id})\circ f_{ah}^{-1}$, as required by the diagram.

In the coordinates $(X,G):=F(x,y,g)=(g(x,y),g^{-1})$, the action \eqref{action Sp2} reads $q\cdot(X,G)=(X,\,qG)$, i.e., left translation on the last coordinate. Therefore, it defines an $\mathrm{S}^3$-principal bundle $\mathrm{p}_1':\mathrm{Sp}(2)\to h^{-1}(\mathrm{S}^4_-)\cup_{\widehat{ah}}h^{-1}(\mathrm{S}^4_+)$, where, in accordance with our gluing convention, $\widehat{ah}$ carries overlap points read on the $h^{-1}(\mathrm{S}^4_-)$-side to their identifications on the $h^{-1}(\mathrm{S}^4_+)$-side. We obtain the desired result by passing to compact charts. Fix $\rho=1/\sqrt{2}$ and set $A_\rho=\{(x,y)\in\mathrm{S}^7:|x|\leq\rho\}$, $B_\rho=\{(x,y)\in\mathrm{S}^7:|y|\leq\rho\}$, which meet along the generalized Clifford torus $|x|=|y|=\rho$, identified with $\mathrm{S}^3\times\mathrm{S}^3$. Let $\mathrm{D}^4_\rho$ denote the closed 4-disk of radius $\rho$. The maps
\[\Psi_\rho:\mathrm{D}^4_\rho\times \mathrm{S}^3\to A_\rho,\ (x,v)\mapsto (x,(1-|x|^2)^{1/2}v),\qquad \Phi_\rho:\mathrm{S}^3\times \mathrm{D}^4_\rho\to B_\rho,\ (u,y)\mapsto ((1-|y|^2)^{1/2}u,y),\]
are genuine diffeomorphisms, since $(1-|x|^2)^{1/2}\geq (1-\rho^2)^{1/2}>0$ on $\mathrm{D}^4_\rho$. Note that $\widehat{ah}$ preserves the torus $|x|=|y|=\rho$, since the action \eqref{action 7} preserves both norms. Restricting to the torus and conjugating by the chart diffeomorphisms gives the boundary map $\Phi_\rho^{-1}\circ \widehat{ah}\circ \Psi_\rho:\mathrm{S}^3\times\mathrm{S}^3\to\mathrm{S}^3\times\mathrm{S}^3$, which we now compute. On boundary points, $\Psi_\rho(\rho x,v)=(\rho x,\rho v)$ and $\Phi_\rho(u,\rho w)=(\rho u,\rho w)$, for unit quaternions $x,v,u,w$, since $(1-\rho^2)^{1/2}=\rho$. As $ah(\rho x,\rho v)=\rho^2x\bar v/|\rho^2 x\bar v|=x\bar v$, setting $g=x\bar v$ in \eqref{action 7} yields \[\widehat{ah}(\rho x,\rho v)=\big(\rho\,gx\bar g,\;\rho\,gv\bar g\big)=\big(\rho\, x\bar v\,x\,v\bar x,\;\rho\, x v\bar x\big),\] whence \[\Phi_\rho^{-1}\circ \widehat{ah}\circ \Psi_\rho(x,v)=\big(x\bar v\,x\,v\bar x,\; x v\bar x\big)=\beta_b^{-1}\circ\beta_a(x,v),\] since $\beta_b^{-1}(u,w)=(\bar w u w,\,w)$ gives $\beta_b^{-1}\beta_a(x,v)=\beta_b^{-1}(x,\,xv\bar x)=\big(x\bar v\bar x\cdot x\cdot xv\bar x,\;xv\bar x\big)=\big(x\bar v\,x\,v\bar x,\; xv\bar x\big)$. The resulting identification of the quotient of $\mathrm{Sp}(2)$ by the action \eqref{action Sp2} with $M_{2,-1}$, a generator of $\Theta^7\cong\mathbb{Z}_{28}$, is the classical Gromoll--Meyer computation \cite{gromoll1974exotic,speranca2016pulling}; the chart presentation above is what makes the explicit clutching map $\beta_b^{-1}\circ\beta_a$ available for the product-preserving logarithmic-transformation statement of this example.
\end{example}

\vspace{1em}

\subsection{Generalized logarithmic transformations and \texorpdfstring{$\star$}{star}-diagrams}

Throughout this section, we view $G$ as a $G$-manifold equipped with the conjugation action. Namely, the action of an element $g \in G$ on an element $h \in G$ is given by $g \cdot h = ghg^{-1}$.

\begin{definition}
    Given two $\star$-collections $\{\phi_{ij} : U_i \cap U_j \to G\}$ and $\{\phi'_{ij} : U_i \cap U_j \to G\}$, we say they are \emph{equivariantly homotopic} if they can be connected by a smooth 1-parameter family of equivariant $\star$-collections $\{\Phi_{ij, t}\}_{t \in [0,1]}$, such that the cocycle and covariance conditions hold for all $t$.
\end{definition}

\begin{lemma}[Speran\c{c}a,~\cite{speranca2012Phd}]\label{lem:speranca+bredon}
If $\{\phi_{ij}\}$ and $\{\phi'_{ij}\}$ are equivariantly homotopic $\star$-collections, then their associated $\star$-bundles are $G \times G$-equivariantly diffeomorphic.
\end{lemma}

\begin{proposition}[Speran\c{c}a,~\cite{speranca2012Phd}]\label{prop:DR}
Let $X$ be a smooth $G$-manifold where $G < \mathrm{Diff}(X)$ is a compact subgroup. Denote by $[X, G]^G$ the group of smooth homotopy classes of $G$-equivariant maps from $X$ to $G$.
\begin{enumerate}
    \item There is a well-defined group homomorphism 
\[DR: [X,G]^G \rightarrow \pi_0(\mathrm{Diff}(X)^G)\]
where $\mathrm{Diff}(X)^G$ is the group of $G$-equivariant self-diffeomorphisms of $X$. This homomorphism is defined by $[\alpha] \mapsto \big[\,\widehat{\alpha'}^{\,-1}\big]$, where $\alpha'$ is a smooth $G$-equivariant approximation of $\alpha$, and $\widehat{\alpha'}$ is the diffeomorphism defined by $\widehat{\alpha'}(x) := \alpha'(x)\,x$.
\item If $M'\leftarrow P\rightarrow M$ is a $\star$-diagram and $N$ is a $G$-manifold, then for any class $[f]\in [N,M]^G$, there is a well-defined map
\[[N,M]^G\rightarrow \mathcal{M}^G\]
which maps $[f]$ to the equivariant diffeomorphism class $[f^*P/\star]$, where $\mathcal{M}^G$ denotes the set of classes of $G$-equivariant manifolds.
\end{enumerate}
\end{proposition}

\begin{proposition}\label{prop:admissible-star}
Let $\Sigma_r$ be the twisted double glued by $\widehat{r}$, built from covariant data $(a,b)$ as above, and set
\[L_\star(\mathrm{S}^k\times\mathrm{S}^l) := \mathrm{Im}(DR) \subseteq \pi_0\big(\mathrm{Diff}(\mathrm{S}^k\times\mathrm{S}^l)^G\big),\]
where $\mathrm{S}^k\times\mathrm{S}^l$ carries the diagonal $G$-action determined by $\Delta_1$ and $\Delta_2$. Assume that
\begin{enumerate}
\item[\textup{(i)}] $k\neq l$, or
\item[\textup{(ii)}] the $\Delta_1$-action of $G$ on $\mathrm{S}^k$ has a fixed point.
\end{enumerate}
Then every class in $L_\star(\mathrm{S}^k\times\mathrm{S}^l)$ consists of admissible boundary diffeomorphisms in the sense of Definition \ref{def:generalized-log}: for any representative $\psi=DR([\alpha])$, the twisted double $\Sigma_\psi = \left(\mathrm{D}^{k+1}\times\mathrm{S}^l\right) \cup_{\widehat{r}\circ\psi} \left(\mathrm{S}^k\times\mathrm{D}^{l+1}\right)$ is a homotopy $(k+l+1)$-sphere, whose diffeomorphism class depends only on the equivariant homotopy class $[\alpha]$.

Moreover, let $P \to \Sigma_r$ be the principal $G$-bundle with transition function $\beta:=(\alpha')^{-1}$ on the overlap of the two charts of $\Sigma_r$, where $\alpha'$ is a smooth equivariant approximation of $\alpha$; then $\Sigma_\psi=P/\star$, and the product-preserving generalized logarithmic transformation on $\Sigma_r \times \mathrm{S}^1$ with boundary map $\phi = \psi \times \mathrm{id}_{\mathrm{S}^1}$ produces exactly the $\star$-quotient of the pullback bundle $\mathrm{pr}_1^* P \cong P \times \mathrm{S}^1$, fitting into the $\star$-diagram
\begin{equation}
\begin{xy}\xymatrix{& G\ar@{..}[d]^{\bullet} & \\ G\ar@{..}[r]^{\star} & \mathrm{pr}_1^* P \ar[d]^{\Pi}\ar[r]^{\Pi'} & \Sigma_\psi \times \mathrm{S}^1\\ & \Sigma_r \times \mathrm{S}^1 &}\end{xy}
\end{equation}
where $\Pi = \pi \times \mathrm{id}_{\mathrm{S}^1}$, $\Pi' = \pi' \times \mathrm{id}_{\mathrm{S}^1}$, and $\pi: P \to \Sigma_r$, $\pi': P \to \Sigma_\psi$ are the two quotient projections.
\end{proposition}
\begin{proof}
Write $\psi=\widehat{\alpha'}^{\,-1}$ and $\beta:=(\alpha')^{-1}$, so that $\beta$ is equivariant and, by the idempotence identity \eqref{eq:hat-idempotence} applied to $\alpha'$, $\widehat\beta=\widehat{\alpha'}^{\,-1}=\psi$. Cover $\Sigma_r$ by the two $G$-invariant open sets $U_1,U_2$ obtained by slightly thickening the pieces $\mathrm{D}^{k+1}\times\mathrm{S}^l$ and $\mathrm{S}^k\times\mathrm{D}^{l+1}$, so that $U_1\cap U_2$ is an invariant collar of the common boundary $\mathrm{S}^k\times\mathrm{S}^l$ and deformation retracts $G$-equivariantly onto it; extend $\beta$ to $U_1\cap U_2$ by composing with this retraction, without change of notation. Since $\beta$ is equivariant for the diagonal action on $\mathrm{S}^k\times\mathrm{S}^l$ and the conjugation action on $G$, it is a $\star$-collection for the cover $\{U_1,U_2\}$. Let $P\to\Sigma_r$ be the associated principal $G$-bundle, glued over the common boundary by $(z,q)\mapsto(\widehat r(z),\,q\,\beta(z))$, with $\bullet$-action $g\bullet(z,q)=(z,gq)$ and $\star$-action $s\star(z,q)=(s\cdot z,\,q\,s^{-1})$: the $\star$-action is free, since $q\,s^{-1}=q$ forces $s=e$, and it commutes with $\bullet$ because left and right translations of $G$ commute (the same mechanism reappears in Proposition \ref{prop:pullback-star}(1) below).

The $\star$-quotient is computed by the slice mechanism (employed again in Proposition \ref{prop:pullback-star}(2) below): each chart quotient is identified with $U_i$ via $z\mapsto[(z,1)]$, the gluing carries $(z,1)$ to $\big(\widehat r(z),\beta(z)\big)$, and the unique slice-returning element is $s=\beta(z)$, so the quotient gluing map is $z\mapsto\beta(z)\cdot\widehat r(z)$. This map is exactly $\widehat r\circ\psi$: by the equivariance of $r$,
\[
\widehat r\big(\widehat\beta(z)\big)=r\big(\beta(z)\cdot z\big)\cdot\beta(z)\cdot z=\beta(z)\,r(z)\,\beta(z)^{-1}\,\beta(z)\cdot z=\beta(z)\cdot\widehat r(z).
\]
Hence $P/\star=\Sigma_\psi$, and Lemma \ref{lem:speranca+bredon} shows that its equivariant diffeomorphism class depends only on $[\alpha]$.

We now verify admissibility. Set $c(z):=\beta(z)\,r(z)$, an equivariant map, so that $\widehat r\circ\psi=\widehat c$ by the displayed identity. By Remark \ref{rem:admissibility}, it suffices to show $P_{00}(\widehat c)=\pm1$.

Case \textup{(i)}: $k\neq l$. By the Künneth formula, $\mathrm{H}_k(\mathrm{S}^k\times\mathrm{S}^l;\mathbb{Z})\cong\mathbb{Z}$, generated by $\mu:=[\mathrm{S}^k\times\mathrm{pt}]$, and $\mathrm{pr}_{1*}:\mathrm{H}_k(\mathrm{S}^k\times\mathrm{S}^l)\to \mathrm{H}_k(\mathrm{S}^k)$ is an isomorphism carrying $\mu$ to the fundamental class. Let $\phi$ be any diffeomorphism of $\mathrm{S}^k\times\mathrm{S}^l$. Then $\phi_*$ is an automorphism of the infinite cyclic group $\mathbb{Z}\mu$, so $\phi_*\mu=\varepsilon\,\mu$ with $\varepsilon=\pm1$. On the other hand, $\mathrm{pr}_{1*}\phi_*\mu=\mathrm{pr}_{1*}\big(\phi_*[\mathrm{S}^k\times\mathrm{pt}]\big)$ is, by definition, $P_{00}(\phi)$ times the fundamental class of $\mathrm{S}^k$. Comparing, $P_{00}(\phi)=\varepsilon=\pm1$. In particular, in this case \emph{every} diffeomorphism of $\mathrm{S}^k\times\mathrm{S}^l$ is admissible, and the hypothesis $[\psi]\in L_\star$ is not needed.

Case \textup{(ii)}: let $x_0\in\mathrm{S}^k$ be a point fixed by $\Delta_1(G)$. Fix $y_0\in\mathrm{S}^l$ and put $R:=\Delta_1\circ c(\cdot,y_0):\mathrm{S}^k\to\mathrm{SO}(k+1)$, so that \[\mathrm{pr}_1\circ\widehat c\,\vert_{\mathrm{S}^k\times\{y_0\}}\;:\;x\longmapsto R(x)\,x .\] For any $p\in\mathrm{S}^k$ write $\mathrm{ev}_p:\mathrm{SO}(k+1)\to\mathrm{S}^k$, $\mathrm{ev}_p(A)=Ap$. Since $k\geq 2$, the standard co-$H$ decomposition of $x\mapsto R(x)x$ gives, in $\pi_k(\mathrm{S}^k)\cong\mathbb{Z}$, \[ \deg\big(x\mapsto R(x)x\big)=1+\deg\big(\mathrm{ev}_{p}\circ R\big), \] and the right-hand side does not depend on $p$: if $A_0\in\mathrm{SO}(k+1)$ satisfies $A_0p=p'$, then $\mathrm{ev}_{p'}=\mathrm{ev}_{p}\circ(\,\cdot\,A_0)$, and right translation by $A_0$ is homotopic to the identity because $\mathrm{SO}(k+1)$ is connected. Taking $p=x_0$, the composite $\mathrm{ev}_{x_0}\circ R=\Delta_1(c(\cdot,y_0))\,x_0$ is the constant map $x\mapsto x_0$, because $x_0$ is fixed by every element of $\Delta_1(G)$. Hence $\deg(\mathrm{ev}_{x_0}\circ R)=0$ and $P_{00}(\widehat c)=1$.  In either case, $P_{00}(\widehat c)=\pm 1$, so $\Sigma_\psi$ is a homotopy $(k+l+1)$-sphere by Remark \ref{rem:admissibility}.

Finally, by Proposition \ref{prop:howtopullback}, the pullback $\mathrm{pr}_1^*P\to\Sigma_r\times\mathrm{S}^1$ is a $\star$-bundle with transition function $\Phi:=\beta\circ\mathrm{pr}_1$ on the overlap $(\mathrm{S}^k\times\mathrm{S}^l)\times\mathrm{S}^1$, the $G$-action on $\Sigma_r\times\mathrm{S}^1$ being trivial on the circle coordinate. The argument above applied verbatim computes the $\star$-quotient's gluing map at $(z,\theta)$ as
\[
\Phi(z,\theta)\cdot\big(\widehat r(z),\theta\big)=\big(\beta(z)\cdot\widehat r(z),\,\theta\big)=\big((\widehat r\circ\psi)(z),\,\theta\big),
\]
i.e., $(\widehat r\circ\psi)\times\mathrm{id}_{\mathrm{S}^1}$. This is precisely the boundary map $\zeta\circ(\psi\times\mathrm{id}_{\mathrm{S}^1})$ of Definition \ref{def:generalized-log}, so $\mathrm{pr}_1^*P/\star\cong\Sigma_\psi\times\mathrm{S}^1$.
\end{proof}

\begin{remark}\label{rem:L-star-caution}
Hypothesis \textup{(ii)} cannot be dropped when $k=l$. Take $k=l=3$, $G=\mathrm{S}^3$ acting on both factors by left translation, and $a\equiv b\equiv e$, so that $\widehat r=\mathrm{id}$ and $\Sigma_r=\mathrm{S}^7$. The map $\alpha'(x,y)=x\bar y$ is smooth and equivariant, and
\[\psi:=DR([\alpha'])=\widehat{\alpha'}^{\,-1}\colon (x,y)\mapsto(y,\;y\bar x y)\]
belongs to $L_\star(\mathrm{S}^3\times\mathrm{S}^3)$. Yet $\mathrm{pr}_1\circ\psi$ is constant on each slice $\mathrm{S}^3\times\{y_0\}$, so $P_{00}(\psi)=0$ and, by Remark \ref{rem:admissibility}, the twisted double glued by $\psi$ has $\mathrm{H}_3\cong\mathbb{Z}$: this $\psi$ is not admissible. The conjugation actions used throughout this paper fix $\pm1\in\mathrm{S}^3$, so hypothesis \textup{(ii)} holds in all our applications.
\end{remark}

\vspace{1em}

\subsection{Spherical T-dualities, \texorpdfstring{$\star$}{star}-diagrams, and logarithmic transformations}
\label{sec:tviastar}

Fix once and for all the following chart presentation of the 4-sphere: $\mathrm{S}^4=\mathrm{D}^4_+\cup_{\mathrm{S}^3}\mathrm{D}^4_-$, where $\mathrm{D}^4_\pm$ are two copies of the closed unit disk in $\mathbb{H}$ glued along the equatorial sphere $\mathrm{S}^3=\partial\mathrm{D}^4_\pm$, and let $G\cong\mathrm{S}^3$ act on each chart by conjugation, $s\cdot x=sx\bar s$. This action is compatible with the gluing and hence defines a smooth $G$-action on $\mathrm{S}^4$. Revisiting the discussion in Section \ref{sec:milnorbundles}, in this presentation, the Milnor bundle $\pi_{m,n}:M_{m,n}\to\mathrm{S}^4$ is the twisted double (we extend the conventions of \eqref{eq:regluing} verbatim to gluings of two copies of $\mathrm{D}^4\times\mathrm{S}^3$, the identification always carrying boundary points of the first-listed piece to the second)
\begin{equation}\label{eq:milnor-twisted-double}
M_{m,n}=\left(\mathrm{D}^4_+\times\mathrm{S}^3\right)\cup_{f_{m,n}}\left(\mathrm{D}^4_-\times\mathrm{S}^3\right),\qquad f_{m,n}(x,v)=\big(x,\,t_{m,n}(x)v\big)=(x,\,x^m v x^n),
\end{equation}
where $t_{m,n}:\mathrm{S}^3\to\mathrm{SO}(4)$, $t_{m,n}(x)v=x^mvx^n$, are Milnor's representatives of $\pi_3(\mathrm{SO}(4))\cong\mathbb{Z}\oplus\mathbb{Z}$ \cite{mi}, $f_{m,n}$ is the clutching diffeomorphism of $\mathrm{S}^3\times\mathrm{S}^3$ determined by $t_{m,n}$, and the projection $\pi_{m,n}$ is given chartwise by the first-coordinate projection. In the terminology of \cite{SperancaCavenaghiPublished}, $t_{m,n}$ is the unique transition function of $\pi_{m,n}$ associated to the cover $\{\mathrm{D}^4_\pm\}$.

We endow $M_{m,n}$ with the $G$-action defined chartwise by
\begin{equation}\label{eq:actionPr}
s \cdot (x,v) = (sx\bar s,\; sv\bar s).
\end{equation}
The chartwise formula defines a global action because $f_{m,n}$ is equivariant for it. Equivalently, $t_{m,n}$ satisfies the covariance condition
\begin{align}
t_{m,n}(sx\bar s)(sv\bar s)=(sx\bar s)^m (sv\bar s) (sx\bar s)^n &= s x^m \bar s\, s v \bar s\, s x^n \bar s \nonumber\\
&= s\big(t_{m,n}(x)v\big)\bar s.
\end{align}
In particular, $\pi_{m,n}$ is equivariant with respect to \eqref{eq:actionPr} and the conjugation action on $\mathrm{S}^4$. Note that \eqref{eq:actionPr} is never free: the central element $-1\in\mathrm{S}^3$ acts trivially.

By Lemma \ref{lem:principal_s3_condition}, $M_{m,n}$ is principal precisely when $m=0$ or $n=0$: the bundles $M_{0,n}$ are principal for the left translation action, their clutching being right multiplication, while the bundles $M_{m,0}$ are principal for the right translation action. The two families are related by quaternionic conjugation of the fiber, $M_{0,n}\cong M_{-n,0}$, an isomorphism of unoriented bundles which reverses the fiber orientation and hence the sign of the Euler class; we keep them notationally distinct, and use $M_{m,0}$ in Theorem \ref{thm:equivalence} and $M_{0,k}$ as the auxiliary family here. Accordingly, fix $k\in\mathbb{Z}$ and consider the principal Milnor bundle $P_k := M_{0,k}$, with clutching diffeomorphism $f_{0,k}(x,q)=(x,\,qx^{k})$. Its principal $\bullet$-action is given chartwise by left translation,
\[
g\bullet(x,q)=(x,\,gq),
\]
globally well defined since left and right translations of $\mathrm{S}^3$ commute. The transition function
\[
\phi(x):=x^{k},\qquad x\in\mathrm{S}^3,
\]
satisfies the covariance condition $\phi(sx\bar s)=(sx\bar s)^{k}=s\,\phi(x)\,s^{-1}$ and therefore constitutes a $\star$-collection for $P_k$ with respect to the conjugation action on $\mathrm{S}^4$. Its associated $\star$-action is given chartwise by
\begin{equation}\label{eq:actionPk}
s \star (x,q) = (sx\bar s,\; q\bar s),
\end{equation}
which is free ($q\bar s=q$ forces $s=1$), commutes with $\bullet$, and is global by the covariance of $\phi$: indeed, $(q\bar s)(sx\bar s)^k = q \bar s\, s x^k \bar s = (q x^k)\bar s$, so $f_{0,k}$ is equivariant for \eqref{eq:actionPk}. Likewise, a principal base $P_r=M_{0,r}$ carries two such structures: the non-free conjugation action \eqref{eq:actionPr}, $s\cdot(x,q)=(sx\bar s,\,sq\bar s)$, and the free action
\begin{equation}\label{eq:free-base}
s \cdot (x,q) = (sx\bar s,\; q\bar s),
\end{equation}
which is \eqref{eq:actionPk} with $k$ replaced by $r$. Whenever both are in play, the equation label indicates which action on $P_r$ is in force.

Recall the hat operator $\widehat{\Phi}(z)=\Phi(z)\cdot z$ of \eqref{eq:hat-operator}; below it is taken with respect to the conjugation action \eqref{eq:actionPr} in items (1)--(2), and with respect to the free action \eqref{eq:free-base} in item (3). The following proposition carries out the pullback-and-quotient construction of \cite{SperancaCavenaghiPublished} in both settings; in each case, the clutching diffeomorphism of the $\star$-quotient is computed through $\widehat{\Phi}$.

\begin{proposition}[Pullback $\star$-diagrams between Milnor bundles]\label{prop:pullback-star}
Let $k\in\mathbb{Z}$ and let $P_k=M_{0,k}$ be as above.
\begin{enumerate}
\item For every $m,n\in\mathbb{Z}$, let $Q:=\pi_{m,n}^{*}P_k$ be the pullback of $P_k$ along $\pi_{m,n}$, an $\mathrm{S}^3$-principal bundle over $M_{m,n}$ whose $\bullet$-action is induced by that of $P_k$. The pulled-back collection $\Phi:=\phi\circ\pi_{m,n}$ is a $\star$-collection relative to the conjugation action \eqref{eq:actionPr} and endows $Q$ with a free $\star$-action commuting with $\bullet$, given chartwise over $M_{m,n}\vert_{\mathrm{D}^4_\pm}$ by
\[
s\star(z,\,q)=\big(s\cdot z,\; q\bar s\big),\qquad z\in M_{m,n}\vert_{\mathrm{D}^4_\pm},\ q\in\mathrm{S}^3.
\]
In particular, $Q$ is an $\mathrm{S}^3$-$\mathrm{S}^3$-bundle.
\item The $\star$-quotient of $Q$ is the twisted double obtained from \eqref{eq:milnor-twisted-double} by composing the clutching diffeomorphism with $\widehat{\Phi}$. Explicitly,
\[
\widehat{\Phi}\circ f_{m,n}=f_{m+k,\,n-k},\qquad\text{whence}\qquad Q/\star\;\cong\;M_{m+k,\,n-k},
\]
and there is a $\star$-diagram
\begin{equation}\label{eq:first-pulling-back} 
M_{m+k,n-k}\;\longleftarrow\;\pi_{m,n}^*P_k\;\longrightarrow\; M_{m,n}.
\end{equation}
In particular, since $\mathrm{e}(M_{m,n})=(m+n)u$ (Section \ref{sec:milnorbundles}), the construction preserves the Euler class: $(m+k)+(n-k)=m+n$.
\item Suppose the base is principal, $P_r=M_{0,r}$, endowed with the free action \eqref{eq:free-base}. The same collection $\Phi=\phi\circ\pi_{r}$, where $\pi_r:=\pi_{0,r}$, is a $\star$-collection relative to \eqref{eq:free-base} and endows $Q_r:=\pi_{r}^{*}P_k$ with a second free $\star$-action, given chartwise by
\[
s\star(x,q,q_k)=\big(sx\bar s,\;q\bar s,\;q_k\bar s\big).
\]
Its quotient clutching is $\widehat{\Phi}\circ f_{0,r}=f_{0,\,r-k}$, whence $Q_r/\star\cong P_{r-k}$, and there is a $\star$-diagram
\begin{equation}\label{eq:principal-pulling-back}
P_{r-k}\;\longleftarrow\;\pi_{r}^*P_k\;\longrightarrow\; P_r.
\end{equation}
Here the Euler class shifts: $\mathrm{e}(P_r)=ru$, while $\mathrm{e}(P_{r-k})=(r-k)u$.
\item For the diagram \eqref{eq:principal-pulling-back}, the combined $G\times G$-action $(h,s)\cdot z:=h\bullet(s\star z)$ is free. We call $\star$-diagrams with this property \emph{bifree}. By contrast, for all $m,n\in\mathbb{Z}$ the diagram \eqref{eq:first-pulling-back} is not bifree: the element $(-1,-1)\in G\times G$ acts trivially on all of $Q$.
\end{enumerate}
\end{proposition}
\begin{proof}
(1) Each chart $Q\vert_\pm:=M_{m,n}\vert_{\mathrm{D}^4_\pm}\times\mathrm{S}^3\subset Q$ is invariant under the chartwise formula, and the only transition function of $Q$ associated to the invariant cover $\{Q\vert_\pm\}$ is $\Phi=\phi\circ\pi_{m,n}$ restricted to the equatorial locus: the gluing of $Q$ reads
\[
(z,\,q)\;\longmapsto\;\big(f_{m,n}(z),\,q\,\Phi(z)\big),\qquad \Phi(z)=x^{k}\ \text{ for } z=(x,v).
\]
Since $\pi_{m,n}$ is equivariant and $\phi$ is covariant, $\Phi$ inherits the covariance condition $\Phi(s\cdot z)=s\,\Phi(z)\,s^{-1}$, so $\Phi$ is a $\star$-collection for $Q$ relative to \eqref{eq:actionPr} (Proposition \ref{prop:howtopullback}). The chartwise formula is compatible with the gluing:
\[
\big(f_{m,n}(s\cdot z),\;q\bar s\,\Phi(s\cdot z)\big)=\big(s\cdot f_{m,n}(z),\;q\bar s\,s\,\Phi(z)\,\bar s\big)=\big(s\cdot f_{m,n}(z),\;q\,\Phi(z)\,\bar s\big),
\]
which is the image of $\big(f_{m,n}(z),\,q\Phi(z)\big)$ under $s\,\star$. Freeness is inherited from the $P_k$-coordinate ($q\bar s=q$ forces $s=1$), and $\bullet$ and $\star$ commute because $\bullet$ acts by left translation on the $P_k$-coordinate only, while left and right translations of $\mathrm{S}^3$ commute.

(2) Over each chart, every $\star$-orbit meets the slice $S_\pm:=\{q=1\}\subset Q\vert_\pm$ exactly once, since $q\bar s=1$ holds precisely for $s=q$. Hence $z\mapsto[(z,1)]$ defines diffeomorphisms $M_{m,n}\vert_{\mathrm{D}^4_\pm}\to \big(Q\vert_{\pm}\big)/\star$, with smooth inverse induced by the $\star$-invariant map $(z,q)\mapsto q\cdot z$, and the quotient manifold is the twisted double $\big(\mathrm{D}^4_+\times\mathrm{S}^3\big)\cup_{\psi}\big(\mathrm{D}^4_-\times\mathrm{S}^3\big)$, where $\psi$ is computed by following the gluing of $Q$ and returning to the slice. For $z=(x,v)$ with $x\in\mathrm{S}^3$, the gluing carries $(z,1)_+$ to $\big(f_{m,n}(z),\,x^{k}\big)_-$, and the unique slice-returning element is $s=x^{k}$:
\[
x^k \star \big(f_{m,n}(z),\,x^k\big) = \big(x^{k}\cdot f_{m,n}(z),\;x^{k}\,\overline{x^{k}}\big)
= \big(\widehat{\Phi}(f_{m,n}(z)),\,1\big),
\]
so $\psi=\widehat{\Phi}\circ f_{m,n}$. A straightforward computation identifies $\psi$:
\[
\widehat{\Phi}(x,w)=x^{k}\cdot(x,w)=\big(x^{k}x\,x^{-k},\;x^{k}w\,x^{-k}\big)=(x,\;x^{k}w\,x^{-k}),
\]
whence $\psi(x,v)=\widehat{\Phi}(x,\,x^m v x^n)=(x,\;x^{m+k}v\,x^{n-k})=f_{m+k,n-k}(x,v)$ and $Q/\star\cong M_{m+k,n-k}$. Since $\bullet$ and $\star$ are free and commute, with $Q/\bullet=M_{m,n}$ and $Q/\star\cong M_{m+k,n-k}$, this is precisely the datum of the $\star$-diagram \eqref{eq:first-pulling-back}.

(3) The collection $\Phi$ factors through the base coordinate, and both \eqref{eq:actionPr} and \eqref{eq:free-base} transform that coordinate by conjugation. Hence, $\Phi\big(s\cdot(x,q)\big)=(sx\bar s)^k=s\,\Phi(x,q)\,s^{-1}$, and $\Phi$ is a $\star$-collection relative to \eqref{eq:free-base} as well. Compatibility with the gluing $(x,q,q_k)\mapsto(x,\,qx^r,\,q_k x^k)$ of $Q_r$ is checked directly:
\[
\big(sx\bar s,\;(q\bar s)(sx\bar s)^r,\;(q_k\bar s)(sx\bar s)^k\big)=\big(sx\bar s,\;qx^r\bar s,\;q_k x^k\bar s\big),
\]
which is the image of $(x,\,qx^r,\,q_k x^k)$ under $s\,\star$. Freeness is again inherited from the $P_k$-coordinate. On the slice $\{q_k=1\}$, the gluing carries $(x,q,1)$ to $(x,\,qx^r,\,x^k)$, and the slice-returning element $s=x^k$ gives
\[
x^k\star\big(x,\,qx^r,\,x^k\big)=\big(x,\;q\,x^{r}x^{-k},\;1\big)=\big(x,\;q\,x^{\,r-k},\;1\big),
\]
so $\psi(x,q)=(x,\,q\,x^{\,r-k})=f_{0,r-k}(x,q)$; equivalently, $\widehat{\Phi}(x,q)=x^k\cdot(x,q)=(x,\,q\,x^{-k})$ and $\psi=\widehat{\Phi}\circ f_{0,r}$. Hence $Q_r/\star\cong P_{r-k}$, and \eqref{eq:principal-pulling-back} follows as in (2).

(4) For \eqref{eq:principal-pulling-back}, the combined action reads $(h,s)\cdot(x,q,q_k)=(sx\bar s,\,q\bar s,\,h\,q_k\bar s)$. A fixed point forces $q\bar s=q$, hence $s=1$, and then $h\,q_k=q_k$, hence $h=1$. For \eqref{eq:first-pulling-back}, the combined action reads $(h,s)\cdot(z,q_k)=(s\cdot z,\,h\,q_k\bar s)$. Taking $(h,s)=(-1,-1)$, conjugation by $-1$ is the identity of $\mathbb{H}$, so $s\cdot z=z$, and $h\,q_k\bar s=(-1)q_k(-1)=q_k$. Thus $(-1,-1)$ acts trivially.
\end{proof}

The two constructions of Proposition \ref{prop:pullback-star} differ in exactly one structural respect, isolated in item (4). The next lemma shows that bifreeness characterizes fiber products of principal bundles among $\star$-diagrams.

\begin{lemma}[Bifree $\star$-diagrams are fiber products]\label{lem:bifree}
Let $M'\xleftarrow{\ \pi'\ } Q\xrightarrow{\ \pi\ } M$ be a $\star$-diagram with $G$ compact and $Q$ closed, and suppose that the diagram is bifree. Then $B:=Q/(G\times G)$ is a smooth closed manifold, the residual actions make $\rho:M\to B$ and $\rho':M'\to B$ principal $G$-bundles, and the map
\[
\Theta:\;z\;\longmapsto\;\big(\pi'(z),\,\pi(z)\big)
\]
is a $G\times G$-equivariant diffeomorphism $Q\to M'\times_B M$ onto the canonical fiber product, intertwining $\bullet$ with the principal action on the $M'$-factor and $\star$ with that on the $M$-factor.
\end{lemma}
\begin{proof}
Since $\bullet$ and $\star$ commute, $(h,s)\cdot z:=h\bullet(s\star z)$ defines a smooth $G\times G$-action. As $G\times G$ is compact and acts freely on the closed manifold $Q$, the quotient $B$ is a smooth closed manifold and $Q\to B$ is a principal $G\times G$-bundle. Because $\bullet$ and $\star$ commute, the $\star$-action descends to $M=Q/\bullet$ via $s\cdot[z]:=[s\star z]$, and the descended action is free: if $[s\star z]=[z]$, then $s\star z=h^{-1}\bullet z$ for some $h\in G$, i.e. $(h,s)\cdot z=z$, forcing $(h,s)=(e,e)$ by bifreeness. Hence $\rho:M\to M/\star=Q/(G\times G)=B$ is a principal $G$-bundle, and symmetrically the descended $\bullet$-action exhibits $\rho':M'\to B$ as a principal $G$-bundle. In particular, $\rho$ and $\rho'$ are submersions, so $M'\times_B M$ is a smooth closed manifold and $M'\times_B M\to B$ is a principal $G\times G$-bundle.

The map $\Theta$ lands in the fiber product, since both coordinates project to the class of $z$ in $B$; it is smooth, being assembled from the quotient submersions. It is $G\times G$-equivariant: $\pi'$ kills $\star$ and transports $\bullet$ to the descended action on $M'$, while $\pi$ kills $\bullet$ and transports $\star$ to the descended action on $M$, so
\[
\Theta\big((h,s)\cdot z\big)=\big(\pi'(h\bullet(s\star z)),\,\pi(h\bullet(s\star z))\big)=\big(h\cdot\pi'(z),\,s\cdot\pi(z)\big)=(h,s)\cdot\Theta(z).
\]
Thus $\Theta$ is a morphism of principal $G\times G$-bundles covering the identity of $B$, hence a diffeomorphism.
\end{proof}

We now bring spherical T-duality into the picture; the following normalization will be in force for the remainder of the section.

Following \cite{Bouwknegt2015}, for an oriented linear $\mathrm{S}^3$-bundle $\pi:E\to\mathrm{S}^4$ we denote by $\pi_{\ast}:\mathrm{H}^{7}(E;\mathbb{Z})\to\mathrm{H}^{4}(\mathrm{S}^4;\mathbb{Z})$ the Gysin homomorphism, and by $u\in\mathrm{H}^4(\mathrm{S}^4;\mathbb{Z})$ the orientation generator. Since $\mathrm{H}^7(\mathrm{S}^4)=\mathrm{H}^8(\mathrm{S}^4)=0$, the Gysin sequence shows that $\pi_\ast$ is an isomorphism; we write $\eta_E\in\mathrm{H}^7(E;\mathbb{Z})$ for the generator normalized by $\pi_\ast\eta_E=u$, and $[a]:=a\,\eta_E$ for $a\in\mathbb{Z}$.

\begin{theorem}\label{thm:equivalence}
Let $P,\widehat{P}\to\mathrm{S}^4$ be $\mathrm{S}^3$-principal bundles (with projections $\pi, \widehat\pi$, respectively) and let $\mathcal{Q}=P\times_{\mathrm{S}^4}\widehat{P}$ be their fiber product. Then, we have a bifree $\star$-diagram
\begin{equation}\label{eq:stardualidade}
\begin{xy}
\xymatrix{
& \mathrm{S}^3 \ar@{..}[d]^{\bullet} & \\
\mathrm{S}^3 \ar@{..}[r]^{\star} & \mathcal{Q} \ar[d]^{\widehat{p}} \ar[r]^{p} & \widehat{P} \\
& P &
}
\end{xy}
\end{equation}
where $\widehat{p}:\mathcal{Q}\to P$ and $p:\mathcal{Q}\to\widehat{P}$ are the two quotient projections. 
Given classes $H\in\mathrm{H}^7(P;\mathbb{Z})$ and $\widehat{H}\in\mathrm{H}^7(\widehat{P};\mathbb{Z})$, the pairs $(P,H),(\widehat{P},\widehat{H})$ are spherical T-dual \emph{if and only if} 
\begin{equation}\label{eq:fluxdecoration}
\mathrm{c}_2(\widehat{P}) = \pi_{\ast}H, \qquad \mathrm{c}_2(P) = \widehat{\pi}_{\ast}\widehat{H}, \qquad \widehat{p}^{\,*}H = p^{*}\widehat{H}\ \ \text{in } \mathrm{H}^7(\mathcal{Q};\mathbb{Z}).
\end{equation}
In particular, writing $P=M_{m,0}$ and $\widehat{P}=M_{j,0}$ (Lemma \ref{lem:principal_s3_condition}), setting $H=[j]$, $\widehat{H}=[m]$ realizes \eqref{eq:fluxdecoration}, with $\mathcal{Q}\cong\pi^{*}\widehat{P}\cong\widehat{\pi}^{\,*}P$ canonically as $\mathrm{S}^3$-$\mathrm{S}^3$-bundles.
\end{theorem}
\begin{proof}
The existence of the $\star$-diagram is straightforward from the definition of the fiber product:
\[\mathcal Q=\left\{(x,\widehat x)\in P\times \widehat P:\pi(x)=\widehat\pi(\widehat x)\right\},\]
a smooth closed manifold since $\pi$ and $\widehat\pi$ are submersions. We choose the $\bullet$-action as 
\[g\bullet (x,\widehat x):=(x,g\widehat x),\]
where $g\widehat x$ denotes the principal $G$-action on $\widehat P$, and likewise, we let the $\star$-action be
\[g\star (x,\widehat x):=(gx,\widehat x).\]
Both are well defined because $\pi$ and $\widehat\pi$ are principal $G$-bundles, so $\pi(gx)=\pi(x)$ and $\widehat\pi(g\widehat x)=\widehat\pi(\widehat x)$; they are free, being free on the respective coordinates, and they commute, acting on distinct coordinates. The $\bullet$-orbits are exactly the fibers of $\widehat p(x,\widehat x)=x$, and the $\star$-orbits exactly the fibers of $p(x,\widehat x)=\widehat x$, so $\mathcal Q/\bullet=P$ and $\mathcal Q/\star=\widehat P$, consistently with the decorations of \eqref{eq:stardualidade}. The combined action $(h,s)\cdot(x,\widehat x)=(sx,\,h\widehat x)$ is free, so the diagram is bifree. This verifies the first part of the claim.

Observe that $\mathcal{Q}=P\times_{\mathrm{S}^4}\widehat{P}$, together with its projections $\widehat{p}$ and $p$, is precisely the correspondence space through which spherical T-duality is formulated: $(P,H)$ and $(\widehat{P},\widehat{H})$ are spherical T-dual if and only if the equations \eqref{eq:fluxdecoration} hold. This verifies the second assertion, whose content is the geometric identification of the T-duality correspondence with the biprincipal $\star$-diagram \eqref{eq:stardualidade}: the two quotient maps of the $\mathrm{S}^3\times\mathrm{S}^3$-structure are the two T-duality projections.

Lastly, write $P=M_{m,0}$ and $\widehat{P}=M_{j,0}$ (Lemma \ref{lem:principal_s3_condition}) and assign the fluxes $H=[j]$ on $P$ and $\widehat{H}=[m]$ on $\widehat{P}$. Theorem \ref{thm:tdualmilnor} verifies the first two equations of \eqref{eq:fluxdecoration} and the pullback equality $\widehat{p}^{\,*}[j]=p^{*}[m]$ for the pair $\big(M_{m,0},[j]\big)$, $\big(M_{0,-j},[m]\big)$; the principal-bundle isomorphism $M_{0,-j}\cong M_{j,0}$, which preserves the normalizations $[\,\cdot\,]$ in the compatible orientations, transports these relations to the present pair and identifies the two correspondence spaces over $\mathrm{S}^4$. This assignment depends only on the Chern classes of $P$ and $\widehat{P}$. For the identification of $\mathcal{Q}$, note that $\pi^{*}\widehat{P}=\{(x,\widehat x)\in P\times\widehat P:\pi(x)=\widehat\pi(\widehat x)\}$ coincides with the fiber product as a smooth manifold; its tautological principal action, on the $\widehat{P}$-coordinate, is the $\bullet$-action above, while the $G$-action inherited from the principal action of $P$ on the first coordinate is the $\star$-action, so $\mathcal{Q}\cong\pi^{*}\widehat{P}$, and symmetrically $\mathcal{Q}\cong\widehat{\pi}^{\,*}P$, at the level of $\mathrm{S}^3$-$\mathrm{S}^3$-bundles.
\end{proof}

\begin{remark}\label{rem:bifree-converse}
By Lemma \ref{lem:bifree}, the passage is reversible: every bifree $\star$-diagram is canonically the fiber product of its residual principal bundles over $B=Q/(G\times G)$. In particular, the bifree diagram \eqref{eq:principal-pulling-back} of Proposition \ref{prop:pullback-star} is a fiber product of principal $\mathrm{S}^3$-bundles.
\end{remark}

The pullback construction of Proposition \ref{prop:pullback-star}, performed with the conjugation action \eqref{eq:actionPr}, never produces a bifree structure, by item (4). It is therefore a genuinely different mechanism from the fiber product of Theorem \ref{thm:equivalence}, and it is this mechanism that produces the dual bundle in what follows.

\begin{theorem}\label{thm:non-principal-equivalence}
Let $E,\widehat{E}\to\mathrm{S}^4$ be oriented linear $\mathrm{S}^3$-bundles with projections $\pi_E, \pi_{\widehat E}$, respectively, and assume they have the same Euler class $\mathrm{e}(E)=\mathrm{e}(\widehat{E})=ku$. By Theorem \ref{thm:so(4)principals}, write $E=M_{m,k-m}$ and $\widehat{E}=M_{j,k-j}$ for some $m,j\in\mathbb{Z}$. Then the $\mathrm{S}^3$-$\mathrm{S}^3$-bundle $\mathcal{Q}:=\pi_E^{*}P_{j-m}$ of Proposition \ref{prop:pullback-star} fits into the $\star$-diagram
\begin{equation}\label{eq:stardualidadenonprinc}
\begin{xy}
\xymatrix{
& \mathrm{S}^3 \ar@{..}[d]^{\bullet} & \\
\mathrm{S}^3 \ar@{..}[r]^{\star} & \mathcal{Q} \ar[d]^{\widehat{p}} \ar[r]^{p} & \widehat{E} \\
& E &
}
\end{xy}
\end{equation}
where $\widehat{p}:\mathcal{Q}\to E$ and $p:\mathcal{Q}\to\widehat{E}$ are the two quotient projections; in contrast with Theorem \ref{thm:equivalence}, this $\star$-diagram is not bifree. Moreover, the degree-seven classes $H = [k]\in\mathrm{H}^7(E;\mathbb{Z})$ and $\widehat{H} = [k]\in\mathrm{H}^7(\widehat{E};\mathbb{Z})$ satisfy
\[
\pi_{E\,\ast}H = \mathrm{e}(\widehat{E}), \qquad \pi_{\widehat{E}\,\ast}\widehat{H} = \mathrm{e}(E),
\]
together with the pullback equality $q^{*}H = \widehat{q}^{\,*}\widehat{H}$ on the correspondence space $E\times_{\mathrm{S}^4}\widehat{E}$ of spherical T-duality, where $q$ and $\widehat{q}$ denote its two canonical projections. Consequently, $(E,H)$ and $(\widehat{E},\widehat{H})$ are spherical T-dual.

Conversely, given a $\star$-diagram generated via Proposition \ref{prop:pullback-star} over a non-principal base $E$, the $\star$-quotient $\widehat{E}$ is an oriented linear $\mathrm{S}^3$-bundle over $\mathrm{S}^4$ with $\mathrm{e}(\widehat{E})=\mathrm{e}(E)=ku$, and the fluxes $H=\widehat{H}=[k]$ form the unique pair of classes making $(E,H)$ and $(\widehat{E},\widehat{H})$ spherical T-dual.
\end{theorem}
\begin{proof}
For the forward implication, apply Proposition \ref{prop:pullback-star} with base $M_{m,k-m}=E$ and auxiliary bundle $P_{j-m}=M_{0,j-m}$: items (1) and (2) produce the $\mathrm{S}^3$-$\mathrm{S}^3$-bundle $\mathcal{Q}=\pi_E^{*}P_{j-m}$ with
\[
\mathcal{Q}/\bullet = E,\qquad \mathcal{Q}/\star \cong M_{m+(j-m),\,(k-m)-(j-m)} = M_{j, k-j} = \widehat{E},
\]
which is the $\star$-diagram \eqref{eq:stardualidadenonprinc}, and item (4) shows it is not bifree. It remains to verify the assertions concerning the fluxes $H=\widehat{H}=[k]$. The two displayed equalities hold by the normalization $\pi_{E\,\ast}\eta_E=u$, $\pi_{\widehat{E}\,\ast}\eta_{\widehat{E}}=u$ fixed above: $\pi_{E\,\ast}[k]=ku=\mathrm{e}(\widehat{E})$ and $\pi_{\widehat{E}\,\ast}[k]=ku=\mathrm{e}(E)$. The pullback equality $q^{*}H = \widehat{q}^{\,*}\widehat{H}$ on the correspondence space $E\times_{\mathrm{S}^4}\widehat{E}$ is exactly the Serre spectral-sequence relation $k(u\otimes x)=k(u\otimes y)$ established in the proof of Theorem \ref{thm:nonprincipaltdualmilnor}. Hence $(E,[k])$ and $(\widehat{E},[k])$ are spherical T-dual.

For the converse, suppose we generate a $\star$-diagram via Proposition \ref{prop:pullback-star} by pulling back a principal bundle $P_c$ over a non-principal base $E = M_{m,n}$. By item (2), the resulting quotient is $\widehat{E} = M_{m+c, n-c}$, with Euler class $(m+c)+(n-c)=m+n=\mathrm{e}(E)$; writing $\mathrm{e}(E)=(m+n)u=:ku$, the construction therefore preserves the Euler class, so that $\mathrm{e}(E)=\mathrm{e}(\widehat{E})=ku$. Assigning to each side the diagonal flux $[k]$, Theorem \ref{thm:nonprincipaltdualmilnor} (whose proof verifies the pullback equality on the correspondence space) guarantees that $(E,[k])$ and $(\widehat{E},[k])$ are spherical T-dual. Uniqueness is immediate: since $\pi_{E\,\ast}$ and $\pi_{\widehat{E}\,\ast}$ are isomorphisms, the equations $\pi_{E\,\ast}H=\mathrm{e}(\widehat{E})=ku$ and $\pi_{\widehat{E}\,\ast}\widehat{H}=\mathrm{e}(E)=ku$ determine $H=\widehat{H}=[k]$ uniquely.
\end{proof}

\begin{remark}\label{rem:two-correspondences}
Two distinctions deserve emphasis. First, the $\star$-diagram total space $\mathcal{Q}=\pi_E^{*}P_{j-m}=E\times_{\mathrm{S}^4}P_{j-m}$ is in general distinct from the T-duality correspondence space $E\times_{\mathrm{S}^4}\widehat{E}$ on which the pullback equality is verified: the $\star$-diagram produces the dual bundle, while the equations characterizing spherical T-duality are stated on the canonical fiber product. Second, the dichotomy with the principal case is structural, not accidental. In Theorem \ref{thm:equivalence} the $\star$-diagram is bifree and, by Lemma \ref{lem:bifree}, canonically a fiber product of principal bundles; correspondingly, in the principal case the relations \eqref{eq:fluxdecoration} interchange the second Chern classes and the fluxes of the two bundles. Here, by Proposition \ref{prop:pullback-star}(4), the element $(-1,-1)\in G\times G$ acts trivially, no bifree structure is produced, and correspondingly the construction is confined to a fixed Euler class: the diagonal fluxes $H=\widehat{H}=[k]$ are available precisely because the pullback construction \eqref{eq:first-pulling-back} preserves the Euler class.
\end{remark}

\vspace{1em}

\begin{remark}\label{rem:def-scope}
Definition \ref{def:generalized-log} extends verbatim to twisted doubles presented as two copies of $\mathrm{D}^4\times\mathrm{S}^3$ glued along the equatorial $\mathrm{S}^3\times\mathrm{S}^3$: admissibility retains its meaning, namely that the reglued twisted double is a homotopy sphere. The geometric identification of Proposition \ref{prop:admissible-star} (between the product-preserving transformation with boundary map $\psi\times\mathrm{id}_{\mathrm{S}^1}$ and the $\star$-quotient of the pulled-back bundle) also carries over unchanged. We caution that the homological criterion of Remark \ref{rem:admissibility} is tied to the mixed pieces $\mathrm{D}^{k+1}\times\mathrm{S}^l$, $\mathrm{S}^k\times\mathrm{D}^{l+1}$ and takes a different matrix form in the bundle presentation. In the application below, admissibility is instead verified directly, by identifying the reglued twisted double with a Milnor bundle of Euler class one.
\end{remark}

\vspace{1em}

\begin{theorem}\label{thm:Sperical-T-Dual-vs-Log}
Let $m,j\in\mathbb{Z}$, and let $M_{m,1-m}$ and $M_{j,1-j}$ be the corresponding homotopy $7$-spheres, so that $\mathrm{e}(M_{m,1-m})=\mathrm{e}(M_{j,1-j})=u$. Equipped with the diagonal fluxes $[1]$, they realize a spherical T-duality (Theorem \ref{thm:nonprincipaltdualmilnor}). These manifolds fit into an explicit $\star$-diagram, generated by the pullback of the principal Milnor bundle $P_{j-m}$ over $M_{m,1-m}$:
\begin{equation}\label{eq:theorem-7d-diagram}
\begin{xy}\xymatrix{
& \mathrm{S}^3\ar@{..}[d]^{\bullet} & \\ 
\mathrm{S}^3\ar@{..}[r]^{\star} & Q\ar[d]^{\widehat{p}}\ar[r]^{p} & M_{j,1-j}\\ 
& M_{m,1-m} &
}\end{xy}
\end{equation}
where the total space is $Q = \pi_{m,1-m}^{*}P_{j-m}$.

Furthermore, upon taking the Cartesian product with $\mathrm{S}^1$, the spaces $M_{m,1-m}\times\mathrm{S}^1$ and $M_{j,1-j}\times\mathrm{S}^1$ are related by a product-preserving generalized logarithmic transformation in the sense of Definition \ref{def:generalized-log}. This transformation is realized as the $\star$-quotient of the pullback bundle $\mathrm{pr}_1^* Q$ along the first-factor projection $\mathrm{pr}_1: M_{m,1-m} \times \mathrm{S}^1 \to M_{m,1-m}$, with boundary map $\psi_{j-m}\times\mathrm{id}_{\mathrm{S}^1}$, where $[\psi_{j-m}] \in L_\star(\mathrm{S}^3\times\mathrm{S}^3)$. Under the canonical identification $\mathrm{pr}_1^* Q\cong Q\times\mathrm{S}^1$, the transformation fits into the $\star$-diagram:
\begin{equation}\label{eq:theorem-8d-diagram}
\begin{xy}\xymatrix{
& \mathrm{S}^3\ar@{..}[d]^{\bullet} & \\ 
\mathrm{S}^3\ar@{..}[r]^{\star} & \mathrm{pr}_1^* Q \ar[d]^{\Pi}\ar[r]^{\Pi'} & M_{j,1-j} \times \mathrm{S}^1\\ 
& M_{m,1-m} \times \mathrm{S}^1 &
}\end{xy}
\end{equation}
where $\Pi = \widehat{p} \times \mathrm{id}_{\mathrm{S}^1}$ and $\Pi' = p \times \mathrm{id}_{\mathrm{S}^1}$.
\end{theorem}

\begin{proof}
Since $\mathrm{e}(M_{m,1-m})=\mathrm{e}(M_{j,1-j})=u$, Theorem \ref{thm:non-principal-equivalence} applies and constructs Diagram \eqref{eq:theorem-7d-diagram} with $Q = \pi_{m,1-m}^{*}P_{j-m}$: here $\widehat{p}$ is the principal quotient and $p$ the $\star$-quotient, and Proposition \ref{prop:pullback-star}(2) identifies the latter, \[ Q/\star\;\cong\;M_{m+(j-m),\;(1-m)-(j-m)}\;=\;M_{j,1-j}, \] which is a homotopy $7$-sphere, its Euler class being $\big(j+(1-j)\big)u=u$. In particular, the cases $m\in\{0,1\}$ and $j\in\{0,1\}$, in which one of the two bundles is principal, are included.

To connect this configuration to the generalized logarithmic transformation, we extract the boundary diffeomorphism induced by the $\star$-quotient. By Proposition \ref{prop:pullback-star}(2), the quotient of $Q$ modifies the clutching diffeomorphism of $M_{m,1-m}$ by postcomposition with $\widehat{\Phi}$, where $\Phi(x,y) = x^{j-m}$ is the pulled-back transition function on the equatorial $\mathrm{S}^3 \times \mathrm{S}^3$. Explicitly, under the conjugation action:
\[
\psi_{j-m}(x,y) := \widehat{\Phi}(x,y) = x^{j-m} \cdot (x,y) = \big(x,\;x^{j-m}\,y\,x^{m-j}\big).
\]

We first record that $\psi_{j-m}$ is a diffeomorphism of $\mathrm{S}^3\times\mathrm{S}^3$, with inverse $(x,y)\mapsto (x,\,x^{m-j}y\,x^{j-m})$, and that it is $\mathrm{S}^3$-equivariant for the diagonal conjugation action. Next, since all the group elements involved are powers of the same quaternion $x$, postcomposition and precomposition agree here: \[\widehat{\Phi}\circ f_{m,1-m}=f_{m,1-m}\circ\psi_{j-m}=f_{j,1-j},\] both sides being $(x,y)\mapsto(x,\,x^{j}\,y\,x^{1-j})$. In particular, regluing the twisted double \eqref{eq:milnor-twisted-double} of $M_{m,1-m}$ by $\psi_{j-m}$ produces $M_{j,1-j}$, a homotopy $7$-sphere; hence $\psi_{j-m}$ is admissible in the extended sense of Remark \ref{rem:def-scope}, and $\phi:=\psi_{j-m}\times\mathrm{id}_{\mathrm{S}^1}$ is a boundary diffeomorphism of the product form required by Definition \ref{def:generalized-log}. It remains to verify that $[\psi_{j-m}] \in L_\star(\mathrm{S}^3\times\mathrm{S}^3) = \mathrm{Im}(DR)$. Consider the smooth map $\alpha:\mathrm{S}^3\times\mathrm{S}^3 \to \mathrm{S}^3$, $\alpha(x,y)=x^{m-j}$, which is equivariant for the diagonal conjugation action on $\mathrm{S}^3\times\mathrm{S}^3$ and the conjugation action on $\mathrm{S}^3$, since $(sx\bar s)^{m-j}=s\,x^{m-j}\,\bar s$; here the hat is taken with respect to the diagonal conjugation action. Its hat evaluates to:
\[
\widehat{\alpha}(x,y) = x^{m-j} \cdot (x,y) = \big(x,\;x^{m-j}yx^{-(m-j)}\big) = \big(x,\;x^{m-j}yx^{j-m}\big).
\]
By the hat-inverse identity established in the proof of Lemma \ref{lem:speranca-sigma-r}, the inverse of $\widehat\alpha$ is the hat of the pointwise inverse $\alpha^{-1}(x,y)=x^{j-m}$, namely \[ \widehat{\alpha}^{-1}(x,y) = \widehat{\alpha^{-1}}(x,y)=\big(x,\;x^{j-m}\,y\,x^{m-j}\big)=\psi_{j-m}(x,y). \] By the definition of $DR$ in Proposition \ref{prop:DR}, $DR([\alpha]) = \big[\widehat{\alpha}^{\,-1}\big]=[\psi_{j-m}]$, confirming that $[\psi_{j-m}] \in L_\star(\mathrm{S}^3\times\mathrm{S}^3)$.

Because $[\psi_{j-m}] \in L_\star$, we may invoke Proposition \ref{prop:admissible-star} in the extended scope of Remark \ref{rem:def-scope}; its hypothesis \textup{(ii)} is satisfied, since the conjugation action of $\mathrm{S}^3$ on $\mathrm{S}^3$ fixes $\pm 1$. That proposition identifies the product-preserving generalized logarithmic transformation on $M_{m,1-m}\times\mathrm{S}^1$ with boundary map $\phi=\psi_{j-m}\times\mathrm{id}_{\mathrm{S}^1}$ with the $\star$-quotient of the pullback bundle $\mathrm{pr}_1^* Q$: the transition function of $\mathrm{pr}_1^*Q$ is $\Phi\circ\mathrm{pr}_1$, constant along the circle coordinate, so the slice mechanism yields the quotient gluing map $(\widehat{\Phi}\circ f_{m,1-m})\times\mathrm{id}_{\mathrm{S}^1}=f_{j,1-j}\times\mathrm{id}_{\mathrm{S}^1}$, whence $\mathrm{pr}_1^*Q/\star\cong M_{j,1-j}\times\mathrm{S}^1$. This is the content of Diagram \eqref{eq:theorem-8d-diagram}.
\end{proof}

\vspace{1em}

\begin{example}[Gromoll--Meyer]\label{ex:gm-as-corollary}
Take $m=1$ and $j=2$. The base is $M_{1,0}=\mathrm{S}^7$, the auxiliary bundle is $P_{1}=M_{0,1}$, and the boundary map is $\psi_{1}(x,y)=(x,\,xy\bar x)$; indeed
\[
\widehat{\Phi}\circ f_{1,0}(x,v)=\widehat{\Phi}(x,\,xv)=\big(x,\;x(xv)\bar x\big)=\big(x,\;x^{2}v\,x^{-1}\big)=f_{2,-1}(x,v).
\]
Theorem \ref{thm:Sperical-T-Dual-vs-Log} therefore produces $M_{2,-1}=\Sigma^7_{GM}$ from $\mathrm{S}^7$, and exhibits $\Sigma^7_{GM}\times\mathrm{S}^1$ as a product-preserving generalized logarithmic transformation of $\mathrm{S}^7\times\mathrm{S}^1$. This is the same conclusion reached in Example \ref{ex:gm-revisited}, obtained there through the mixed presentation $\mathrm{S}^7=(\mathrm{D}^4\times\mathrm{S}^3)\cup_{\mathrm{id}}(\mathrm{S}^3\times\mathrm{D}^4)$ and the boundary map $\beta_b^{-1}\circ\beta_a$; the two boundary maps differ because the two presentations of $\mathrm{S}^7$ differ, but the resulting manifolds agree.
\end{example}

\vspace{1em}

\section*{Acknowledgments}
When this work began, the work of L.F.C. was supported by S\~ao Paulo Research Foundation FAPESP grants 2022/09603-9 and 2023/14316-1. 

L.~F.~C. and L.~K. are supported by the Simons Foundation, grant SFI-MPS-T-Institutes-00007697, and the Ministry of Education and Science of the Republic of Bulgaria, grant DO1-239/10.12.2024. 

L.~K. is supported by the Simons Investigators Award (no. 003136), the Simons Collaboration on Homological Mirror Symmetry (award no. 003093), and by the NSF FRG grant DMS-2245099.

L. G. is partially supported by S\~ao Paulo Research Foundation (FAPESP) grants 2023/13131-8, 2021/04065-6 and by CNPq grant 306021/2024-2.

\bibliographystyle{elsarticle-harv}
\bibliography{main}

\end{document}